\documentclass[a4paper,11pt]{amsart}
\usepackage{amsmath,amsthm,amssymb}
\usepackage{slashed}
\usepackage{comment}
\newcommand{\mychoice}[3]{#1
}
\mychoice{
\newcommand{\plabel}[1]{ \label{#1}}
\newcommand{\gbibitem}[1]{ \bibitem{#1}}
\newcommand{\snewpage}{}

\specialcomment{commentx}{}{}\excludecomment{commentx}
\specialcomment{commenty}{}{}\excludecomment{commenty}
\newcommand{\rechoicecomm}[1]{}
}{
\usepackage{xcolor}
\newcommand{\plabel}[1]{ \label{#1}\rlap{\smash{${}^{^{[#1]}}$}}}
\newcommand{\gbibitem}[1]{ \bibitem{#1}\rlap{\smash{${}^{^{[#1]}}$}}}
\newcommand{\snewpage}{\newpage}

\newenvironment{commentx}{\color{magenta} }{\color{black} }
\newenvironment{commenty}{\color{blue} }{\color{black} }
\newcommand{\rechoicecomm}[1]{#1}
}{
\usepackage{xcolor}
\newcommand{\plabel}[1]{ \label{#1}}
\newcommand{\gbibitem}[1]{ \bibitem{#1}}
\newcommand{\snewpage}{}

\specialcomment{commentx}{}{}\excludecomment{commentx}
\newenvironment{commenty}{\scriptsize }{\normalsize }
\newcommand{\rechoicecomm}[1]{}
}

\DeclareMathOperator{\Id}{Id}

\DeclareMathOperator{\des}{des}
\DeclareMathOperator{\artanh}{artanh}

\DeclareMathOperator*{\esssup}{ess\,sup}
\DeclareMathOperator*{\essinf}{ess\,inf}
\DeclareMathOperator{\asc}{asc}

\DeclareMathOperator{\intD}{\mathring D}
\newcommand{\real}{\mathrm{real}}

\DeclareMathOperator{\Dbar}{D}

\DeclareMathOperator{\Rexp}{exp_{R}}

\DeclareMathOperator{\BCH}{BCH}
\DeclareMathOperator{\spec}{sp}

\theoremstyle{definition}
\newtheorem{point}{}[section]
\newtheorem{defin}[point]{Definition}
\newtheorem{conven}[point]{Convention}

\newtheorem{remark}[point]{Remark}
\newtheorem{example}[point]{Example}
\theoremstyle{plain}

\newtheorem{lemma}[point]{Lemma}
\newtheorem{cor}[point]{Corollary}
\newtheorem{theorem}[point]{Theorem}

\newcommand{\leaveout}[1]{}

\newcommand{\eqed}{
\pushQED{\qed}
\qedhere
\popQED
}
\newcommand{\eqedexer}{
\renewcommand{\qedsymbol}{$\diamondsuit$}
\pushQED{\qed}
\qedhere
\popQED
\renewcommand{\qedsymbol}{$\Box$}
}

\newcommand{\qedexer}{  \renewcommand{\qedsymbol}{$\diamondsuit$} \qed \renewcommand{\qedsymbol}{$\Box$}}
\newcommand{\qedremark}{  \renewcommand{\qedsymbol}{$\triangle$} \qed \renewcommand{\qedsymbol}{$\Box$}}
\newcommand{\qedno}{\renewcommand{\qedsymbol}{}}

\newcommand{\proofremarkqed}[1]{
\begin{proof}[Remark] #1
\renewcommand{\qedsymbol}{$\triangle$}
\end{proof}
\renewcommand{\qedsymbol}{$\Box$}
}

\newcommand{\marginextend}[1]{ \addtolength{\oddsidemargin}{-#1}  \addtolength{\evensidemargin}{-#1}
  \addtolength{\textwidth}{#1}\addtolength{\textwidth}{#1}}
\newcommand{\updownextend}[1]{ \addtolength{\topmargin}{-#1}  \addtolength{\textheight}{#1}
\addtolength{\textheight}{#1}}
\marginextend{1cm}
\updownextend{0cm}
\allowdisplaybreaks[4]
\title[Convergence estimates for the Magnus expansion IA.]{Convergence estimates for the Magnus expansion IA. Uniformly convex algebras}
\author{Gyula Lakos}
\address{Department of Geometry, Institute of Mathematics, E\"otv\"os University,
P\'azm\'any P\'eter s.~1/C,  Budapest, H--1117, Hungary}
\email{lakos@cs.elte.hu}
\keywords{Magnus expansion, Baker--Campbell--Hausdorff expansion, convergence estimates, uniform convexity, resolvent method}
\subjclass[2020]{Primary: 16W60, 47A56, Secondary:46H30.}
\begin{document}
\begin{abstract}
We  review and provide simplified proofs related to the Magnus expansion, and improve convergence estimates.
Observations and improvements concerning the Baker--Campbell--Hausdorff expansion are also made.

In this Part IA, we consider uniform convexity.
Notions of uniformly convex algebras are discussed, and uniform convexity is shown to improve convergence estimates.
\end{abstract}
\maketitle

\section*{Introduction to Part IA}

In this Part IA, which is a direct continuation of Part I \cite{L1},
we aim to demonstrate that how uniform convexity improves the convergence properties of the Magnus expansion of
Magnus \cite{M}.
For notation and terminology, as well as a general overview of the convergence problem of the
Magnus expansion  in the case of Banach algebras, we refer to \cite{L1}.
For the sake of comparison, we make occasional references to Part II \cite{L2} and Part III \cite{L3},
but they are not needed for this present development.

\textbf{Introduction to the setting of uniformly convex algebras.}
As it is known, in the general setting of Banach algebras, the guaranteed convergence radius of the Magnus expansion
in terms of the cumulative norm (i.~e.~the variation) of the Banach algebra valued ordered measure is exactly $2$,
see Moan, Oteo \cite{MO} and \cite{L1}.
Yet, it is also known that in the setting of operators on Hilbert spaces the corresponding
value is $\pi$, see Moan, Niesen \cite{MN} and Casas \cite{Ca}, cf.~also Sch\"affer \cite{Scha}.
One may wonder whether this convergence improvement phenomenon extends to a class of Banach algebras more general
than the operators on Hilbert spaces (i.~e.~beyond $C^*$-algebras).
It is a possibility to attribute the convergence improvement to the ``roundedness'' of the unit balls of Hilbert spaces.
On a technical level, this manifests in the conformal range,
which is a reduced version of the Davis--Wielandt shell of
Wielandt \cite{Wie} and Davis \cite{D1}, \cite{D2}; see this explained in \cite{L2}.
We could try to generalize the notion of conformal range
for operators acting on $L^p$ spaces, which is quite possible up to a certain degree;
but this would lead to a geometric discussion applicable only to a relatively
limited class of Banach algebras.
Here, in Part IA, we take another approach,
which can be applied to exhibit convergence improvement in a relatively large class of Banach algebras.

\textbf{Notions of uniform convexity.}
As a main point, the classical notion of uniform convexity is the sense Clarkson is too restrictive for Banach algebras.
Therefore, we will use higher order notions of uniform convexity, which are weaker.
In fact, most of the discussion will be conducted under the 4th order convexity condition $(\mathcal{UMQ}_q)$.

Our primary objective here is not to obtain the possibly strongest numerical estimates but to demonstrate (the applicability of)
our methods related to uniform convexity.
Nevertheless, as a consequence, we will see that the convergence radius of (the exponential generating function of)
the Magnus commutators in Hilbert spaces is $\mathrm C_{\infty}^{\mathrm{hHil}/\mathbb K}>2.0408\ldots$.
(This value can easily be improved, but it is far from the upper bound $\pi$.)

\textbf{Outline of content.}
In Section \ref{sec:Clarkson},
 we consider and discuss the relations between the notions of uniform convexity in the sense of Clarkson, Dixmier, and permutation type.
In Section \ref{sec:universalC}, we discuss the associated universal Banach algebras.
In the following sections we will consider various methods which can be used to obtain estimates for the Magnus expansion
 but all of which are variants of the resolvent method.
In Section \ref{sec:resUniform}, we discuss the general principles of the resolvent approach and consider the ``delay method''.
In Section \ref{sec:resChrono}, we consider the  ``chronological decompositon method''.
In Section \ref{sec:resKernel}, we consider the resolvent generating and estimating kernels.
In Section \ref{sec:resConcrete},
 we see in explicit terms that how uniform convexity affects the guaranteed convergence radius of the
 Magnus expansion and, in particular, of the Cayley transform of the time-ordered exponential.
In Section \ref{sec:resBCH}, we consider the case of the Baker--Campbell--Hausdorff expansion.
In Section \ref{sec:resConclude}, we make remarks considering the applicability of the
 resolvent method in the case of operators on Hilbert spaces and Banach--Lie algebras.
In Appendix \ref{sec:int}, some properties of  positive integral operators on $[0,1]$ are reminded.

\textbf{Acknowledgements.}
The author would like to thank Istv\'an \'Agoston in connection to Perron--Frobenius theory.

\snewpage
\begin{commentx}
\tableofcontents
\end{commentx}
\snewpage
\section{Uniform convexity in Banach algebras}\plabel{sec:Clarkson}
\subsection{Uniform convexity --- definitions}\plabel{ss:unidef}~\\

In order to understand how uniform convexity enters into the picture,
let us review the notion(s) of uniform convexity we will use.
For guidance, we can still consider the case of operators acting on uniformly convex Banach spaces other than Hilbert spaces.
The standard definition for uniform convexity is
\begin{defin}\plabel{def:UC} (Uniform convexity in the sense of Clarkson \cite{C1}.)
A Banach space $\mathfrak B$ is uniformly convex ($\mathcal{UC}$) if to
each $\varepsilon\in (0,2]$, there corresponds a value $ \delta(\varepsilon) >0$ such that the conditions
$|x|_{\mathfrak B}=|y|_{\mathfrak B}=1$ and $|x-y|_{\mathfrak B}\geq\varepsilon$
imply
\begin{equation}
\left|\dfrac{x+y}2\right|_{\mathfrak B}\leq1-\delta(\varepsilon).
\tag{$\mathcal{UC}_\delta$}\plabel{eq:UC}
\end{equation}
\begin{proof}[Remark]
Instead of $\varepsilon\in (0,2]$, any right-neighbourhood of $0$ can be prescribed for $\varepsilon$.
In fact, uniform convexity is induced by any appropriate sequence $\varepsilon_n\searrow0$.
\renewcommand{\qedsymbol}{$\triangle$}
\end{proof}
\end{defin}
Clarkson \cite{C1} shows that the $L^p$ spaces for $1<p<+\infty$ are uniformly convex
with $\delta(\varepsilon)=1-(1-(\varepsilon/2)^q)^{1/q}$ where $q=\max(p,\frac{p}{p-1})$.

A condition of weaker type is given by
\begin{defin}\plabel{def:UMC} (Uniform mean convexity.)
A Banach space $\mathfrak B$ is mean uniformly convex if there is a number $1\leq q<+\infty$
such that $|x|_{\mathfrak B},|y|_{\mathfrak B}\leq1$
implies
\begin{equation}
\dfrac{\left|x+y\right|_{\mathfrak B}+\left|x-y\right|_{\mathfrak B} }4\leq2^{-\frac1q}.
\tag{$\mathcal{UMC}_q$}\plabel{eq:UMC}
\end{equation}
\proofremarkqed{It is sufficient to ask for  $|x|_{\mathfrak B}=|y|_{\mathfrak B}=1$; see later.
}
\end{defin}

However, the definitions above can be used for algebras only in a limited way, as
the  operator algebras on $L^p$ spaces are typically not even mean uniformly convex.
For this reason, yet inspired by bounded operators on $L^p$ spaces for $1<p<+\infty$, we take the
\begin{defin}\plabel{def:UMD} (Uniform mean convexity of Dixmier type.)
 A Banach algebra $\mathfrak A$ is a $\mathcal{UMD}_q$-algebra, $0\leq q<+\infty$,
if  $X,Y,Z,W\in\mathfrak A$ implies
\begin{equation}
\left|\frac{XZ+YZ+XW-YW}4\right|_{\mathfrak A}\leq 2^{-\frac1q}\max(|X|_{\mathfrak A},|Y|_{\mathfrak A})\max(|Z|_{\mathfrak A},|W|_{\mathfrak A}).
\tag{$\mathcal{UMD}_q$}\plabel{eq:UMD}
\end{equation}
We say that $\mathfrak A$ is $\mathcal{UMD}$-convex, if it is a $\mathcal{UMD}_q$-algebra with some $0\leq q<+\infty$.
\begin{proof}[Remark]
$\mathbb R$ is a $\mathcal{UMD}_1$-algebra; but $\mathbb C$, in the usual way,
is only a $\mathcal{UMD}_{2}$-algebra.
\renewcommand{\qedsymbol}{$\triangle$}
\end{proof}
\end{defin}
We will see that the bounded operators on an $L^p$ space for $1<p<+\infty$ form a
$\mathcal{UMD}_q$-algebra with $q=\max(p,\frac{p}{p-1})$.
Moreover, any Banach algebra which is a $\mathcal{UMC}_q$-space is automatically a $\mathcal{UMD}_q$-algebra.

In terms of the definitions, one can say that we have passed from a uniform convexity property of order $1$
to a uniform  convexity property of order $2$, which is weaker but more widely applicable.
Now, the mean convexity properties above were selected because they are the
weakest conditions among many similar ones.
However, for our purposes, an even weaker uniform convexity property of order $4$ will suffice:

\begin{defin}\plabel{def:UMQ} (Uniform mean convexity of Kleinian permutation type.)
 A Banach algebra $\mathfrak A$ is a $\mathcal{UMQ}_q$-algebra, $0\leq q<+\infty$,
if  $S_1,S_2,S_3,S_4\in\mathfrak A$ implies
\begin{multline}
\left|\frac{S_1S_2S_3S_4+S_2S_1S_3S_4+S_1S_2S_4S_3-S_2S_1S_4S_3}4\right|_{\mathfrak A}\leq\\\leq 2^{-\frac1q}\cdot |S_1|_{\mathfrak A} \cdot |S_2|_{\mathfrak A} \cdot |S_3|_{\mathfrak A} \cdot |S_4|_{\mathfrak A} .
\tag{$\mathcal{UMQ}_q$}\plabel{eq:UMQ}
\end{multline}
We say that $\mathfrak A$ is  $\mathcal{UMQ}$-convex, if it is a $\mathcal{UMQ}_q$-algebra with some $0\leq q<+\infty$.
\begin{proof}[Remark]
All commutative Banach algebras are $\mathcal{UMQ}_1$.
\renewcommand{\qedsymbol}{$\triangle$}
\end{proof}
\end{defin}
It is easy to see that condition \eqref{eq:UMD} implies \eqref{eq:UMQ};  thus this latter condition is the weakest one here.
(One can also see that condition \eqref{eq:UMQ} is far from encompassing all conceivably relevant permutation patterns.
In fact, nontrivial patterns of higher order are easy to create even from \eqref{eq:UMD}.)

At this point it becomes understandable how uniform convexity will have consequences
regarding the Magnus expansion:
The Magnus commutators are linear combination of permutation monomials.
As long as the permutation pattern of \eqref{eq:UMQ} is sufficiently abundant in the Magnus expansion
(or just in the case of the expansion of the Cayley transform of the exponential), it leads to convergence improvement
relative to the general Banach algebraic case.
In the rest of the paper we translate this to technical terms.
The resolvent method of Mielnik,  Pleba\'{n}ski \cite{MP} will be used.
\snewpage

\subsection{On the variants of uniform convexity}
\plabel{ss:Clarkson}~\\

The objective of this section is to motivate Definitions \ref{def:UMD} and \ref{def:UMQ}.
Let us recall
\begin{theorem}[Boas \cite{Bo} (1940)]
\plabel{th:B}
Consider the Banach space $L^p(\mu)$ where $1<p<+\infty$.
Let us denote the norm by $|\cdot|$.
Let $q=\max(p,\frac{p}{p-1})$ and $q'=\min(p,\frac{p}{p-1})$ .
Let $r$ be such that  $q\leq r <+\infty$, and $r'=\frac{r}{r-1}$,  thus $1<r'\leq q'$.
Then
\[\left(|x-y|^r+|x+y|^r\right)^{\frac1r}\leq2 \left(\frac{|x|^{r'}+|y|^{r'}}2\right)^{\frac1{r'}}.\]
(This is actually the special case ``$s=r'$'' of Boas' inequality.)
\qed
\end{theorem}
The important special case is
\begin{theorem}[Clarkson \cite{C1} (1936)]
\plabel{th:C}
Consider the Banach space $L^p(\mu)$ where $1<p<+\infty$.
Let us denote the norm by $|\cdot|$.
Let $q=\max(p,\frac{p}{p-1})$ and $q'=\min(p,\frac{p}{p-1})$ .
Then
\[\left(|x-y|^q+|x+y|^q\right)^{\frac1q}\leq2 \left(\frac{|x|^{q'}+|y|^{q'}}2\right)^{\frac1{q'}}.\]
(This is actually a subset of Clarkson's inequalities.)
\qed
\end{theorem}
This leads to
\begin{theorem}[Clarkson \cite{C1}]
\plabel{th:UC}
Consider the Banach space $L^p(\mu)$ where $1<p<+\infty$.
Let $q=\max(p,\frac{p}{p-1})$.
Then the space $L^p(\mu)$ is uniformly convex with
\[\delta(\varepsilon)=1-(1-(\varepsilon/2)^q)^{1/q}.\eqed\]
\end{theorem}
\begin{remark}
\plabel{rem:Hanner}
For sufficiently nontrivial measures $\mu$, Hanner \cite{Ha} obtains the optimal version of Theorem \ref{th:UC}.
See Mitrinovi\'c,   Pe\v{c}ari\'c, Fink \cite{MPF} for further discussion.
\qedremark
\end{remark}
Another easy consequence of Clarkson's inequality is
\begin{theorem}\plabel{th:UMC}
Consider the Banach space $L^p(\mu)$ where $1<p<+\infty$.
Let us denote the norm by $|\cdot|$.
Let $q=\max(p,\frac{p}{p-1})$.
Then the space $L^p(\mu)$ is  mean uniformly convex with property $\mathcal{UMC}_q$.
\begin{proof} Assume that $|x|,|y|\leq1$. By comparing means, and using Clarkson's inequality,
and then using the norm assumptions here,  we find
\begin{multline*}
\frac{|x+y|+|x-y|}4=\frac12\cdot\frac{|x+y|+|x-y|}2\leq \frac12\cdot\left(\frac{|x-y|^q+|x+y|^q}2\right)^{\frac1q}=
\\
=\frac12 \cdot 2^{-\frac1q}\cdot \left(|x-y|^q+|x+y|^q\right)^{\frac1q}\leq
\frac12 \cdot 2^{-\frac1q}\cdot2\left(\frac{|x|^{q'}+|y|^{q'}}2\right)^{\frac1{q'}}
\leq 2^{-\frac1q}.\qedhere
\end{multline*}
\end{proof}
\end{theorem}
More generally, beyond the context of $L^p$ spaces,
\begin{lemma}\plabel{lem:UMC}
In the property  $\mathcal{UMC}_q$ $(1\leq q<+\infty)$
the requirement `$|x|_{\mathfrak B},|y|_{\mathfrak B}\leq1$' can be replaced by `$|x|_{\mathfrak B},|y|_{\mathfrak B}=1$'.
\begin{proof}
Let denote the norm by $|\cdot|$.
Suppose we know \eqref{eq:UMC} only under the second condition.
Let us consider the case  $|x|\geq |y|>0$.
Then it is easy to see that
\begin{multline*}
\frac{|x+y|+|x-y|}4\leq \frac{\left|\frac{|x|-|y|}{|x|}x\right|+\left|\frac{|x|-|y|}{|x|}x\right|}4
+\frac{\left|\frac{|y|}{|x|}x+y\right|+\left|\frac{|y|}{|x|}x-y\right|}4\leq \\
\leq 2\frac{|x|-|y|}{4}+ |y|\frac{\left|\frac{x}{|x|}+\frac{y}{|y|} \right| +\left|\frac{x}{|x|}-\frac{y}{|y|} \right| }{4}
\leq 2^{-\frac1q}(|x|-|y|)+2^{-\frac1q}|y|\leq2^{-\frac1q}|x|
\end{multline*}
implies \eqref{eq:UMC} with the first condition. The other cases are similar.
\end{proof}
\end{lemma}
\begin{theorem}\plabel{lem:UUMC}
Uniform convexity $(\mathcal{UC})$ implies uniform mean convexity $(\mathcal{UMC})$.
\begin{proof}
Let us denote the norm by $|\cdot|$.
Let us assume that one has uniform convexity with a function $\delta$.
Let us consider $x,y$ such that $|x|=|y|=1$.
One of the following three cases holds: (a) $|x+y|\geq 1$; (b) $|x-y|\geq 1$; (c) $|x+y|, |x-y|<1$.
Then by uniform convexity (a) ${|x-y|}<2(1-\delta(1))$; (b) ${|x+y|}<2(1-\delta(1))$; or simply (c) $|x+y|+|x-y|<2$.
In cases (a) and (b),
$\frac{|x+y|+|x-y|}{4}\leq1-\frac{\delta(1)}2$ holds;
and in case (c), $\frac{|x+y|+|x-y|}{4}\leq\frac12$
holds. Thus, ultimately, the choice
$2^{-\frac1q}=\max\left(\frac12,1-\frac{\delta(1)}2\right)$
is sufficient for \eqref{eq:UMC}.
\end{proof}
\end{theorem}
This summarizes the most important phenomena related of uniform convexity of first order
(which are either well-known or trivial).
Let us consider how these statements translate to some conditions of second order:
\begin{theorem}\plabel{th:DB}  (A ``Dixmier's version'' of Boas' inequality)
Consider the Banach space $L^p(\mu)$ where $1<p<+\infty$.
Assume that $X,Y,Z,W$ are bounded operators on $L^p(\mu)$.
Let us denote the operator norm by $\|\cdot\|$.
Let $q=\max(p,\frac{p}{p-1})$ and $q'=\min(p,\frac{p}{p-1})$.
Let $r$ be such that  $q\leq r <+\infty$, and $r'=\frac{r}{r-1}$,  thus $1<r'\leq q'$.
Then
\begin{equation}
\left\|\frac{XZ+YZ+XW-YW}4\right\|\leq2^{-\frac1r} \left(\frac{\|X\|^{r'}+\|Y\|^{r'}}2\right)^{\frac1{r'}} \left(\frac{\|Z\|^{r'}+\|W\|^{r'}}2\right)^{\frac1{r'}}.
\plabel{eq:DB}
\end{equation}
\begin{proof}
Let $x\in L^p(\mu)$ be arbitrary.
Then
\[\left|\left(\frac{XZ+YZ+XW-YW}4\right)x\right|=\frac14\left|  X(Zx+Wx)+Y(Zx-Wx)\right|;\]
by the properties of the operator norm,
\[\ldots\leq\frac14\left( \|X\||Zx+Wx| +\|Y\||Zx-Wx|\right);\]
by H\"older's inequality,
\[\ldots\leq\frac14\left( \|X\|^{r'} +\|Y\|^{r'}\right)^{\frac 1{r'}}\left(|Zx+Wx|^r+|Zx-Wx|^r\right)^{\frac1r};\]
using Boas' inequality,
\[\ldots\leq\frac14\left( \|X\|^{r'} +\|Y\|^{r'}\right)^{\frac 1{r'}}\cdot2 \left(\frac{|Zx|^{r'}+|Wx|^{r'}}2\right)^{\frac1{r'}};\]
by the properties of the operator norm,
\[\ldots\leq\frac14\left( \|X\|^{r'} +\|Y\|^{r'}\right)^{\frac 1{r'}}\cdot2 \left(\frac{\|Z\|^{r'}+\|W\|^{r'}}2\right)^{\frac1{r'}}\cdot|x|;\]
which is arithmetically
\[\ldots= 2^{-\frac1r}\left(\frac{ \|X\|^{r'} +\|Y\|^{r'}}2\right)^{\frac 1{r'}}\cdot\left(\frac{\|Z\|^{r'}+\|W\|^{r'}}2\right)^{\frac1{r'}}\cdot|x|.\]
As this estimate is valid for any $x\in L^p(\mu)$, we obtain the statement.
\end{proof}
\end{theorem}
\begin{theorem}\plabel{th:DC}  (A ``Dixmier's version'' of Clarkson's inequality)
Consider the Banach space $L^p(\mu)$ where $1<p<+\infty$.
Assume that $X,Y,Z,W$ are bounded operators on $L^p(\mu)$.
Let us denote the operator norm by $\|\cdot\|$.
Let $q=\max(p,\frac{p}{p-1})$ and $q'=\min(p,\frac{p}{p-1})$ .
Then
\begin{equation}
\left\|\frac{XZ+YZ+XW-YW}4\right\|\leq2^{-\frac1q} \left(\frac{\|X\|^{q'}+\|Y\|^{q'}}2\right)^{\frac1{q'}} \left(\frac{\|Z\|^{q'}+\|W\|^{q'}}2\right)^{\frac1{q'}}.
\plabel{eq:DC}
\end{equation}
\begin{proof}
This is an immediate corollary of the previous Theorem.
\end{proof}
\end{theorem}
\begin{theorem} \plabel{th:A}
Consider the Banach space $L^p(\mu)$ where $1<p<+\infty$.
Let $q=\max(p,\frac{p}{p-1})$.
Then the bounded operators on $L^p(\mu)$ form a $\mathcal{UMD}_q$-algebra  with the operator norm.
\begin{proof}
This follows from Theorem \ref{th:DC} immediately.
\end{proof}
\end{theorem}
More generally, beyond $L^p$ spaces,
\begin{theorem} \plabel{lem:UUMD}
Suppose that the Banach space $\mathfrak B$ is a $\mathcal{UMC}_q$-space.
Then the bounded operators on $\mathfrak B$ form a $\mathcal{UMD}_q$-algebra with the operator norm.
\begin{proof}
Let us denote the norm on $\mathfrak B$ by $|\cdot|$, and the operator norm by $\|\cdot\|$.
Let $x\in \mathfrak B$ be arbitrary.
Then
\[\left|\left(\frac{XZ+YZ+XW-YW}4\right)x\right|=\frac14\left|  X(Zx+Wx)+Y(Zx-Wx)\right|;\]
by the properties of the operator norm,
\[\ldots\leq\frac14\left( \|X\||Zx+Wx| +\|Y\||Zx-Wx|\right)\leq \max(\|X\|, \|Y\|  )\frac{|Zx+Wx| +|Zx-Wx|}4;\]
and, by the \eqref{eq:UMC} property, and the properties of the operator norm,
\[\ldots\leq2^{-\frac1q}\max(\|X\|, \|Y\|  ) \max(|Zx|, |Wx|  )\leq 2^{-\frac1q}\max(\|X\|, \|Y\|  ) \max(\|Z\|, \|W\|  )|x| .\]
As this is valid for any $x$, the statement follows.
\end{proof}
\end{theorem}

One can notice that the condition \eqref{eq:UMD} of Theorem \ref{th:A}
is quite distant from the inequalities   of Theorem \ref{th:DC} and Theorem \ref{th:DB}.
In fact, even the Jordan--von Neumann constant can be inserted in the middle.
Let us recall that the Banach space $\mathfrak B$ satisfies the Jordan--von Neumann condition with $C$ if the inequality
\[|x+y|^2+|x-y|^2\leq C\cdot 2(|x|^2+|y|^2)\]
holds for any $x,y\in\mathfrak B$ (cf. Jordan, von Neumann \cite{JN}).
This condition is vacuous for $C=2$, and nontrivial with $1\leq C<2$.

\begin{theorem}\plabel{th:DJN}
(An operator algebraic consequence of the Jordan--von Neumann condition.)
Assume that  the Banach space $\mathfrak B$ satisfies the Jordan--von Neumann condition with $1\leq C<2$.
Assume that $X,Y,Z,W$ are bounded operators on $\mathfrak B$.
Let us denote the operator norm by $\|\cdot\|$.
Then
\begin{equation}
\left\|\frac{XZ+YZ+XW-YW}4\right\|\leq\sqrt{\frac C2}\cdot \left(\frac{\|X\|^{2}+\|Y\|^{2}}2\right)^{\frac1{2}} \left(\frac{\|Z\|^{2}+\|W\|^{2}}2\right)^{\frac1{2}}.
\plabel{eq:DJN}
\end{equation}
\begin{proof}
This is analogous to the proof of Theorem \ref{th:DB}.
\end{proof}
\end{theorem}
\begin{theorem}\plabel{th:DDJN}
Assume that  the Banach space $\mathfrak B$ satisfies the Jordan--von Neumann condition with $1\leq C<2$.
Then the bounded operators on $\mathfrak B$ form a $\mathcal{UMD}_q$-algebra with
\[2^{-\frac1q}=\sqrt{\frac C2}.\]
\begin{proof}
This follows from Theorem \ref{th:DJN} immediately.
\end{proof}
\end{theorem}
\proofremarkqed{
This is in accordance to the (for nontrivial measures) optimal choice of $C=2^{1-\frac2q}$  with
$q=\max(p,\frac{p}{p-1})$ for $L^p$ spaces with $1<p<+\infty$, cf. Clarkson \cite{C2}.
}
Another line of statements is that if the Banach algebra $\mathfrak A$ is a $\mathcal{UMC}_q$-space, then it is a
$\mathcal{UMD}_q$-algebra, etc.
(That is we consider the regular representations.)
As the proofs of these statements are analogous to the statements for the operator algebras except simpler, we leave them to the reader.
This, hopefully, demonstrates that the condition of second order \eqref{eq:UMD} can be applied relatively widely.
Finally, we note
\begin{theorem}\plabel{th:UMQ}
For a Banach algebra, property \eqref{eq:UMD} implies \eqref{eq:UMQ}.
\begin{proof}
Consider \eqref{eq:UMD} with $X=S_1S_2$, $Y=S_2S_1$, $Z=S_3S_4$, $W=S_4S_3$.
\end{proof}
\end{theorem}
\begin{remark}\plabel{rem:UC}
There \textit{are} Banach algebras where the conditions of first order \eqref{eq:UC} or \eqref{eq:UMC},
or even their original, stronger versions are valid.
It is shown by Dixmier \cite{Dx}, and, ultimately, by McCarthy \cite{MC} that Clarkson's uniform convexity (Theorem \ref{th:UC})
extends to the Schatten classes of Hilbert space operators.
Moreover,  McCarthy \cite{MC} shows that Clarkson's and Boas' inequalities (Theorem \ref{th:C} and Theorem \ref{th:B})
extend to the Schatten classes; see Simon \cite{Si} for further discussion.
Theorem \ref{th:DC} and Theorem \ref{th:DB} also extend, and were, in fact, already used by Dixmier \cite{Dx}
in order to obtain his results.
\qedremark
\end{remark}
\snewpage
\section{Universal algebras and the convergence problem}
\plabel{sec:universalC}
As we are not seeking exact convergence bounds,
universal algebras could be omitted from the discussion;
it suffices merely to use effective estimates.
Yet, universal algebras can be used to describe the nature of the convergence problem,
and demonstrate that how the various notions of convergence differ from each other.
~\\

\subsection{Some special algebras}~\\

In this paper we consider only unital Banach algebras. Thus we make

\begin{conven}
\plabel{conv:u}
For $\mathcal{UMD}_q$- and $\mathcal{UMQ}_q$-algebras over
 $\mathbb K=\mathbb R$ or $\mathbb C$,
we will assume that $q\geq1$.
Furthermore, in case of $\mathcal{UMD}_q$ over $\mathbb K=\mathbb C$
we will also assume $q\geq2$.
\qedexer
\end{conven}

As these kinds of algebras are characterized by norm inequalities, certain universal (i. e. ``free'') algebras can be defined.
For the sake of simplicity, we start by algebras generated by (non-commutative) variables $Y_\lambda$ ($\lambda\in\Lambda$)
such that $|Y_\lambda|=1$.

Now we describe the construction of the universal algebras $\mathrm F^{\mathcal A}[ Y_\lambda :\lambda\in\Lambda]$,
where $\mathcal A$ is a placeholder for $\mathcal{UMD}_q/\mathbb K$ or $\mathcal{UMQ}_q/\mathbb K$.
First, we consider the unital non-commutative polynomial algebra
$\mathrm F_{\mathbb K}[ Y_\lambda :\lambda\in\Lambda]$.
We start with an original set of norm inequalities containing all `$|Y_\lambda|\leq1$'
and  `$|\boldsymbol 1|\leq1$' symbolically.
Next, we introduce further norm inequalities iteratively, from  the norm relations of normed algebras in general, and also from the conditions of
\eqref{eq:UMD} or \eqref{eq:UMQ}:
We do this in a manner such that we always have symbolical expressions `$|X|\leq u$' where
$X$ is a concrete element of $\mathrm F_{\mathbb K}[ Y_\lambda :\lambda\in\Lambda]$, and $c$ is a concrete element of $[0,+\infty)$.
If `$|X_i|\leq u_i$' ($1\leq i\leq 4$)  are older relations, $\lambda\in \mathbb K$,
then `$|X_1+X_2|\leq u_1+u_2$', `$|\lambda X_1|\leq |\lambda|u_1$', `$| X_1X_2|\leq u_1u_2$' are newer relations;
if $\mathcal A=\mathcal{UMD}_q/\mathbb K$, then
\[\text{ `$\left|\displaystyle{\frac{X_1X_3+X_2X_3+X_1X_4-X_2X_4}{4}   }\right|\leq 2^{-\frac1q}\max(u_1,u_2)\max(u_3,u_4)$'}\]
is another new relation;
and
if $\mathcal A=\mathcal{UMQ}_q/\mathbb K$, then
\[\text{ `$\left|\displaystyle{\frac{X_1X_2X_3X_4+X_2X_1X_3X_4+X_1X_2X_4X_3-X_2X_1X_4X_3}{4}   }\right|\leq 2^{-\frac1q}u_1u_2u_3u_4$'}\]
is another new relation.
Then we introduce the seminorm $|\cdot|_{\mathrm F\mathcal A}$, such that from any $X\in\mathrm F_{\mathbb K}[ Y_\lambda :\lambda\in\Lambda]$ we let
\[ |X|_{\mathrm F\mathcal A^{\mathrm{pre}}} :=\inf \{u: \text{  `$|X|\leq u$' is previously generated  } \}.\]
It is easy to see that $|X|_{\mathrm F\mathcal A^{\mathrm{pre}}}<+\infty$ (in fact,  majorized by the monomially induced $\ell^1$ norm).
Now $|Y_\lambda|_{\mathrm F\mathcal A^{\mathrm{pre}}}=1$, because of the existence of the trivial representation
sending $Y_\lambda$ to $1$.
Thus, $\mathrm F_{\mathbb K}[ Y_\lambda :\lambda\in\Lambda]$ becomes a semi-normed algebra
with $|\cdot|_{\mathrm F\mathcal A^{\mathrm{pre}}}$.
(Actually it is normed as there are plenty of representations of
$\mathrm F_{\mathbb K}[ Y_\lambda :\lambda\in\Lambda]$ with operators acting on $L^p$ spaces, even if with somewhat decreased norms.)
Next, we complete $( \mathrm F_{\mathbb K}[ Y_\lambda :\lambda\in\Lambda],|\cdot|_{\mathrm F\mathcal A^{\mathrm{pre}}} )$.
This completion may induce factorization by elements of norm $0$.
(But we know that, in the present case, it does not.)
Due to the nature of the relations, we know that the completed algebra
$( \mathrm F^{\mathcal A}[ Y_\lambda :\lambda\in\Lambda],|\cdot|_{\mathrm F\mathcal A} )$
also satisfies the relations \eqref{eq:UMD} or \eqref{eq:UMQ}.
This realizes the Banach-algebra generated by $Y_\lambda$ with $|Y_\lambda|=1$, such that
the polynomials of generated $Y_\lambda$ have the greatest possible norm allowed by
\eqref{eq:UMD} or \eqref{eq:UMQ}.
Regarding $|\cdot|_{\mathrm F\mathcal A^{\mathrm{pre}}}$ and $|\cdot|_{\mathrm F\mathcal A}$,
our notation may seem sloppy, because we have not indicated the set of variables.
However, introducing new variables will not decrease the norms:
Indeed, even adding the further assumption that the new variables are equal to $0$ will not.
The construction allows several modifications.

For our purposes, it is better to consider the algebra
$\mathrm F^{\mathcal A}([a,b))$.
This is constructed analogously.
We start with  $\mathrm F_{\mathbb K}^{* }([a,b))$
which is generated by various $Z_{ [c,d )}$  with $ \emptyset\neq[ c,d )\subset[a,b)$
subject to the conditions $ Z_{ [c,e )} +Z_{ [e,d )}=Z_{ [c,d )}$   for $c< e<  d$.
(The direction of the half-open intervals have no importance.)
Then we impose $|Z_{[ c,d )}|=|d-c|$ similarly, and further norm relations
coming from the Banach algebra structure and from the conditions  \eqref{eq:UMD} or \eqref{eq:UMQ}, as before;
in order to obtain $|\cdot|_{\mathrm F\mathcal A^{\mathrm{pre}} }$.
Then it is completed to $( \mathrm F^{\mathcal A}([a,b)),
|\cdot|_{\mathrm F\mathcal A} )$.
In fact, $I\mapsto Z_{I}$ can be extended as a Banach algebra valued interval measure $\mathrm Z^{\mathcal A}_{[a,b)}$ which allows to
take product measures,  which allow to integrate the characteristic functions of simplices.
If $r>0$, then $r \cdot \mathrm Z^{\mathcal A}_{[0,1)}$ is isometric to $\mathrm Z^{\mathcal A}_{[0,r)}$ (by scaling).
Thus $\mathrm Z^{\mathcal A}_{[0,1)}$ is quite appropriate for a prototype of an $\mathcal A$-algebra valued measure.
(Formally, we could write $\mathrm Z^{\mathcal A}_{[a,b)}(t)=Y_t^{\mathcal A}\,\mathrm dt|_{[a,b)}$,
but it is not much meaningful.)
Moreover, there is little danger in using the same notation $|\cdot|_{\mathrm F\mathcal A}$
for the norms in $\mathrm F^{\mathcal A}[ Y_\lambda :\lambda\in\Lambda]$ and $\mathrm F^{\mathcal A}([a,b))$,
 because there is a common generalization over (appropriate) measures.
(That is when the tautological ``non-commutative valued'' measure generalizes an ordinary measure,
not only a discrete measure or interval measure.)

Note that for $X\in\mathrm F_{\mathbb K}[ Y_\lambda :\lambda\in\Lambda]$, the inequality $|X|_{\mathrm F\mathcal A}\leq |X|_{\ell^1}$
holds, where $|\cdot|_{\ell^1}$ is the monomially induced $\ell^1$ norm.
This means that there is a (weakly contractive)
natural continuous map $\mathrm F_{\mathbb K}^1[ Y_\lambda :\lambda\in\Lambda]\rightarrow \mathrm F^{\mathcal A}[ Y_\lambda :\lambda\in\Lambda]$.
In particular, for any element $\mathrm F_{\mathbb K}^1[ Y_\lambda :\lambda\in\Lambda]$ we can take the norm $|\cdot|_{\mathrm F\mathcal A}$.
Similarly for $\mathrm F_{\mathbb K}^1([a,b))\rightarrow \mathrm F^{\mathcal A}([a,b))$.

We can define the $\mathcal A$-characteristic of the Magnus expansion as the formal power series
\[\Theta^{\mathcal A}(x)=\sum_{k=1}^\infty\Theta^{\mathcal A}_kx^k, \]
where
\begin{equation}
\Theta^{\mathcal A}_k=\left|\int_{0\leq t_1\leq\ldots\leq t_k\leq1}
 \mu_k(\mathrm Z^{\mathcal A}_{[0,1)}(t_1),\ldots,\mathrm Z^{\mathcal A}_{[0,1)}(t_n))\right|_{\mathrm F\mathcal A}.
\plabel{eq:muniva}
\end{equation}
(The integral makes sense, as it already makes sense $\mathrm F_{\mathbb K}^1([0,1))$, in fact its variation measure in bounded
by the the corresponding variation measure; only the value of the norm is in question.)
Now, if $\phi$ is an $\mathcal A$-valued ordered measure, then
\[|\mu_{k,\mathrm R}(\phi)|\leq\Theta^{\mathcal A}_k\cdot\left( \int |\phi|\right)^k\]
holds, with equality realized for $\phi=r\cdot \mathrm Z_{[0,1)}^{\mathcal A}$.
Thus if $\Theta^{\mathcal A}(\smallint |\phi|)<+\infty$, then the Magnus expansion is absolutely convergent; while
if $\Theta_{\mathrm{real}}^{\mathcal A}(s) =+\infty$ holds for $s>0$, then the Magnus expansion of $\phi=s\cdot\mathrm Z_{[0,1)}^{\mathcal A}$
is not absolutely convergent. Moreover, if $s$ is greater than the convergence radius of
$\Theta^{\mathcal A}(x)$, then the Magnus expansion is  divergent.

Now, one can define a norm $|\cdot|_{\mathrm F\mathrm h\mathcal A}$ ``between'' $|\cdot|_{\mathrm F\mathcal A}$ and $|\cdot|_{\ell^1}$.
Let us consider  $X\in \mathrm F_{\mathbb K}[ Y_\lambda :\lambda\in\Lambda]$.
If
\begin{equation}
X=\sum_{\chi\in L^1(\Lambda;\mathbb N)} X_\chi
\plabel{eq:dechi}
\end{equation}
is a decomposition to homogenous components in $Y_\lambda$, then we set
\[
|X|_{\mathrm F\mathrm h\mathcal A}=\sum_{\chi\in L^1(\Lambda;\mathbb N)} |X_\lambda|_{\mathrm F\mathcal A}.
\]
Then
\begin{equation}
|X|_{\mathrm F\mathcal A}\leq |X|_{\mathrm F\mathrm h\mathcal A}
\plabel{eq:mhom}
\end{equation}
holds.
Furthermore, $|\cdot|_{\mathrm F\mathrm h\mathcal A}$ makes $\mathrm F_{\mathbb K}[ Y_\lambda :\lambda\in\Lambda]$
a normed algebra, which can be completed to a Banach algebra $\mathrm F^{\mathrm h\mathcal A}[ Y_\lambda :\lambda\in\Lambda]$.
There is a continuous homomorphism
$( \mathrm F^{\mathcal A}[ Y_\lambda :\lambda\in\Lambda],|\cdot|_{\mathrm F\mathcal A} )\rightarrow
\mathrm F^{\mathrm h\mathcal A}[ Y_\lambda :\lambda\in\Lambda],|\cdot|_{\mathrm F\mathrm h\mathcal A} )$;
but more practically, the norms can be compared on $\mathrm F_{\mathbb K}[ Y_\lambda :\lambda\in\Lambda]$,
or even on $\mathrm F_{\mathbb K}^1[ Y_\lambda :\lambda\in\Lambda]$.

If $X\in \mathrm F^{*}_{\mathbb K}([a,b))$, then $X$
is better to be first decomposed according to global homogeneity (``degree in $Z$'').
According to global homogeneity,
\[X=\sum_{k\in\mathbb N} X_k\]
can be written.
In a global homogeneity degree $k$, the component $X_k$ can be represented
by a step function $h_k$ with respect to a rectangular measure on $[a,b)^k$, such that
\begin{equation}
X_k=\int h_k(t_1,\ldots, t_k) \mathrm Z^{\mathcal A}_{[0,1)}(t_1)\ldots \mathrm Z^{\mathcal A}_{[0,1)}(t_k) .
\plabel{eq:las2}
\end{equation}
($X_0=h_1\cdot 1$.)
Then, with some abuse of notation, we set
\begin{multline}
|X|_{\mathrm F\mathrm h\mathcal A}=\sum_{k\in\mathbb N}
\int_{0\leq t_1\leq \ldots\leq t_k\leq1}
\left|\sum_{\sigma\in\Sigma_k}h_k(t_{\sigma(1)},\ldots, t_{\sigma(k)})
 \mathrm Z^{\mathcal A}_{[0,1)}(t_{\sigma(1)})\ldots \mathrm Z^{\mathcal A}_{[0,1)}(t_{\sigma(k)})
\right|_{\mathrm F\mathcal A}
\\=
\text{``$\displaystyle{
\sum_{k\in\mathbb N}
\int_{0\leq t_1\leq \ldots\leq t_k\leq1}
\left|\sum_{\sigma\in\Sigma_k}h_k(t_{\sigma(1)},\ldots, t_{\sigma(k)})Y_{t_{\sigma(1)}} \ldots Y_{t_{\sigma(k)}}
\right|_{\mathrm F\mathcal A}
\mathrm dt_1\ldots \mathrm dt_k
}$''}.
\plabel{eq:las3}
\end{multline}
We will not clarify the formula above further because it is quite clear what to do.
Note that the integrand will be a rectangular step function restricted.
Then
\begin{equation}
|X|_{\mathrm F\mathcal A}\leq |X|_{\mathrm F\mathrm h\mathcal A}
\plabel{eq:mize}
\end{equation}
holds.
Again, $\mathrm F^{\mathrm h\mathcal A}([a,b))$ can prepared, but what is more important,
the norms can be compared on $\mathrm F^{*}_{\mathbb K}([a,b))$, or even on $\mathrm F^{1}_{\mathbb K}([a,b))$.

We define the $\mathcal A$-characteristic of the Magnus commutators as
\[\Theta^{\mathrm h\mathcal A}(x)=\sum_{k=1}^\infty\Theta_k^{\mathrm h\mathcal A}x^k,\]
where
\[\Theta_k^{\mathrm h\mathcal A}=\frac1{k!}\cdot|\mu_k(Y_1,\ldots,Y_k)|_{\mathrm F\mathcal A}.\]
Then
\begin{equation}
\Theta_k^{\mathcal A}\leq \Theta_k^{\mathrm h\mathcal A},
\plabel{eq:compar}
\end{equation}
and in fact,
\[ \Theta^{\mathrm h\mathcal A}_k=\left|\int_{0\leq t_1\leq\ldots\leq t_k\leq1}
 \mu_k(\mathrm Z^{\mathrm h\mathcal A}_{[0,1)}(t_1),\ldots,\mathrm Z^{\mathrm h\mathcal A}_{[0,1)}(t_n))\right|_{\mathrm F\mathrm h\mathcal A} .\]
(Again, instead of $\mathrm Z^{\mathrm h\mathcal A}_{[0,1)}$
we could take $\mathrm Z^{1}_{[0,1)}$; the integral will be well-defined even in $|\cdot|_{\ell^1}$, only the norm is of question.
In making the comparison in \eqref{eq:compar}, we can think that there is a single element
of $\mathrm F^1_{\mathbb K}([0,1))$ for which the norms are compared.)

The  convergence radius of $\Theta^{\mathcal A}(x)$ is, of course, greater or
equal than the convergence radius of $\Theta^{\mathrm h\mathcal A}(x)$.
This expresses something very simple:
The Magnus expansion can be estimated through the Magnus commutators; but there might
analytical phenomena helping the Magnus expansion to do better.
Indeed, this might be the case for $\mathcal A=\mathcal{UMD}_q/\mathbb K$.
But not for $\mathcal A=\mathcal{UMQ}_q/\mathbb K$:

\begin{lemma}
\plabel{eq:Asharp}
For $\mathcal A=\mathcal{UMQ}_q/\mathbb K$, equality holds in
\eqref{eq:mhom} and \eqref{eq:mize}.
\begin{proof}
The relation \eqref{eq:UMQ} is compatible to being homogeneous splitting,
thus it stays respected.
\end{proof}
\end{lemma}
Thus, for $\mathcal A=\mathcal{UMQ}_q/\mathbb K$,
we can say that $|\cdot|_{\mathrm F\mathcal A}$ is homogeneously induced.
Then, it is sufficient to compute the norm  for monomials grouped up to permutations.
Here it is not true that $|\cdot|_{\mathrm F\mathcal A}$ is monomially induced,
but the concrete construction shows that it is
quasi monomially generated.
By quasi monomially generated we mean the following.
We let
\[\Xi^{\mathrm{eval}}(S_1,S_2,S_3,S_4)=\frac{S_1S_2S_3S_4+S_2S_1S_3S_4+S_1S_2S_4S_3-S_2S_1S_4S_3}4.\]
A quasi-monomial is an expression obtained from the $Y_\lambda$ by taking products and formal 4-variable operations $\Xi^{\mathrm{symb}}$
in some order (i.~e.~along a tree).
For any homogeneity degree in \eqref{eq:dechi}, there are only finitely many quasi-monomials.
The evaluated version of quasi-monomial is when $\Xi^{\mathrm{symb}}$ is replaced by $\Xi^{\mathrm{eval}}$.
Then, in each homogeneity degree $\chi$,
we can take the corresponding quasi-monomials $M$; and
the symbolic relations
\[\text{` $ |M^{\mathrm{eval}}|\leq \left(2^{-\frac1q}\right)^{\deg_\Xi M} $'}\]
alone will generate norm linearly.
I. e., in each homogeneity degree, we have to minimize
\[\sum_M|c_M|\left(2^{-\frac1q}\right)^{\deg_\Xi M}  \]
subject to the linear constraint
\[\sum_M c_M M^{\mathrm{eval}}=X_\chi. \]
Thus, for $\mathbb K=\mathbb R$, the norm can be computed by linear programming.
But even in the complex case, if the coefficients of $X_\chi$ are real,
then coefficients of the minimizing representation are also real, thus
linear programming suffices.
\begin{lemma}
\plabel{lem:linprog}
For $\mathcal A=\mathcal{UMQ}_q/\mathbb K$, the value
$\Theta_k^{\mathcal A}$ can be computed by linear programming;
it does not depend on the choice of $\mathbb K$.
\begin{proof}
This follows from the previous discussion.
\end{proof}
\end{lemma}

Hence, in practical sense,  the convergence radius of the Magnus expansion
is much easier to describe for $\mathcal{UMQ}_q$ (compared to $\mathcal{UMD}_q$): there is no difference
between the expansion and commutator estimates (the terms `$\mathrm h$' can be dropped),
there is no dependence on the base field $\mathbb K$, and the norms of individual expressions
(with real coefficients) are easily computable (in theory).

\snewpage

The earlier discussions also apply if $\mu_k(X_1,\ldots,X_k)$ is
replaced with $\mu_k^{(\lambda)}(X_1,\ldots, X_k)$ where $\lambda\in [0,1]$,
yielding $\Theta^{(\lambda),\mathcal A}(x)$  instead of $\Theta^{\mathcal A}(x)$, etc.

Furthermore,
from $\mathrm F^{\mathrm h\mathcal A} [ Y_\lambda :\lambda\in\Lambda]$
we can pass to
$\mathrm F^{\mathrm h\mathcal A,\mathrm{loc}} [ Y_\lambda :\lambda\in\Lambda]$,
the locally convex algebra induced with the components of
the global grading (``degree in $Y$'').
In the homogeneously induced case of $\mathcal A=\mathcal{UMQ}_q/\mathbb K$, this is
 has the following consequence:
If $\Theta_{\mathrm{real}}^{ \mathcal A}(s) =+\infty$ holds for $s>0$,
then $\Rexp\left(s\cdot \mathrm Z^{ \mathcal A}_{[0,1)}\right)$ does not allow a logarithm in $\mathrm F^{\mathcal A}([0,1))$.
The reason is that, by quasi-nilpotency, it must allow a unique one (up to $2\pi\mathrm i$) in
$\mathrm F^{\mathcal A,\mathrm{loc}}([0,1))$, but such one that its global norm is $+\infty$.
Thus, this logarithm is not in $\mathrm F^{\mathcal A}([0,1))$.
Consequently, the Magnus expansion is divergent, moreover, $\Rexp\left(s\cdot \mathrm Z^{ \mathcal A}_{[0,1)}\right)$ is not log-able.
Thus, the spectrum of $\Rexp\left(s\cdot \mathrm Z^{ \mathcal A}_{[0,1)}\right)$ intersects $(-\infty,0]$, yielding, in particular,  $\Theta_{\mathrm{real}}^{ (\lambda), \mathcal A}(s) =+\infty$ for some $\lambda\in[0,1]$.

\begin{commentx}
\begin{remark}
\plabel{rem:transpose}
For $\mathcal A=\mathcal{UMD}_q/\mathbb K$ or $\mathcal{UMQ}_q/\mathbb K$,
all our norms $|\cdot|_{\mathrm F\mathcal A}$, etc.
are transposition invariant.
This means the following, the formal transposition operation
\[(Y_{\lambda_1}\cdot\ldots\cdot Y_{\lambda_k})^{\mathrm f\top}=Y_{\lambda_k}\cdot\ldots\cdot Y_{\lambda_1}\]
extends linearly to an isometry of the relevant algebras.
The reason is simple: all the generating rules for the norms are transposition invariant.
(That is the transposed norm rules are implied by the original ones.
It is relatively particular for the $\mathcal A$ prescribed here,
it is not hard to construct permutation inequalities which are not transposition invariant.)
The same can be said about the formal skew-transposition operation extended linearly from
\[(Y_{\lambda_1}\cdot\ldots\cdot Y_{\lambda_k})^{\mathrm f*}=(-Y_{\lambda_k})\cdot\ldots\cdot (-Y_{\lambda_1});\]
which is a minor variant.

These properties are useful if we consider the adjoint representation (cf. Part I), which we do not.
\qedremark
\end{remark}
\end{commentx}
~
\snewpage

\subsection{Certain more general algebras}~\\

One can consider $\mathcal A$ more generally:

\begin{defin}
We say that the unital Banach algebra $\mathfrak A$ over $\mathbb K$ is quasi-free with generators
$\bar Y_\lambda$ $(\lambda\in\Lambda)$,
if the following conditions hold:

(i) The
$\bar Y_\lambda$ $(\lambda\in\Lambda)$ generate $\mathfrak A$ (i. e. their noncommutative polynomials are dense in $\mathfrak A$).

(ii) $|\bar Y_\lambda|_{\mathfrak A}=1$.

(iii)
For any non-commutative polynomial $P$ over $\mathbb K$, the inequality
\begin{equation}
|P(X_{\lambda_1},\ldots,X_{\lambda_k})|_{\mathfrak A}\leq |P(\bar Y_{\lambda_1},\ldots,\bar Y_{\lambda_k})|_{\mathfrak A}
\plabel{eq:class}
\end{equation}
holds whenever
\[
X_{\lambda_i}=c_{\lambda_i}1_{\mathfrak A} +\sum_{\nu\in\Lambda_i} c_{\lambda_i, \nu}\bar Y_\nu
\]
such that $\Lambda_i\subset \Lambda$ is finite  $c_\lambda, c_{\lambda,\mu}\in\mathbb K$, and
\[|c_{\lambda_i}| +\sum_{\nu\in\Lambda_i} |c_{\lambda_i, \nu}|\leq1.\eqedexer\]
\end{defin}

Then, there is a natural (weakly) contractive map $\mathrm F^1_{\mathbb K}[Y_\lambda\,:\,\lambda\in \Lambda]\rightarrow \mathfrak A$.
Due to condition (iii), it is easy to see that quasi-free  Banach algebras are symmetric in their generators.
Consequently, one can essentially freely relabel in the generating variables.
We let $\mathcal A$ be the abstract isomorphism class of $\mathfrak A$ with distinguished generators
(with the choice $\mathbb K=\mathbb R$ or $\mathbb K=\mathbb C$ noted).
Then we can write
\[\mathfrak A=\mathrm F^{\mathcal A}[Y_\lambda\,:\,\lambda\in\lambda]\]
and
\[|\cdot|_{\mathfrak A}=|\cdot|_{\mathrm F\mathcal A}\]
(which is not an actual construction but an interpretation of matters).
Regarding general  quasi-free classes, and considering
the map $\mathrm F^1_{\mathbb K}[Y_\lambda\,:\,\lambda\in \Lambda]\rightarrow  \mathrm F^{\mathcal A}_{\mathbb K}[Y_\lambda\,:\,\lambda\in \Lambda]$,
and the relationship of the two algebras above,
it is actually better to write $Y_\lambda^{\mathcal A}$ for the generators of the latter algebra
as there might be algebraic relations between them (like commutativity).

Let us consider a quasi-free class $\mathcal A$ associated to
a countably infinite index set $\Lambda$.
(There is no essential difference between the infinite cases.
One can also restrict to fewer variables easily.
However, extension from finitely many variables to more variables is typically ambiguous,
although there are unique minimally normed and maximally normed quasi-free extensions.
Hence, for our purposes, quasi-free classes with countably infinite index sets are needed.)
\snewpage

Let $\mathcal I\subset\mathbb R$ be a nontrivial interval; say, half-open as before (but that is not essential).
We can consider $\mathrm F^*_{\mathbb K}(I)$ as before.
Then we can impose a norm on $\mathrm F^*_{\mathbb K}(I)$ such that if
\[X=P(Z_{[a_1,b_1)},\ldots, Z_{[a_k,b_k)})\]
with pairwise disjoint $[a_i,b_i)$, then we define the seminorm
\[ |X|_{\mathrm F\mathcal A^{\mathrm{pre}}}= \left|P((b_1-a_1)Y_{[a_1,b_1)},\ldots, (b_k-a_k)Y_{[a_k,b_k)}) \right|_{\mathrm F\mathcal A}.\]
(The variables $Y_*$ can be labeled arbitrarily.)
Due to the quasi-freeness property,  this is well-defined.
Then we can complete the algebra with respect to $|\cdot|_{\mathrm F\mathcal A^{\mathrm{pre}}}$,
factoring $Z_{[a,b)}$ into $Z^{\mathcal A}_{[a,b)}$.
Thus we obtain $\mathrm F^{\mathcal A}(I)$.
In fact, one can also see that from  $\mathrm F^{\mathcal A}(I)$ (with generators $Z_{[a,b)}$ distingushed)
one can reconstruct $\mathrm F^{\mathcal A}[Y_\lambda\,:\,\lambda\in\lambda]$ with countably infinitely many generators.
Hence the countably infinite discrete quasifree algebras and the continuous quasifree algebras are not that different from each other.
Again, there is a natural (weakly) contractive map $\mathrm F^1_{\mathbb K}(I)\rightarrow\mathrm F^{\mathcal A}(I) $; etc.

Similarly as before, if $\mathcal A$ is a quasi-free class, then,
generating from the norm relations homogeneous in the generators,
 there is an associated homogeneous class $\mathrm h\mathcal A$,
such that $\mathrm F^{\mathrm h\mathcal A}[Y_\lambda\,:\,\lambda\in \Lambda]$
is ``between'' $\mathrm F^1_{\mathbb K}[Y_\lambda\,:\,\lambda\in \Lambda]$ and  $\mathrm F^{\mathcal A}[Y_\lambda\,:\,\lambda\in \Lambda]$; etc.
The details are left to the reader.

\begin{conven}
\plabel{conv:qf}
In the forthcoming discussions $\mathcal A$ will always be a quasi-free class given with countably infinitely many generators.
(However, the reader may conveniently assume that $\mathcal A$ is  $\mathcal{UMD}_q/\mathbb K$ or $\mathcal{UMQ}_q/\mathbb K$.)
\qedexer
\end{conven}

\begin{remark}
\plabel{rem:qf}
For $\mathcal A=\mathcal{UMD}_q/\mathbb K$ or $\mathcal{UMQ}_q/\mathbb K$,  the corresponding inequality \eqref{eq:class}
holds more generally, even under the conditions
\[|X_{\lambda_i}|_{\mathfrak A}\leq|\bar Y_{\lambda_i}|_{\mathfrak A}.\]
For this reason, these particular choices for $\mathcal A$ could be termed as ``free classes''.

The definition for the quasifree classes is certainly more modest.
(Later $\mathrm{Lie}/\mathbb K$ will be an example for that.)
In principle, the quasifree classes do not really describe algebras but the
relationship between a sufficiently generic measure and an algebra.
\qedremark
\end{remark}
\snewpage
\begin{commenty}
\subsection{Selection estimates}~\\

Finally, we can note that $\mathcal A=\mathcal{UMQ}_q/\mathbb K$ has yet some additional good properties.
\begin{point}

Assume that $\Lambda^{\mathrm L}$, $\Lambda$, $\Lambda^{\mathrm R}$ are index sets.
We define $\mathrm F_{\mathbb K}Y_{\langle\Lambda^{\mathrm L} |\Lambda|\Lambda^{\mathrm R}\rangle}$
as the noncommutative free algebra over $\mathbb K$ generated by the variables
$Y_\xi^{\mathrm L}$ $(\xi\in\Lambda^{\mathrm L})$,
$Y_\lambda$ $(\lambda\in\Lambda)$,
$Y_\mu^{\mathrm R}$ $(\mu\in\Lambda^{\mathrm R})$,
but subject to the relations
$Y_\xi^{\mathrm L} Y_\mu^{\mathrm L}=0$, $Y_\xi Y_\mu^{\mathrm L}=0$,
$Y_\xi^{\mathrm R} Y_\mu^{\mathrm L}=0$,  $Y_\xi^{\mathrm R} Y_\mu=0$, $Y_\xi^{\mathrm R} Y_\mu^{\mathrm R}=0$.
Then every element $X\in \mathrm F_{\mathbb K}Y_{\langle\Lambda^{\mathrm L} |\Lambda|\Lambda^{\mathrm R}\rangle}$
can uniquely be written in form
\[X=X_{|}+\sum_{\xi}Y_{\xi}^{\mathrm L}X_{\xi|}+\sum_{\mu}X_{|\mu}Y_{\mu}^{\mathrm R}
+\sum_{\xi,\mu}Y_{\xi}^{\mathrm L}X_{\xi|\mu}Y_{\mu}^{\mathrm R},\]
where $X_{|}, X_{\xi|}, X_{|\mu}, X_{\xi|\mu}\in\mathrm F_{\mathbb K}[Y_{\lambda}:\lambda\in\Lambda] $.
We define
\[|X|_{\mathrm F\mathrm h\mathcal A}':=
\left|X_{|}\right|_{\mathrm F\mathrm h\mathcal A}
+\sum_{\xi}\left|X_{\xi|}\right|_{\mathrm F\mathrm h\mathcal A}
+\sum_{\mu}\left|X_{|\mu}\right|_{\mathrm F\mathrm h\mathcal A}
+\sum_{\xi,\mu}\left|X_{\xi|\mu}\right|_{\mathrm F\mathrm h\mathcal A}.\]
\end{point}
\begin{lemma}
\plabel{lem:diar}
For $\mathcal A=\mathcal{UMQ}_q/\mathbb K$, $|\cdot|_{\mathrm F\mathrm h\mathcal A}'$ defines
a quasi-monomially induced norm which makes $\mathrm F_{\mathbb K}Y_{\langle\Lambda^{\mathrm L} |\Lambda|\Lambda^{\mathrm R}\rangle}$
a $\mathcal A=\mathcal{UMQ}_q/\mathbb K$-algebra.
\begin{proof}
Quasi-monomial inducedness is easy to prove, which, in turn,  makes easier to check the property \eqref{eq:UMQ}.
\end{proof}
\end{lemma}

\begin{theorem}
\plabel{lem:antiest1}
Let $\mathcal A=\mathcal{UMQ}_q/\mathbb K$.
Assume that
$n\geq2$, and
$X\in \mathrm F_{\mathbb K}[Y_{\lambda}:\lambda\in\Lambda]$ is $n$-homogeneous (in $Y$),
and
\[X=\sum_{\xi,\mu\in\Lambda} Y_\xi X_{\xi,\mu} Y_\mu,\]
where $X_{\xi,\mu}\in\mathrm F_{\mathbb K}[Y_{\lambda}:\lambda\in\Lambda] $.
Then

(a)
\[
\underbrace{\frac1{n^2}\left(1-\frac2n\right)^{n-2}}_{ \frac{\mathrm e^{-2}}{n(n-1)} <}\,
\sum_{\xi,\mu\in\Lambda} | X_{\xi,\mu} |_{\mathrm F\mathcal A}
\leq
|X|_{\mathrm F\mathcal A}.
\]

(b) Moreover; for any subset  $\Lambda_{00}\subset \Lambda\times\Lambda$,
\[
\,
\left|\sum_{(\xi,\mu)\in\Lambda_{00}}  Y_\xi X_{\xi,\mu} Y_\mu \right|_{\mathrm F\mathcal A}
\leq\underbrace{n^2\left(1-\frac2n\right)^{2-n}}_{ <{\mathrm e^{2}}{n(n-1)} }
|X|_{\mathrm F\mathcal A}.
\]
\begin{proof}
(a) Let us consider the map $\mathrm F_{\mathbb K}[Y_{\lambda}:\lambda\in\Lambda]\rightarrow
\mathrm F_{\mathbb K}Y_{\langle\Lambda  |\Lambda|\Lambda \rangle}$ induced by
\[Y_\lambda\mapsto \frac1nY_\lambda^{\mathrm L}+\left(1-\frac2n\right)Y_\lambda+ \frac1nY_\lambda^{\mathrm R} .\]
Due to the universal properties of $|\cdot|_{\mathrm F\mathrm h\mathcal A}$, this map is (weakly) contractive
if the target is equipped by $|\cdot|_{\mathrm F\mathrm h\mathcal A}'$.
This leads to the statement.

(b) This is a direct consequence of part (a).
\end{proof}
\end{theorem}
\snewpage

\begin{theorem}\plabel{lem:antiest2}
Let $\mathcal A=\mathcal{UMQ}_q/\mathbb K$.
Suppose that $k\geq 2$.
Then

(a)
\begin{multline*}
\underbrace{\frac1{n^2}\left(1-\frac2n\right)^{n-2}}_{ \frac{\mathrm e^{-2}}{n(n-1)} <}
\int_{t_1=0}^1 \int_{t_k=0}^1
\left|\int h_k(t_1,\ldots, t_{k}) \mathrm Z^{\mathcal A}_{[0,1)}(t_2)\ldots \mathrm Z^{\mathcal A}_{[0,1)}(t_{k-1})
\right|_{\mathrm F\mathcal A}  \,\mathrm dt_1\,\mathrm dt_2
\leq\\\leq
\left|\int h_k(t_1,\ldots, t_k) \mathrm Z^{\mathcal A}_{[0,1)}(t_1)\ldots \mathrm Z^{\mathcal A}_{[0,1)}(t_k)
\right|_{\mathrm F\mathcal A};
\end{multline*}
assuming that $h:[0,1]^k\rightarrow\mathbb K$ is Lebesgue-integrable.
($|\mathrm Z^1_{[0,1)}(t)|_{\mathrm F\mathcal A}=\mathbf 1_{[0,1)}(t)=\mathrm dt|_{[0,1)}$.)

(b) Moreover; for any Lebesgue integrable function $0\leq f\leq1$ on $[0,1]\times[0,1]$,
\begin{multline*}
\int_{t_1=0}^1 \int_{t_k=0}^1 f(t_1,t_k)
\left|\int h_k(t_1,\ldots, t_{k}) \mathrm Z^{\mathcal A}_{[0,1)}(t_2)\ldots \mathrm Z^{\mathcal A}_{[0,1)}(t_{k-1})
\right|_{\mathrm F\mathcal A}  \,\mathrm dt_1\,\mathrm dt_2
\leq\\\leq
\underbrace{n^2\left(1-\frac2n\right)^{2-n}}_{ <{\mathrm e^{2}}{n(n-1)} }
\left|\int h_k(t_1,\ldots, t_k) \mathrm Z^{\mathcal A}_{[0,1)}(t_1)\ldots \mathrm Z^{\mathcal A}_{[0,1)}(t_k)
\right|_{\mathrm F\mathcal A}.
\end{multline*}

\begin{proof}
This is the analogue of the previous statement.
\end{proof}
\end{theorem}
\begin{commentx}
\begin{remark}
\plabel{rem:antiest}
In the context of the Lemma \ref{lem:antiest1} the inequality
\[
|X|_{\mathrm F\mathcal A}
\leq
\sum_{\xi,\mu\in\Lambda} | X_{\xi,\mu} |_{\mathrm F\mathcal A},
\]
holds; in the context of the Lemma \ref{lem:antiest2} the inequality
\begin{multline*}
\left|
\int h_k(t_1,\ldots, t_k) \mathrm Z^{\mathcal A}_{[0,1)}(t_1)\ldots \mathrm Z^{\mathcal A}_{[0,1)}(t_k)
\right|_{\mathrm F\mathcal A}
\leq\\\leq
\int_{t_1=0}^1 \int_{t_k=0}^1
\left|\int h_k(t_1,\ldots, t_k) \mathrm Z^{\mathcal A}_{[0,1)}(t_2)\ldots \mathrm Z^{\mathcal A}_{[0,1)}(t_{k-1})
\right|_{\mathrm F\mathcal A}  \,\mathrm dt_1\,\mathrm dt_2
,
\end{multline*}
holds; trivially, without restriction on $\mathcal A$.
\qedremark
\end{remark}
\end{commentx}
\end{commenty}

\snewpage
\section{The basics of the resolvent approach and the delay method}
\plabel{sec:resUniform}
\subsection{The principles of the resolvent approach}
\plabel{ss:resprinc}
~\\

Our objective is to estimate the convergence radius of
$\Theta^{\mathcal A}(x)=\sum_{k=1}^\infty \Theta_k^{\mathcal A} x^k$,
where
\begin{align}
\Theta_k^{\mathcal A}&=\left|\int_{t_1\leq\ldots\leq t_k\in[0,1]}
\mu_k(\mathrm Z^{1}_{[0,1]}(t_1)\ldots \mathrm Z^{1}_{[0,1]}(t_k)) \right|_{\mathrm F\mathcal A}
\notag\\
&=\left|\int_{\lambda=0}^1\int_{\mathbf t=(t_1,\ldots,t_k)\in[0,1]^k} \lambda^{\asc(\mathbf t)}
(\lambda-1)^{\des(\mathbf t)}\mathrm Z^{1}_{[0,1]}(t_1)\ldots \mathrm Z^{1}_{[0,1]}(t_k) \,\mathrm d\lambda\right|_{\mathrm F\mathcal A}.
\plabel{eq:multcomp}
\end{align}
(Strictly speaking, $\mathrm Z^{\mathcal A}_{[0,1]}$ would have been the correct notation,
but it does not matter as the integral is well defined already on the $\ell^1$ level.)
For $k\geq 1$, we have
\[\Theta_k^{(\lambda),\mathcal A}=\left| \int_{\mathbf t=(t_1,\ldots,t_k)\in[0,1]^k} \lambda^{\asc(\mathbf t)}
(\lambda-1)^{\des(\mathbf t)}\mathrm Z^{1}_{[0,1]}(t_1)\ldots \mathrm Z^{1}_{[0,1]}(t_k) \right|_{\mathrm F\mathcal A}.\]

Recall from Part I, that for $\lambda\in[0,1]$, we have already considered the expressions
\[{\mathrm C}_\infty^{(\lambda)}=\begin{cases}
2&\text{if }\lambda=\frac12,
\\
\dfrac{2\artanh (1-2\lambda)}{1-2\lambda}=\dfrac{\log\dfrac{1-\lambda}{\lambda}}{1-2\lambda}&\text{if }\lambda\in(0,1)
\setminus\{\frac12\},\\
+\infty&\text{if }\lambda\in\{0,1\};
\end{cases}\]
\[w^{(\lambda)}=1/{\mathrm C}_\infty^{(\lambda)};\]
\[{\mathrm C}_\infty^{(\lambda),\boldsymbol\varepsilon}=\text{``$\left|\log \frac\lambda{\lambda-1}\right|$''}=
\begin{cases}
 \sqrt{\pi^2+\left(\log\frac\lambda{1-\lambda}\right)^2}&\text{if}\quad\lambda\in(0,1),\\
 \\
+\infty&\text{if}\quad\lambda\in\{0,1\};
\end{cases}\]
\[w^{(\lambda),\boldsymbol\varepsilon}=1/{\mathrm C}_\infty^{(\lambda),\boldsymbol\varepsilon}.\]

Let ${\mathrm C}_\infty^{(\lambda),\mathcal A}$ be the convergence radius of
$\Theta^{(\lambda),\mathcal A}(x)$, and let
$w^{(\lambda),\mathcal A}=1/{\mathrm C}_\infty^{(\lambda),\mathcal A}$.

\begin{lemma}
\plabel{lem:biBound}
\[ w^{(\lambda),\boldsymbol\varepsilon}\leq w^{(\lambda),\mathcal A}\leq w^{(\lambda)};\]
or, equivalently,
\[{\mathrm C}_\infty^{(\lambda)}\leq {\mathrm C}_\infty^{(\lambda),\mathcal A} \leq{\mathrm C}_\infty^{(\lambda),\boldsymbol\varepsilon} . \]
\begin{proof}
Considering $\limsup_k \sqrt[k]{\Theta_k^{(\lambda),\mathcal A}}$, we obtain the first set of estimates:
The upper estimate for $w^{(\lambda),\mathcal A}$ is the general $\ell^1$ estimate,
while the lower estimate for $w^{(\lambda),\mathcal A}$ comes from replacing $\mathrm Z^{\mathcal A}_{[0,1]}$ by the Lebesgue measure.
\end{proof}
\end{lemma}
\snewpage
\begin{lemma}\plabel{lem:cont}
For $\lambda_1,\lambda_2\in(0,1)$,
\[\left|\mathrm C^{(\lambda_1),\mathcal A}_\infty-\mathrm C^{(\lambda_2),\mathcal A}_\infty \right|\leq
\left| \log\frac{\lambda_1}{1-\lambda_1}  - \log\frac{\lambda_2}{1-\lambda_2} \right|\]
holds.
\begin{proof}
Indirectly, let us assume that
\[\mathrm C^{(\lambda_1),\mathcal A}_\infty-\mathrm C^{(\lambda_2),\mathcal A}_\infty>
\left| \log\frac{\lambda_1}{1-\lambda_1}  - \log\frac{\lambda_2}{1-\lambda_2}\right|. \]
Then
\begin{multline*}
\mathcal R^{(\lambda_2)}(\Rexp ( (t\cdot \mathrm Z^{\mathcal A}_{[0,1)})   )=\\=\frac{\lambda_1(1-\lambda_1)}{\lambda_2(1-\lambda_2)}
\mathcal R^{(\lambda_1)}\left(\Rexp \left( (t\cdot \mathrm Z^{\mathcal A}_{[0,1)})\boldsymbol.
\left(\log\frac{\lambda_1}{1-\lambda_1}  - \log\frac{\lambda_2}{1-\lambda_2}\right)
\mathbf 1_{[1,2)}  \right) \right)+\frac{\lambda_2-\lambda_1}{\lambda_2(1-\lambda_2)}
\end{multline*}
exists for
\begin{equation}
|t|<\mathrm C^{(\lambda_1),\mathcal A}_\infty- \left| \log\frac{\lambda_1}{1-\lambda_1}  - \log\frac{\lambda_2}{1-\lambda_2}\right|,
\plabel{eq:resdom}
\end{equation}
where $t\in\mathbb C$.
If it exists, then it must be analytic in $t$.
Ultimately, we find that
\[\mathrm C^{(\lambda_2),\mathcal A}_\infty\geq
\mathrm C^{(\lambda_1),\mathcal A}_\infty- \left| \log\frac{\lambda_1}{1-\lambda_1}  - \log\frac{\lambda_2}{1-\lambda_2}\right|.
\]
This is a contradiction.
\end{proof}
\end{lemma}
\begin{theorem}
\plabel{th:cont}
$\lambda\mapsto w^{(\lambda),\mathcal A}$ is continuous as a $[0,1/2]$-valued function;
$\lambda\mapsto {\mathrm C}_\infty^{(\lambda),\mathcal A}$ is continuous as a $[2,+\infty]$-valued function.
\begin{proof}
This is an immediate consequence of the previous lemma.
\end{proof}
\end{theorem}
\snewpage

Let
\[w^{(\log),\mathcal A}=\max_{\lambda\in[0,1]}  w^{(\lambda),\mathcal A},\]
and
\begin{equation}
{\mathrm C}_\infty^{(\log),\mathcal A}=\min_{\lambda\in[0,1]}  {\mathrm C}_\infty^{(\lambda),\mathcal A}.
\plabel{eq:Clogdef}
\end{equation}
Here $w^{(\log),\mathcal A} =1/{\mathrm C}_\infty^{(\log),\mathcal A}$ holds.
\begin{lemma}
\plabel{lem:uniest}
\[2\leq {\mathrm C}_\infty^{(\log),\mathcal A}\leq\pi ;\]
or, equivalently,
\[\frac1\pi\leq w^{(\log),\mathcal A}\leq\frac12 .\]
\begin{proof}
This an immediate consequence of Lemma \ref{lem:biBound}.
\end{proof}
\end{lemma}

Let ${\mathrm C}_\infty^{\mathcal A}$ be the convergence radius of
$\Theta^{\mathcal A}(x)$, and let
$w^{\mathcal A}=1/{\mathrm C}_\infty^{\mathcal A}$.
\begin{lemma}
\plabel{lem:mmm}
\[ {\mathrm C}_\infty^{(\log),\mathcal A}\leq {\mathrm C}_\infty^{\mathcal A};\]
or, equivalently,
\[w^{\mathcal A}\leq w^{(\log),\mathcal A}.\]
\begin{proof}
It is sufficient to prove the first statement.
$(\lambda, t)\mapsto \lambda+(1-\lambda)\Rexp(t\cdot\mathrm Z^{\mathcal A}_{[0,1]}) $ is analytic and invertible
on  $[0,1]_\lambda\times \intD(0,[{\mathrm C}_\infty^{(\log),\mathcal A})$,
thus the resolvent expression is also analytic.
By the Cauchy formula
\[\frac{f^{(k)}(0)}{k!}=\frac1{2\pi\mathrm i}\int_{z\in \partial\Dbar(0,r)^\circlearrowleft}\frac{f(z)}{(z-0)^k}\,\mathrm dz,\]
we have some uniform estimates (independently from $\lambda$) for the
coefficients in $t$, which can be integrated in $\lambda$.
\end{proof}
\end{lemma}
\begin{cor}
\plabel{cor:half}
If ${\mathrm C}_\infty^{(1/2),\mathcal A}>2$, then
\[ 2<{\mathrm C}_\infty^{(\log),\mathcal A}\leq {\mathrm C}_\infty^{\mathcal A}.\]
\begin{proof}
By Theorem \ref{th:cont}, $\mathrm C^{(\lambda),\mathcal A}_{\infty}>\mathrm C^{(\lambda)}_{\infty}$
for $\lambda\sim1/2$.
This is already sufficient for $\mathrm C^{(\log),\mathcal A}_{\infty}>\mathrm C^{(\log)}_{\infty}=2$.
\end{proof}
\end{cor}

Note, however, that ${\mathrm C}_\infty^{(\log),\mathcal A}$ has more meaning than a simple numerical value set up by \eqref{eq:Clogdef}.
It is exactly threshold value which guarantees the existence of $\mu_{\mathrm R}(s\cdot Z_{[0,1)}^{\mathcal A})$ realized
as $\log(\Rexp(s\cdot Z_{[0,1)}^{\mathcal A}))$ (correctly, as analytical continuation shows).
Thus it is convergence radius of the $\mathcal A$-Magnus expansion in (the stronger) logarithmic sense.
\snewpage

\subsection{Resolvent estimates via Euler's recursion and the delay method}
\plabel{ss:reseuler}
~\\

\begin{theorem}
\plabel{th:eulerest}
For $\lambda\in[0,1]$,
\[
\frac{\mathrm d}{\mathrm dx}\Theta^{(\lambda),\mathcal A}(x)
\stackrel{\forall}{\leq} (1+ \lambda\Theta^{(\lambda),\mathcal A}(x))(1+ (1-\lambda)\Theta^{(\lambda),\mathcal A}(x)).
\]
\begin{proof}[Note]
In the light of $\Theta^{(\lambda),\mathcal A}_0=0$, $\Theta^{(\lambda),\mathcal A}_1=1$; the inequality above is equivalent to
\begin{multline}
(k+1)\Theta^{(\lambda),\mathcal A}_{k+1}\leq
\lambda\cdot \Theta^{(\lambda),\mathcal A}_{k}+
(1-\lambda)\cdot \Theta^{(\lambda),\mathcal A}_{k}
+
\lambda(1-\lambda)\cdot\sum_{j=1}^{k-1} \Theta^{(\lambda),\mathcal A}_{j} \Theta^{(\lambda),\mathcal A}_{k-j}
\plabel{eq:ecala}
\end{multline}
for $k\geq 1$.
\qedno
\end{proof}
\begin{proof}
Let $k\geq 1$.
Let us consider
\[\int_{\mathbf t=(t_1,\ldots,t_{k+1})\in[0,1]^{k+1}} \lambda^{\asc(\mathbf t)}
(\lambda-1)^{\des(\mathbf t)}\mathrm Z^{1}_{[0,1]}(t_1)\ldots \mathrm Z^{1}_{[0,1]}(t_{k+1}).\]
Decomposing in $\tau=\max(t_1,\ldots,t_{k+1})$, we find this
\begin{align*}
=&\int_{\tau=0}^1 \Biggl(
\lambda\cdot\left(\int_{\mathbf t_1=(t_1,\ldots,t_{k})\in[0,\tau]^{k}} \lambda^{\asc(\mathbf t_1)}
(\lambda-1)^{\des(\mathbf t_1)}\mathrm Z^{1}_{[0,1]}(t_1)\ldots \mathrm Z^{1}_{[0,1]}(t_{k})\right)\mathrm Z^{1}_{[0,1]}(\tau)
\\
&+(\lambda-1)\cdot\mathrm Z^{1}_{[0,1]}(\tau)\left(\int_{\mathbf t_2=(t_2,\ldots,t_{k+1})\in[0,\tau]^{k}} \lambda^{\asc(\mathbf t_2)}
(\lambda-1)^{\des(\mathbf t_2)}\mathrm Z^{1}_{[0,1]}(t_2)\ldots \mathrm Z^{1}_{[0,1]}(t_{k+1})\right)
\\
&+\lambda(\lambda-1)\cdot\sum_{j=1}^{k-1}\left(\int_{\mathbf t_1=(t_1,\ldots,t_{j})\in[0,\tau]^{j}} \lambda^{\asc(\mathbf t_1)}
(\lambda-1)^{\des(\mathbf t_1)}\mathrm Z^{1}_{[0,1]}(t_1)\ldots \mathrm Z^{1}_{[0,1]}(t_{k})\right)\cdot
\\
&\cdot\mathrm Z^{1}_{[0,1]}(\tau)\left(\int_{\mathbf t_2=(t_{j+2},\ldots,t_{k+1})\in[0,\tau]^{k-j}} \lambda^{\asc(\mathbf t_2)}
(\lambda-1)^{\des(\mathbf t_2)}\mathrm Z^{1}_{[0,1]}(t_{j+2})\ldots \mathrm Z^{1}_{[0,1]}(t_{k+1})\right)
\Biggr)\,\mathrm d\tau.
\end{align*}

Applying $|\cdot|_{\mathrm F\mathcal A}$, and its submultiplicativity, we find
\[\Theta^{(\lambda),\mathcal A}_{k+1}\leq\int_{\tau=0}^1\left(
\lambda\cdot \tau^k\Theta^{(\lambda),\mathcal A}_{k}+(1-\lambda)\cdot \tau^k\Theta^{(\lambda),\mathcal A}_{k}+\lambda(1-\lambda)\cdot
\sum_{j=1}^{k-1} \tau^j\Theta^{(\lambda),\mathcal A}_{j} \tau^{k-j}\Theta^{(\lambda),\mathcal A}_{k-j}  \right)\,\mathrm d\tau.\]
Carrying out the integration in $\tau$, we obtain \eqref{eq:ecala}.
\end{proof}
\end{theorem}
Note that in the plain Banach algebraic case ($\mathcal A$ omitted), we have equality above.

\begin{theorem}
\plabel{th:dec}
 Assume that $\lambda\in(0,1)$.
If there is a $k$ such that $\Theta^{(\lambda),\mathcal A}_k<\Theta^{(\lambda)}_k$, then
$\mathrm C^{(\lambda),\mathcal A}_{\infty}>\mathrm C^{(\lambda)}_{\infty}$.
\snewpage
\begin{proof}
Let $\breve\Theta^{(\lambda),\mathcal A}(x)$ be the solution of the formal IVP
\[
\frac{\mathrm d}{\mathrm dx}\breve\Theta^{(\lambda),\mathcal A}(x)
= (1+ \lambda\breve\Theta^{(\lambda),\mathcal A}(x))(1+ (1-\lambda)\breve\Theta^{(\lambda),\mathcal A}(x))
-( \Theta^{(\lambda)}_k- \Theta^{(\lambda),\mathcal A}_k)kx^{k-1},
\]
\[\breve\Theta^{(\lambda),\mathcal A}(0)=0.\]
Then
\[\Theta^{(\lambda),\mathcal A}(x)
\stackrel{\forall}{\leq}\breve\Theta^{(\lambda),\mathcal A}(x)
\stackrel{\forall}{\leq}\Theta^{(\lambda)}(x).\]

Taking the ODE viewpont, however, we see that
$\Theta^{(\lambda),\mathcal A}_{\real}(x)$ falls behind $\Theta^{(\lambda)}_{\real}(x)$ in the very beginning (from the Taylor series).
In fact, due to the delaying term, the time lag of $\Theta^{(\lambda),\mathcal A}_{\real}(x)$  behind $\Theta^{(\lambda)}_{\real}(x)$
(in value) only grows.
This causes $\breve\Theta^{(\lambda),\mathcal A}_{\real}(x)$ to blow up later than $\Theta^{(\lambda)}_{\real}(x)$.
\end{proof}
\end{theorem}
\begin{remark}
\plabel{rem:codec}
By Theorem \ref{th:eulerest}, $\Theta^{(\lambda),\mathcal A}_k<\Theta^{(\lambda)}_k$
implies $\Theta^{(\lambda),\mathcal A}_{k+1}<\Theta^{(\lambda)}_{k+1}$.
\qedremark
\end{remark}
\begin{cor}
\plabel{cor:dec}
If there is a $k$ such that $\Theta^{(1/2),\mathcal A}_k<\Theta^{(1/2)}_k$, then
$\mathrm C^{(\log),\mathcal A}_{\infty}>2$.
\begin{proof}
By Theorem \ref{th:dec}, $\mathrm C^{(1/2),\mathcal A}_{\infty}>\mathrm C^{(1/2)}_{\infty}$.
Then Corollary \ref{cor:half} can be applied.
\end{proof}
\end{cor}

We can be systematic in the correction process of the  proof Theorem \ref{th:dec}.
First, we can correct to $\breve\Theta^{(\lambda),\mathcal A}(x)$ from $\Theta^{(\lambda)}(x)$
by $( \Theta^{(\lambda)}_k- \Theta^{(\lambda),\mathcal A}_k)kx^{k-1}$ in the ODE using the smallest possibly nontrivial $k$.
Then, we can correct to $\breve{\breve\Theta}^{(\lambda),\mathcal A}(x)$ from  $\breve\Theta^{(\lambda),\mathcal A}(x)$
by $  ( \breve\Theta^{(\lambda),\mathcal A}_m - \Theta^{(\lambda),\mathcal A}_m)mx^{m-1}$ in the ODE using the smallest
possibly nontrivial $m$; etc.
In that manner we have IVPs
\[
\frac{\mathrm d}{\mathrm dx}\hat\Theta^{(\lambda),\mathcal A}(x)
= (1+ \lambda\hat\Theta^{(\lambda),\mathcal A}(x))(1+ (1-\lambda)\hat\Theta^{(\lambda),\mathcal A}(x))
-\mathcal E_l^{(\lambda),\mathcal A}(x),
\]
\[\hat\Theta^{(\lambda),\mathcal A}(0)=0;\]
such that $\mathcal E_l^{(\lambda),\mathcal A}(x)\stackrel{\forall}{\geq}0$;
the degree of $\mathcal E_l^{(\lambda),\mathcal A}(x)$ is at most $l-1$ but the solution
$\hat\Theta^{(\lambda),\mathcal A}(x)$ agrees to $\Theta^{(\lambda),\mathcal A}(x)$ up to (including) the coefficient of $x^l$.

This approach is also useful when we do not have complete information about the
$\Theta^{(\lambda),\mathcal A}_k$ but just upper estimates.
In that case $\mathcal E_l^{(\lambda),\mathcal A}(x)$ is just used to correct the coefficients to the best known value
if it is not yet achieved.

Estimating the blow up point (i.~e.~ the convergence radius) for $\hat\Theta^{(\lambda),\mathcal A}(x)$
is  a delicate matter numerically, but we can advantageously use the information that the time delay
$\left(\Theta^{(\lambda)}\right)^{-1}\left(\hat\Theta^{(\lambda),\mathcal A}(x)\right)-x$
is monotone increasing.

This method (the ``delay method'') can be used in order to obtain explicit estimates $\mathrm C^{(\log),\mathcal A}_{\infty}$.
Nevertheless, using ODEs in the above manner is somewhat cumbersome.

This setting is very suggestive regarding what would be a relatively distinguished
family of norm inequalities of (higher) permutation type.
Indeed, for $k\geq2$,
\begin{equation}
\frac1{k!}\mu^{(1/2)}(X_1,\ldots, X_k)\leq2^{-\frac1q}\cdot\frac1{2^{k-1}}|X_1|\cdot\ldots\cdot|X_k|
\tag{$\mathcal{UMP}_q^{[k]}$}
\plabel{eq:UMP}
\end{equation}
would be such an inequality.
\snewpage

\section{The chronological decomposition method}
\plabel{sec:resChrono}
This is a kind of improved version of the delay method.
The main idea is as follows:
Assume that $\phi=\phi_1\boldsymbol.\phi_2$.
If $\mathcal R^{(\lambda)}(\Rexp(\phi_1))$ and $\mathcal R^{(\lambda)}(\Rexp(\phi_2))$
exist, then, as it was explained in part I \cite{L1}, the existence of
 $\left(1-\lambda(\lambda-1)\mathcal R^{(\lambda)}(\Rexp(\phi_1)) \mathcal R^{(\lambda)}(\Rexp(\phi_2)) \right)^{-1}$
is equivalent to the existence of
  $\mathcal R^{(\lambda)}(\Rexp(\phi_1)\Rexp(\phi_2) )\equiv \mathcal R^{(\lambda)}(\Rexp(\phi))$.
~\\

\subsection{The plain method.}~\\

Assume that $\phi=\phi_1\boldsymbol.\phi_2$ (concatenation in time).
Let $T$ be a formal commutative variable and
$Z=\mathcal R^{(\lambda)}(T\cdot\phi)$, $X=\mathcal R^{(\lambda)}(T\cdot\phi_1)$, $Y=\mathcal R^{(\lambda)}(T\cdot\phi_2)$.
Then
\begin{multline}
Z= X(1-\lambda(\lambda-1) YX)^{-1}+Y(1-\lambda(\lambda-1) XY)^{-1} \\
 +\lambda XY(1-\lambda(\lambda-1) XY)^{-1}+(\lambda-1) YX(1-\lambda(\lambda-1) YX)^{-1} .
 \plabel{eq:decor1}
\end{multline}
(cf. Part I).
Applying this for, say,
$\mathrm Z^{\mathcal A}_{[0,1)}= \mathrm Z^{\mathcal A}_{[0,1/2)}\boldsymbol.\mathrm Z^{\mathcal A}_{[1/2,1)} $, we
see that
\[\Theta^{(\lambda),\mathcal A}(T)\stackrel{\forall T}\leq
\frac{2\Theta^{(\lambda),\mathcal A}(T/2)+ (|\lambda|+|\lambda-1|)\Theta^{(\lambda),\mathcal A}(T/2)^2
}{1-|\lambda|\cdot|\lambda-1|\Theta^{(\lambda),\mathcal A}(T/2)^2}\]
In what follows, we will assume $\lambda\in[0,1]$. Then
\[\Theta^{(\lambda),\mathcal A}(T)\stackrel{\forall T}\leq
\frac{2\Theta^{(\lambda),\mathcal A}(T/2)+  \Theta^{(\lambda),\mathcal A}(T/2)^2
}{1-\lambda(1-\lambda)\Theta^{(\lambda),\mathcal A}(T/2)^2}.\]
This can be used to obtain an iterative process for the upper estimate of $\Theta^{(\lambda),\mathcal A}(T)$.
Indeed, let us assume that we already have some upper estimates regarding the first $p$ many coefficients
 $\Theta^{(\lambda),\mathcal A}_1,\ldots, \Theta^{(\lambda),\mathcal A}_p$.
This is implies that one has
 \[\Theta^{(\lambda),\mathcal A}(T)\stackrel{\forall T}\leq U_0^{(\lambda)}(T):=\Theta^{(\lambda)}(T)
 -\mathcal E^{(\lambda),\mathcal A}_{ 0}(T),\]
where $\mathcal E^{(\lambda),\mathcal A}_{p,0}(T)$ is a finite correction term with nonnegative coefficients to
incorporate earlier information from earlier.
We will assume that $\mathcal E^{(\lambda),\mathcal A}_{ 0}(T)\neq0$.
Then one has
\[\Theta^{(\lambda),\mathcal A}(T)\stackrel{\forall T}\leq U_1^{(\lambda)}(T):=
\frac{2U_0^{(\lambda)}(T/2)+  U_0^{(\lambda)}(T/2)^2
}{1-\lambda(1-\lambda)U_0^{(\lambda)}(T/2)^2}-\mathcal E^{(\lambda),\mathcal A}_{ 1}(T),\]
where again, $\mathcal E^{(\lambda),\mathcal A}_{p,1}(T)$ is a valid finite correction term with nonnegative coefficients
to our liking  but we can leave it to be $0$.
Iterating this procedure, leads to a series of estimates
\[\Theta^{(\lambda),\mathcal A}(T)\stackrel{\forall T}\leq U_{k+1}^{(\lambda)}(T):=
\frac{2U_k^{(\lambda)}(T/2)+  U_k^{(\lambda)}(T/2)^2
}{1-\lambda(1-\lambda)U_k^{(\lambda)}(T/2)^2}-\mathcal E^{(\lambda),\mathcal A}_{k+1}(T).\]
(Again, $\mathcal E^{(\lambda),\mathcal A}_{k+1}(T)$ is already allowed to be $0$.)
By induction,
\begin{equation}
\text{$U_k^{(\lambda)}(x)<\Theta^{(\lambda) }(x)$ holds for any $0<x<\mathrm C_\infty^{(\lambda)}$,}
\plabel{eq:chrondelay}
\end{equation}
 or more generally, it holds
if $U_k^{(\lambda)}(x)<+\infty$.
It is also easy to see by induction that the $U_k^{(\lambda)}(x)$ is continuous for $y\in[0,+\infty)$
as an $[0,\mathrm +\infty]$ valued function.
\snewpage

Let $\mathrm r( U_k^{(\lambda)}(T) )$ denote the convergence radius of $U_k^{(\lambda)}(T) $, i. e. the point where it blows up.
Then for $k\geq1$, this is exactly the $x\in\mathbb (0,+\infty)$, where
\begin{equation}
\lambda(1-\lambda) U_{k-1}^{(\lambda)}(x/2)^2=1.
\plabel{eq:insta}
\end{equation}
Then, by  \eqref{eq:chrondelay} and continuity, we can see that
\[\mathrm C_\infty^{(\lambda),\mathcal A}\geq \mathrm r( U_k^{(\lambda)}(T) )>\mathrm C_\infty^{(\lambda)}.\]
The point is that the solution of \eqref{eq:insta} is quite well-computable numerically.

Although not bad, the chronological decomposition method as presented above estimates the convergence
radius of $\Theta^{(\lambda),A}$ from quite earlier values $\Theta^{(\lambda),A}(x)$
(making improvements in higher coefficients relatively uneconomical).
This can be countered by non-equitemporal and multiple decompositions,
and also by some other improvements.\\

\subsection{The spectrally improved method.}\plabel{ss:spectimp}~\\

By simple arithmetic, for $k\geq2$ we can change \eqref{eq:decor1} into
\begin{equation}
Z=
 \ldots\left((1-\left(\lambda(\lambda-1) XY\right)^k\right)^{-1}+\ldots\left((1-\left(\lambda(\lambda-1) YX\right)^k\right)^{-1}
 \plabel{eq:decor2}
\end{equation}
(the exact shape is not important).
Having
$\Theta^{(\lambda),\mathcal A}(T)\stackrel{\forall T}\leq U^{(\lambda)}(T)$,
instead of just   using
\[|XY|^{\forall T}\stackrel{\forall T}\leq U^{(\lambda) }(T\cdot\smallint|\phi_1|)U^{(\lambda) }(T\cdot\smallint|\phi_2|),\]
we can use
\[|(XY)^k|^{\forall T}\stackrel{\forall T}\leq
\left( U^{(\lambda) }(T\cdot\smallint|\phi_1|)U^{(\lambda) }(T\cdot\smallint|\phi_2|)\right)^k
 -\mathcal E_{\phi_1,\phi_2}(T),\]
where $\mathcal E_{\phi_1,\phi_2}(T)$
is just any valid correction term we can find by any mean.
Ultimately, instead of using just the norm of $XY$, we go to the direction of the spectral radius of $XY$,
alleviating the theoretical constraint of the plain method.
\snewpage

\section{The resolvent kernel method}
\plabel{sec:resKernel}

\subsection{The resolvent generating kernels}
\plabel{ss:resgen}
~\\

For $p-1\geq0$,  and $t_0,t_p\in[0,1]$, we let
\begin{multline}
{\mathcal K}_{\mathrm R, p-1}^{(\lambda),\mathcal A}(t_0,t_p)=\\=\int_{\mathbf t_1=(t_1,\ldots,t_{p-1})\in[0,1]^k}
\lambda^{\asc(t_0,\mathbf t_1,t_p)}
(\lambda-1)^{\des(t_0,\mathbf t_1,t_p)}\mathrm Z^{\mathcal A}_{[0,1]}(t_1)\ldots \mathrm Z^{\mathcal A}_{[0,1]}(t_{p-1}).
\plabel{eq:rkernel}
\end{multline}
For $p-1=1$, this is a scalar valued discontinuous kernel (although it is very simple.)
For $p-1\geq1$, as the inducing functions (i. e. integrands) are continuous in $\ell^1$ sense depending on $t_0,t_p,\lambda$,
we find that the expression ${\mathcal K}_{\mathrm R, p-1}^{(\lambda),\mathcal A}(t_0,t_p)$ is continuous
as a function of $t_0,t_p,\lambda$.
We call ${\mathcal K}_{\mathrm R, p-1}^{(\lambda),\mathcal A}(t_0,t_p)$ resolvent generating kernels, as
\begin{equation}
\mu^{(\lambda)}_{\mathrm R,p+1}(\mathrm Z^{\mathcal A}_{[0,1)})=
\int_{t_1=0}^1\int_{t_p=0}^1
\mathrm Z^{\mathcal A}_{[0,1]}(t_0) {\mathcal K}_{\mathrm R, p-1}^{(\lambda),\mathcal A}(t_0,t_p) \mathrm Z^{\mathcal A}_{[0,1]}(t_{p})
\plabel{eq:rgen}
\end{equation}
holds. They also have the composition property
\begin{equation}
{\mathcal K}_{\mathrm R, p+q-1}^{(\lambda),\mathcal A}(t_0,t_p)=
\int_{t_p=0}^1 {\mathcal K}_{\mathrm R, p-1}^{(\lambda),\mathcal A}(t_0,t_p)
\mathrm Z_{[0,1)}^{\mathcal A}(t_p)
{\mathcal K}_{\mathrm R, q-1}^{(\lambda),\mathcal A}(t_{p+1},t_{p+q-1})
.
\plabel{eq:rcomp}
\end{equation}

We define some particular linear maps on $\mathrm F^{\mathcal A}([0,1))$.
For $\nu\in[0,1)$, let $\mathrm{Tns}_\nu$ be the linear map induced by the prescription
\[\mathrm{Tns}_\nu(Z_{[a,b)})=
\begin{cases}
Z_{[a+\nu,b+\nu)} &\text{if }[a,b)\subset[0,1-\nu),
\\
Z_{[a+\nu-1,b+\nu-1)} &\text{if }[a,b)\subset[1-\nu,1).
\end{cases}\]
It is easy to check that this extends to an isometry of $\mathrm F^{\mathcal A}([0,1))$.
In spirit, it sends the formal variable $Y_t$ into $Y_{t+\nu}$ if $t\in[0,1-\nu)$,
and it sends the formal variable $Y_t$ into $Y_{t+\nu-1}$ if $t\in[1-\nu,1)$.
We extend the range of $\nu$ by setting $\mathrm{Tns}_\nu=\mathrm{Tns}_{\nu+1}$.
One can see that $\mathrm{Tns}_{-\nu}=\mathrm{Tns}_{1-\nu} $ inverts
$\mathrm{Tns}_\nu$.

\begin{lemma}
\plabel{lem:shift}
Assume that $\nu\leq t_0,t_p$. Then
\[{\mathcal K}_{\mathrm R, p-1}^{(\lambda),\mathcal A}(t_0-\nu,t_p-\nu)
=\mathrm{Tns}_{-\nu}
\left({\mathcal K}_{\mathrm R, p-1}^{(\lambda),\mathcal A}(t_0,t_p)\right).\]

Similarly, if $t_0,t_p\leq 1-\nu$, then
\[{\mathcal K}_{\mathrm R, p-1}^{(\lambda),\mathcal A}(t_0+\nu,t_p+\nu)
=\mathrm{Tns}_\nu
\left({\mathcal K}_{\mathrm R, p-1}^{(\lambda),\mathcal A}(t_0,t_p)\right).\]

\begin{proof}
We give only an intuitive argument. We integrate
\[\lambda^{\asc(t_0,\mathbf t_1,t_p)}
(\lambda-1)^{\des(t_0,\mathbf t_1,t_p)}Y_{t_1}\ldots Y_{t_{p-1}}.\]
Whenever $\mathbf t$ makes an excursion into $[1-\nu,1)$ or $[0,\nu)$, respectively,
in terms of the ascent-descent patterns it introduces only a multiplier $\lambda(\lambda-1)$.
Thus the difference between very top and very bottom does not really matter if it is outside interval
of the two variables of the kernel.
(The argument can be carried out on the $\ell^1$ level, then contracted.)
\end{proof}
\end{lemma}

Another isometry on $\mathrm F^{\mathcal A}([0,1))$ can be defined as follows.
Let $\mathrm{Rfl}$ be the linear map induced by the prescription
\[\mathrm{Rfl}(Z_{[a,b)})=-Z_{[1-b,1-a)}.\]
One can see again that this extends to an isometry.
\begin{lemma}
\plabel{lem:ref}
\[{\mathcal K}_{\mathrm R, p-1}^{(1-\lambda),\mathcal A}(1-t_0,1-t_p)
=-\mathrm{Rfl}
\left({\mathcal K}_{\mathrm R, p-1}^{(\lambda),\mathcal A}(t_0,t_p)\right).\]
\begin{proof}
Again, this follows from the nature of the ascent-descent patterns.
\end{proof}
\end{lemma}
\begin{commentx}
\begin{lemma}
\plabel{lem:trann}
\[{\mathcal K}_{\mathrm R, p-1}^{(1-\lambda),\mathcal A}(t_p,t_0)
=-
\left({\mathcal K}_{\mathrm R, p-1}^{(\lambda),\mathcal A}(t_0,t_p)\right)^{\mathrm f*}.\]
\begin{proof}
Again, this follows from the nature of the ascent-descent patterns.
\end{proof}
\end{lemma}
\end{commentx}

Let us recall that
\[\mathcal R^{(\lambda)}(A)=\frac{A-1}{\lambda+(1-\lambda)A}.\]
As long as the expressions make sense, the identities
\begin{equation}
\mathcal R^{(1-\lambda)}(A^{-1})=-\mathcal R^{(\lambda)}(A)
\plabel{eq:resid1}
\end{equation}
and
\begin{equation}
A+(\lambda-1) \mathcal R^{(\lambda)}(AB) A= B^{-1}+(-\lambda) \mathcal R^{(1-\lambda)}(B^{-1}A^{-1}) B^{-1}
\plabel{eq:resid2}
\end{equation}
hold.
Furthermore, by ``real analyticity'',
\[\mathcal R^{(\lambda)}(AB)A=A\mathcal R^{(\lambda)}(BA)\]
also holds.

\begin{lemma}
\plabel{th:grandrev}
(a) Assume that $t_0\leq \tau $.
Let $T$ be a formal variable.
Then, in terms of generating functions,
\begin{multline*}
\sum_{p=2}^{\infty}\left({\mathcal K}_{\mathrm R, p-1}^{(\lambda),\mathcal A}(t_0,\tau )\right)T^{p-1}=
\lambda\cdot((\Rexp(T\cdot\mathrm Z_{[t_0,\tau )}))-1)+\lambda(\lambda-1)\cdot\\
(\Rexp(T\cdot\mathrm Z_{[t_0,\tau )})){\mathcal R}^{(\lambda)}
\left( (\Rexp(T\cdot\mathrm Z_{[\tau ,1)})) (\Rexp(T\cdot\mathrm Z_{[0,t_0)})) (\Rexp(T\cdot\mathrm Z_{[t_0,\tau )})) \right)
.
\end{multline*}
(b) Assume that $t_0\geq \tau $. Then,
\begin{multline*}
\sum_{p=2}^{\infty}\left({\mathcal K}_{\mathrm R, p-1}^{(\lambda),\mathcal A}(t_0,\tau )\right)T^{p-1}=
(\lambda-1)\cdot((   (\Rexp(T\cdot\mathrm Z_{[t_0,1)})) (\Rexp(T\cdot\mathrm Z_{[0,\tau )}))   -1)+(\lambda-1)^2\cdot
\\
{\mathcal R}^{(\lambda)}
\left( (\Rexp(T\cdot\mathrm Z_{[t_0,1)})) (\Rexp(T\cdot\mathrm Z_{[0,\tau )})) (\Rexp(T\cdot\mathrm Z_{[\tau ,t_0)})) \right)
(\Rexp(T\cdot\mathrm Z_{[t_0,1)})) (\Rexp(T\cdot\mathrm Z_{[0,\tau )}))
.
\end{multline*}
Or, written alternatively,
\begin{multline*}
\sum_{p=2}^{\infty}\left({\mathcal K}_{\mathrm R, p-1}^{(\lambda),\mathcal A}(t_0,\tau )\right)T^{p-1}=
(\lambda-1)\cdot((\Rexp(T\cdot\mathrm Z_{[\tau ,t_0)}))^{-1}-1)+\lambda(\lambda-1)\cdot
\\
{\mathcal R}^{(\lambda)}
\left( (\Rexp(T\cdot\mathrm Z_{[t_0,1)})) (\Rexp(T\cdot\mathrm Z_{[0,\tau )})) (\Rexp(T\cdot\mathrm Z_{[\tau ,t_0)})) \right)
(\Rexp(T\cdot\mathrm Z_{[\tau ,t_0)}))^{-1}
.
\end{multline*}

\begin{proof}
(a) Let us apply the notation
$U_1= \Rexp(T\cdot\mathrm Z_{[0,t_0)}))$,
$U_2= \Rexp(T\cdot\mathrm Z_{[t_0,\tau )}))$,
$U_2= \Rexp(T\cdot\mathrm Z_{[\tau ,1)}))$.
Using Lemma \ref{lem:shift},
$
\mathrm{Tns}_{t_0}
\left({\mathcal K}_{\mathrm R, p-1}^{(\lambda),\mathcal A}(0,\tau -t_0)\right)
={\mathcal K}_{\mathrm R, p-1}^{(\lambda),\mathcal A}(t_0,\tau );
$
thus we can reduce the problem to the $t_0=0$ case.
Using the (half-formal) resolvent expansion, and translating back, one finds that
the statement is
\[\lambda\cdot\biggl(
\left(1+\mathcal R^{(\lambda)}(U_3U_1)\cdot(\lambda-1)\right)
\left(1-\mathcal R^{(\lambda)}(U_2)\mathcal R^{(\lambda)}(U_3U_1)\cdot\lambda(\lambda-1)\right)^{-1}
\mathcal R^{(\lambda)}(U_2)\lambda+\]\[
\left(1+\mathcal R^{(\lambda)}(U_2)\lambda\right)
\left(1-\mathcal R^{(\lambda)}(U_3U_1)\mathcal R^{(\lambda)}(U_2)\cdot\lambda(\lambda-1)\right)^{-1}
\mathcal R^{(\lambda)}(U_3U_1)\cdot(\lambda-1)
\biggr)=\]
\[=\lambda(U_2-1)+\lambda(\lambda-1)U_2\mathcal R^{(\lambda)}(U_3U_1U_2);\]
which is an identity.
(b) follows by similar methods, and by applying \eqref{eq:resid2}.
\end{proof}
\end{lemma}

\begin{theorem}
\plabel{th:ranal}
If $\mathcal R^{(\lambda)}(\Rexp(t\cdot \mathrm Z_{[0,1)}))$ extends analytically to $x\in\intD(0,r)$, then so is
\[\sum_{p=2}^{\infty}\left({\mathcal K}_{\mathrm R, p-1}^{(\lambda),\mathcal A}(t_0,\tau )\right)x^{p-1},\]
and it does so continuously in $t_0,\tau $.
\begin{proof}
Let us use the notation of the previous proof.
Consider the domain $t_0\leq \tau $.
Then
\begin{align*}
\lambda(U_2-1)+\lambda(\lambda-1)U_2\mathcal R^{(\lambda)}(U_3U_1U_2)
&=\lambda(U_2-1)+\lambda(\lambda-1)U_2U_3\mathcal R^{(\lambda)}(U_1U_2U_3)U_3^{-1}\\
&=\lambda(U_2-1)+\lambda(\lambda-1)U_1^{-1}\mathcal R^{(\lambda)}(U_1U_2 U_3)U_1 U_2
.
\end{align*}
In the latter two expressions $U_1U_2U_3=\Rexp(T\cdot \mathrm Z_{[0,1)})$, while
$U_1$, $U_2$, $U_3$ are exponential expressions, entire in $x\in\mathbb C$ if $T\rightsquigarrow x$ is substituted.
This shows equianalyticity to the resolvent, etc.
\end{proof}
\end{theorem}

For $t\in[0,1]$, let us define $\widetilde{\mathcal K}^{(\lambda),\mathcal A}_{\mathrm R, p-1}(t)$ such that
\begin{multline*}
\sum_{p=1}^\infty \widetilde{\mathcal K}^{(\lambda),\mathcal A}_{\mathrm R, p-1}(t)\cdot T^{p-1}=
(\Rexp(T\cdot\mathrm Z_{[0,t)}))+\\
+(\lambda-1)\cdot(\Rexp(T\cdot\mathrm Z_{[0,t)})){\mathcal R}^{(\lambda)}
\left( (\Rexp(T\cdot\mathrm Z_{[1-t,1)}))  (\Rexp(T\cdot\mathrm Z_{[0,t)})) \right)
.
\end{multline*}
\begin{lemma}
\plabel{lem:riol}
\[\widetilde{\mathcal K}^{(\lambda),\mathcal A}_{\mathrm R, p-1}(t)
=\mathrm{Rfl}\left(\widetilde{\mathcal K}^{(1-\lambda),\mathcal A}_{\mathrm R, p-1}(1-t)\right)
.
\]
\begin{proof}
This follows from \eqref{eq:resid1}.
\end{proof}
\end{lemma}


\subsection{The resolvent estimating kernels}
\plabel{ss:resest}
~\\

For $p-1\geq0$,  and $t_0,t_p\in[0,1]$, we  set
\begin{multline}
K_{p-1}^{(\lambda),\mathcal A}(t_0,t_p)=\biggl|{\mathcal K}_{\mathrm R, p-1}^{(\lambda),\mathcal A}(t_0,t_p) \biggr|_{\mathrm F\mathcal A}=\\=\biggl|\int_{\mathbf t_1=(t_1,\ldots,t_{p-1})\in[0,1]^k}
\lambda^{\asc(t_0,\mathbf t_1,t_p)}
(\lambda-1)^{\des(t_0,\mathbf t_1,t_p)}\mathrm Z^{1}_{[0,1]}(t_1)\ldots \mathrm Z^{1}_{[0,1]}(t_{p-1}) \biggr|_{\mathrm F\mathcal A}.
\plabel{eq:kernel}
\end{multline}
Then $K_{p-1}^{(\lambda),\mathcal A} $ is nonnegative, and a trivial estimate is $K_{p-1}^{(\lambda),\mathcal A} \leq 1$.
For $p-1\geq1$, the function $K_{p-1}^{(\lambda),\mathcal A}(t_0,t_p)$
is continuous.
We will naturally consider these $K_{p-1}^{(\lambda),\mathcal A} $ as nonnegative integral kernels.
(See Appendix \ref{sec:int} for their discussion in general.)
\begin{lemma}
\plabel{lem:criol}
For $k\geq2$,
\[\Theta_k^{(\lambda),\mathcal A}\leq\int_{(t,s)\in[0,1]^2} K_{k-2}^{(\lambda),\mathcal A} \,\mathrm ds \,\mathrm dt
\equiv\langle 1_{[0,1]},I_{K_{k-2}^{(\lambda),\mathcal A}} 1_{[0,1]}\rangle
. \]
\begin{proof}
This follows from \eqref{eq:rgen} and the submultiplicavity of the norm.
\end{proof}
\end{lemma}

\begin{lemma}
\plabel{lem:submult}
For $p-1,q-1\geq0$,
\[K_{p+q-1}^{(\lambda),\mathcal A}\leq K_{p-1}^{(\lambda),\mathcal A}*K_{q-1}^{(\lambda),\mathcal A}\]
holds. In other terms, the assignment $p\mapsto  K_{p-1}^{(\lambda),\mathcal A} $ is a submultiplicative family of nonnegative kernels.
\begin{proof}
This follows from \eqref{eq:rcomp} and the submultiplicativity of the norm.
\end{proof}
\end{lemma}

\begin{lemma}
\plabel{lem:kerneldiff}
(a) For a fixed p,
$K^{(\lambda),\mathrm h\mathcal A}_{p-1}(t_0,t_p)$ depends only on $\lambda$ and $t_p-t_0$.
Hence, the notation
\begin{equation}
K^{(\lambda), \mathcal A}_{p-1}(t_0,t_p)\equiv K^{(\lambda),\mathcal A}_{p-1}(t_p-t_0)
\plabel{eq:Kson}
\end{equation}
 is reasonable.

(b) Furthermore,
 \[K^{(\lambda),\mathcal A}_{p-1}(t)=K^{(1-\lambda),\mathcal A}_{p-1}(-t) .\]
\begin{proof}
(a) is immediate from Lemma \ref{lem:shift}. (b) follows from Lemma \ref{lem:ref}.
\end{proof}
\end{lemma}

Let
\[\widetilde K^{(\lambda),\mathcal A}_{p-1}(t)
=\biggl|\widetilde{\mathcal K}^{(\lambda),\mathcal A}_{\mathrm R, p-1}(t)\biggr|_{\mathrm F\mathcal A}.\]

\begin{lemma}
\plabel{lem:redkernel}
For $t\in[0,1]$,
\[K^{(\lambda),\mathcal A}_{p-1}(t) =\lambda\widetilde K^{(\lambda),\mathcal A}_{p-1}(t) ;\]
For $t\in[-1,0]$,
\[K^{(\lambda),\mathcal A}_{p-1}(t) =(1-\lambda)\widetilde K^{(\lambda),\mathcal A}_{p-1}(t+1) .\]
\begin{proof}
This follows from Lemma \ref{th:grandrev}.
\end{proof}
\end{lemma}

\begin{lemma}
\plabel{lem:redtrans}
\[
K^{(\lambda),\mathcal A}_{p-1}(t)=K^{(1-\lambda),\mathcal A}_{p-1}(1-t)
.
\]
\begin{proof}
This follows from Lemma \ref{lem:riol}.
\end{proof}
\end{lemma}

Instead of a the class $\mathcal A$, we can also apply this kernel
formalism to the general Banach algebraic setting (in notation: omitting $\mathcal A$).
Then, the situation is much simpler:

\begin{lemma}
\plabel{lem:preplain}

(a) The assignment $p \mapsto K^{(\lambda)}_{p-1}$ is  multiplicative:
\[ K^{(\lambda)}_{p-1}=\left(K^{(\lambda)}_{0} \right)^{*p}.\]

(b) For $k\geq2$,
\[\Theta_k =\int_{(t,s)\in[0,1]^2} K_{k-2}^{(\lambda) } \,\mathrm ds \,\mathrm dt
\equiv\langle 1_{[0,1]^2},I_{K_{k-2}^{(\lambda) }} 1_{[0,1]^2}\rangle
. \]

\begin{proof}
(a) and (b) are induced from the $\ell^1$ norm.
\end{proof}
\end{lemma}
Let us recall that
\[\Theta^{(\lambda)}(x)=\sum_{p=1}^\infty\Theta^{(\lambda)}_px^p=G(\lambda x,(1-\lambda)x),\]
where
\[G(u,v)=\dfrac{\frac{\sinh\frac{u-v}2}{\frac{u-v}2}}{\cosh\frac{u-v}2-\frac{u+v}2\frac{\sinh\frac{u-v}2}{\frac{u-v}2}}
=\frac{\mathrm e^u-\mathrm e^v}{u\mathrm e^v-v\mathrm e^u}.\]
We can also write down the generating function of the ``resolvent estimating'' kernels explicitly.
For the sake of simplicity, we give only the reduced kernel.
\begin{lemma}\plabel{lem:PlainResKernel} For $t\in[0,1]$
\[\widetilde{\Theta}^{(\lambda)}(x\,\pmb|\,t)\equiv
\sum_{p=1}^\infty \widetilde{\mathcal K}^{(\lambda),\mathcal A}_{\mathrm R, p-1}(t)\cdot x^{p-1}
=\widetilde{G}(\lambda x,(1-\lambda) x\,\pmb|\,t)
\]
where
\[\widetilde{G}(u,v\,\pmb|\,t)=\dfrac{
\mathrm e^{\frac{u-v}2\cdot(2t-1)}
}{\cosh\frac{u-v}2-\frac{u+v}2\frac{\sinh\frac{u-v}2}{\frac{u-v}2}}
=\frac{u-v}{u\mathrm e^v-v\mathrm e^u}\cdot\mathrm e^{tu+(1-t)v}.\]
\begin{proof} By considering the structure of the resolvent kernel for $t_0=0$, $\tau=t$, we find
\begin{align*}
\widetilde{\Theta}^{(\lambda)}&(x\,\pmb|\,t)=1+ \\
& +\left(1+\Theta^{(\lambda)}((1-t)x)(1-\lambda)\right)
\left(1-\Theta^{(\lambda)}(tx)\lambda\Theta^{(\lambda)}((1-t)x)(1-\lambda)\right)^{-1}
\Theta^{(\lambda)}(tx)\lambda
\\
&+ \left(1+\Theta^{(\lambda)}(tx)\lambda\right)
\left(1-\Theta^{(\lambda)}((1-t)x)(1-\lambda)\Theta^{(\lambda)}(tx)\lambda\right)^{-1}
\Theta^{(\lambda)}((1-t)x)(1-\lambda)
.
\end{align*}
This simplifies as indicated.
\end{proof}
\end{lemma}
\begin{commentx}
\begin{remark}
\plabel{rem:resexpid}
The expected identities
\[\widetilde{\Theta}^{(\lambda)}(T\,\pmb|\,0)=1+(1-\lambda){\Theta}^{(\lambda)}(T);
\qquad
\widetilde{\Theta}^{(\lambda)}(T\,\pmb|\,1)=1+\lambda{\Theta}^{(\lambda)}(T);\]
and
\begin{multline*}
{\Theta}^{(\lambda)}(T)
=T+T^2
\cdot\Bigl(
\lambda\int_{0\leq t_0\leq \tau\leq1} \widetilde{\Theta}^{(\lambda)}(T\,\pmb|\,\tau-t_0)\,\mathrm dt_0\,\mathrm d\tau
+
(1-\lambda)\int_{0\leq \tau\leq t_0\leq1} \widetilde{\Theta}^{(\lambda)}(T\,\pmb|\,\tau-t_0+1)\,\mathrm dt_0\,\mathrm d\tau
\Bigr)
\end{multline*}
check out.
\qedremark
\end{remark}
\end{commentx}

\snewpage
\subsection{The spectral properties of the kernels}
\plabel{ss:specker}
~\\

Let
\[w_{p-1}^{(\lambda),\mathcal A}=\mathrm r\left( I_{ K_{p-1}^{(\lambda),\mathcal A} } \right),\]
i.~e.~the spectral radius of the integral operator associated to $ K_{p-1}^{(\lambda),\mathcal A}$.
\begin{theorem}
\plabel{th:rrad}
\[w^{(\lambda),\mathcal A}=\inf_p \sqrt[p]{w_{p-1}^{(\lambda),\mathcal A}} = \lim_p \sqrt[p]{w_{p-1}^{(\lambda),\mathcal A}}.\]
\begin{proof}
By submultiplicativity, the infimum  and the limit are equal
(cf. \eqref{eq:bunch}, but `$n\mapsto K_n$' is replaced by `$k-1\mapsto K_{k-2}^{(\lambda),\mathcal A}$').
By Lemma \ref{lem:criol},
\begin{multline*}
\limsup_k \sqrt[k]{\Theta_k^{\mathcal A}}
=\limsup_k \sqrt[k-1]{\Theta_k^{\mathcal A}}
\leq\limsup_k \sqrt[k-1]{\langle 1_{[0,1]},I_{K_{k-2}^{(\lambda),\mathcal A}} 1_{[0,1]}\rangle}
\leq\\
\leq \limsup_k \sqrt[k-1]{\left\|I_{K_{k-2}^{(\lambda),\mathcal A}} \right\|_{L^2}}
=\inf \sqrt[k-1]{ w_{k-1}^{(\lambda),\mathcal A} };
\end{multline*}
leading to $w^{(\lambda),\mathcal A}\leq\inf_p \sqrt[p]{w_{p-1}^{(\lambda),\mathcal A}} = \lim_p \sqrt[p]{w_{p-1}^{(\lambda),\mathcal A}}$.
Let $0<\varepsilon< C^{(\lambda),\mathcal A}_\infty$.
By Theorem \ref{th:ranal}, we can apply Cauchy's theorem in order to obtain uniform bounds
\[K_{p-1}^{(\lambda),\mathcal A}(t_0,t_p)\leq\frac{C_{\varepsilon,\lambda}}{  (\mathrm C^{(\lambda),\mathcal A}_\infty-\varepsilon)^{p}}\]
(uniformly in $t_0,\tau $).
As the integral operator acts on the unit interval,
we can majorize the norm by the maximum norm, leading to
$w^{(\lambda),\mathcal A}\geq\inf_p \sqrt[p]{w_{p-1}^{(\lambda),\mathcal A}} = \lim_p \sqrt[p]{w_{p-1}^{(\lambda),\mathcal A}}$.
\end{proof}
\end{theorem}
\begin{commentx}
\begin{remark}
\plabel{rem:hn}
Let $\mathcal A=\mathcal{UMQ}_q/\mathbb K$.
Then, by Lemma \ref{lem:antiest2}, one can directly see that
$\sqrt[k]{\Theta_k^{(\lambda)\mathcal A}} $
and
$\sqrt[k]{\langle 1_{[0,1]^2},I_{K_{k-2}^{(\lambda),\mathcal A}} 1_{[0,1]^2}\rangle}$
are equiconvergent.
\qedremark
\end{remark}
\end{commentx}
We know that for a fixed $p$ the function $\lambda\in[0,1]\mapsto w_{p-1}^{(\lambda),\mathcal A}$ is continuous
(even for $p-1=0$).
As such, it takes its maximum, let
\[w^{(\log),\mathcal A}_{p-1}=\max_{\lambda\in[0,1]}w^{(\lambda),\mathcal A}_{p-1}.\]
\begin{theorem}
\plabel{th:unif}
As $p\rightarrow +\infty$, the functions $\lambda\in[0,1]\mapsto \sqrt[p]{w_{p-1}^{(\lambda),\mathcal A}}$
converge to the function $\lambda\in[0,1]\mapsto w^{(\lambda),\mathcal A}$ uniformly.
\begin{proof}
Let $\varepsilon>0$ be arbitrary.
By standard compactness arguments and monotonicity with respect to $p\mapsto K_{p-1}^{(\lambda),\mathcal A} $,
 there is a natural number $p_0>0$, such that for any $p\geq p_0$
 \[\sqrt[p]{   \left\| I_{ K_{p-1}^{(\lambda),\mathcal A} } \right\|_{L^2}}\leq w^{(\lambda),\mathcal A}+\varepsilon \]
holds for the associated integral operators, uniformly in $\lambda\in[0,1]$.
(One can pass from $p$ to $p!$ to provide strict monotonicity in order to arrive
to a threshold with $\leq w^{(\lambda),\mathcal A}+\varepsilon/2$.
Then one can use the trivial estimate and submultiplicativity to extend to large  general values.)
Then, for   $p\geq p_0$,
\[\sqrt[p]{w_{p-1}^{(\lambda),\mathcal A}}\leq w^{(\log),\mathcal A}+\varepsilon \]
holds uniformly in $\lambda\in[0,1]$.
\end{proof}

\end{theorem}

\begin{theorem}For $p-1\geq0$,
\plabel{th:runilog}
\[w^{(\log),\mathcal A}
=\inf_p \sqrt[p]{w^{(\log),\mathcal A}_{p-1}}
=\lim_p \sqrt[p]{w^{(\log),\mathcal A}_{p-1}}.\]
\begin{proof}
This follows from Theorem \ref{th:unif} immediately.
\end{proof}
\end{theorem}

\begin{lemma}
\plabel{lem:plain}
In the plain Banach algebraic case,
 \[w^{(\lambda) }
= \sqrt[p]{w^{(\lambda) }_{p-1}}=\mathrm r (I_{K^{(\lambda)}_0}).\]
\begin{proof}
It follows from Lemma \ref{lem:preplain}.
\end{proof}
\end{lemma}
\begin{remark}
\plabel{rem:plain}
The dominant eigenvector of $I_{K_n^{(\lambda)}}$ (up to scalar multiples, for $\lambda\in(0,1)$)
is given by $t\in[0,1]\mapsto \left(\frac{1-\lambda}{\lambda}\right)^t$.
\qedremark
\end{remark}
\begin{lemma}
\plabel{lem:pst}
 For $\lambda\in[0,1]$, $p-1\geq0$,
\[\sqrt[p]{w^{(\lambda),\mathcal A}_{p-1}}\leq w^{(\lambda)}.\]
\begin{proof}
This follows from the monotonicity property
$\sqrt[p]{w^{(\lambda),\mathcal A}_{p-1}}\leq \sqrt[p]{w^{(\lambda)}_{p-1}}$.
\end{proof}
\end{lemma}
\begin{theorem}
\plabel{th:pst}
\[w^{(\lambda),\boldsymbol\varepsilon}\leq w^{(\lambda),\mathcal A}
\leq \sqrt[p]{w^{(\lambda),\mathcal A}_{p-1}}\leq w^{(\lambda)};\]
and
\[\frac1\pi, w^{\mathcal A}\leq w^{(\log),\mathcal A}\leq \sqrt[p]{w^{(\log),\mathcal A}_{p-1}}\leq \frac12.\]

Or, taking the general notation ${\mathrm C}_\infty^\otimes=1/w^\otimes$,
\[{\mathrm C}_\infty^{(\lambda)}\leq \sqrt[p]{{\mathrm C}_{\infty,{p-1}}^{(\lambda),\mathcal A}}\leq
 {\mathrm C}_\infty^{(\lambda),\mathcal A}\leq{\mathrm C}_\infty^{(\lambda),\boldsymbol\varepsilon};\]
and
\[2\leq \sqrt[p]{{\mathrm C}_{\infty,p-1}^{(\log),\mathcal A}}
\leq {\mathrm C}_\infty^{(\log),\mathcal A}\leq {\mathrm C}_\infty^{\mathcal A},\pi.\]
\begin{proof}
This is just some of the previous information put together.
\end{proof}
\end{theorem}

Our general strategy is that if we obtain an upper estimate
$w^{(\log),\mathcal A}\leq C$, then it yields a lower estimate
$\frac1C\leq{\mathrm C}_\infty^{(\log),\mathcal A}\leq{\mathrm C}_\infty^{\mathcal A}$.

\snewpage

\subsection{Some crude estimates}
\plabel{ss:crude}
~\\

Although precise numerical estimates for  $w^{(\log),\mathcal A}_{p-1}=\mathrm r\left(I_{ K^{(\lambda),\mathcal A}_{p-1} }\right)$
are quite doable (cf. monotonicity, Theorem \ref{th:average},  Theorem \ref{th:Hopf}),
certain estimates may be useful in practice:

\begin{lemma}
\plabel{th:Compar}
Let
\[S_{p-1}^{\mathcal A}(\lambda)=
 \mathrm{ess\,sup}\,\frac{K_{p-1}^{(\lambda),\mathcal A}}{K_{p-1}^{(\lambda)}}=
 \mathrm{ess\,sup}\,\frac{\widetilde K_{p-1}^{(\lambda),\mathcal A}}{\widetilde K_{p-1}^{(\lambda)}}
\]
(where $\frac00=0$).
Then
\[w^{(\lambda),\mathcal A}\leq\sqrt[p]{w^{(\lambda),\mathcal A}_{p-1}} \leq w^{(\lambda)}\sqrt[p]{S^{\mathcal A}_{p-1}{(\lambda)}}.\]
In particular,
\[{\mathrm C}^{(\lambda),\mathcal A}_\infty\geq\frac1{w^{(\lambda)}\sqrt[p]{S^{\mathcal A}_{p-1}{(\lambda)}}}.\]
\begin{proof}
This is immediate from the monotonicity of the spectral radius.
\end{proof}
\end{lemma}

\begin{lemma}
\plabel{lem:sicompar}
\[w^{(\lambda),\mathcal A}\leq  \sqrt[p]{ \max(\lambda,1-\lambda) \int_{t=0}^{1}\widetilde K_{p-1}^{(\lambda)}(t)\,\mathrm dt}.\]
\begin{proof}
Using Lemma \ref{lem:redkernel}, this follows by estimating
$\lambda,1-\lambda\leq \max(\lambda,1-\lambda)$, and considering the reduced kernel as a convolution kernel.
\end{proof}
\end{lemma}
\begin{lemma}
\plabel{lem:ricompar}
\[w^{(\lambda),\mathcal A}\leq\sqrt[p]{ w^{(\lambda)}\max_{t\in[0,1]} \widetilde K_{p-1}^{(\lambda)}(t)}.\]
\begin{proof}
Using Lemma \ref{lem:redkernel}, this follows by estimating the reduced kernel trivially.
\end{proof}
\end{lemma}

Now, everywhere up this point in the section, `$\mathcal A$' can be replaced `$\mathrm h\mathcal A$'.
If we develop estimates only for $\mathrm h\mathcal A$, it is still useful for us, as
\[w^{(\lambda),\mathcal A}\leq w^{(\lambda),\mathrm h\mathcal A} \]
and
\[{\mathrm C}^{(\lambda),\mathcal A}_\infty\geq{\mathrm C}^{(\lambda),\mathrm h\mathcal A}_\infty,\]
etc., hold.\\

\snewpage

\subsection{The estimating kernels in the homogeneous case}
~\\

In the setting of `$\mathrm h\mathcal A$',
the kernels can be presented and their properties can be redeveloped in more discrete and explicit terms.
Let us take a closer look at $K_{p-1}^{(\lambda),\mathrm h\mathcal A}(t_0,t_p)$.

Assume that $t_0<t_p$.
In \eqref{eq:kernel}, the integrand is best to be decomposed according to the distribution of
$\{t_1,\ldots,t_{p-1}\}$ relative to $t_0,t_p$.
Here we imagine $a$ to be the number of indices smaller than $t_0$ and $t_p$;
$b$ to be the number of indices between $t_0$ and $t_p$;
$c$ to be the number of indices greater than $t_0$ and $t_p$.
For $a+b+c=p-1$, let
\[p_{a,b,c}(t_0,t_p)=\frac{(p-1)!}{a!b!c!}t_0^a(t_p-t_0)^b(1-t_p)^c;\]
and
\[\mu_{a,b,c}^{(\lambda)}(X_1,\ldots,X_{p-1})=\sum_{\sigma\in\Sigma_{p-1}}
\lambda^{\asc(a+\frac12,\sigma,p-\frac12-c)}(\lambda-1)^{\des(a+\frac12,\sigma,p-\frac12-c)}
 X_{\sigma(1)}\ldots X_{\sigma(p-1)};\]
and
\[\Theta_{a,b,c}^{(\lambda),\mathrm h\mathcal A}=\frac1{(p-1)!}
\left| \mu_{a,b,c}^{(\lambda)}(Y_1,\ldots,Y_{p-1})\right|_{\mathrm F\mathcal A}.\]
Then
\begin{equation}
K^{(\lambda),\mathrm h\mathcal A}_{p-1}(t_0,t_p)=\sum_{a+b+c=p-1}p_{a,b,c}(t_0,t_p)\Theta_{a,b,c}^{(\lambda),\mathrm h\mathcal A}.
\plabel{eq:Kdeco}
\end{equation}
Here $p_{a,b,c}(t_0,t_p)$ refers to the probability of the configuration, and $\Theta_{a,b,c}^{\mathrm h\mathcal A}$
is the contribution of the corresponding noncommutative term.

There is a similar analysis for $t_0>t_p$.
Let
\[\tilde \mu_{a,b,c}^{(\lambda)}(X_1,\ldots,X_{p-1})=\sum_{\sigma\in\Sigma_{p-1}}
\lambda^{\asc(p-\frac12-c,\sigma,a+\frac12)}(\lambda-1)^{\des(p-\frac12-c,\sigma,a+\frac12)}
 X_{\sigma(1)}\ldots X_{\sigma(p-1)};\]
and
\[\tilde\Theta_{a,b,c}^{(\lambda),\mathrm h\mathcal A}
=\frac1{(p-1)!}\left| \tilde\mu_{a,b,c}^{(\lambda)}(Y_1,\ldots,Y_{p-1})\right|_{\mathrm F\mathcal A}.\]
Then
\begin{equation}
K^{(\lambda),\mathrm h\mathcal A}_{p-1}(t_0,t_p)=\sum_{a+b+c=p-1}p_{a,b,c}(t_p,t_0)\tilde\Theta_{a,b,c}^{(\lambda),\mathrm h\mathcal A}.
\plabel{eq:KdecoAnt}
\end{equation}
By simple combinatorial principles,
\begin{equation}
\tilde \mu_{a,b,c}^{(\lambda)}(X_1,\ldots,X_{p-1})=-\mu_{c,b,a}^{(1-\lambda)}(-X_{p-1},\ldots,-X_{1})
.
\plabel{eq:door1}
\end{equation}
This implies
\begin{equation}
\tilde \Theta_{a,b,c}^{(\lambda),\mathrm h\mathcal A}=\Theta_{c,b,a}^{(1-\lambda),\mathrm h\mathcal A}.
\plabel{eq:qoor1}
\end{equation}
Also,
\[p_{a,b,c}(t_0,t_p)=p_{c,b,a}(1-t_p,1-t_0)\]
holds. Thus
\begin{equation}
K^{(\lambda),\mathrm h\mathcal A}_{p-1}(t_0,t_p)=\sum_{a+b+c=p-1}p_{c,b,a}(1-t_0,1-t_p)\Theta_{c,b,a}^{(1-\lambda),\mathrm h\mathcal A}.
\plabel{eq:KdecoAntis}
\end{equation}

Therefore,
\begin{equation}
K^{(\lambda),\mathrm h\mathcal A}_{p-1}(t_0,t_p)=K^{(1-\lambda),\mathrm h\mathcal A}_{p-1}(1-t_0 ,1-t_p)
\plabel{eq:Ksymm}
\end{equation}
holds generally.
Now, one can greatly simplify \eqref{eq:Kdeco} and \eqref{eq:KdecoAnt}/\eqref{eq:KdecoAntis}.
\begin{lemma}\plabel{lem:csere}

(a) For $a+b+c+1=p-1$,
\begin{equation}
\mu_{a+1,b,c}^{(\lambda)}(X_1,\ldots,X_{p-1})=\mu_{a,b,c+1}^{(\lambda)}(X_2,\ldots,X_{p-1},X_1).
\plabel{eq:door2}
\end{equation}

(b) In particular,
$\Theta_{a,b,c}^{(\lambda)}$ depends only on $\lambda$, $a+c$, and $b$.
\begin{proof}
(a) If we rename the lowest position to the highest position, then it also yields one descent and one ascent, while
the descent/ascent relations between other indices remain the same.
(b) This is an immediate corollary.
\end{proof}
\end{lemma}

We set
\[p_{a,b}(t)=\frac{(p-1)!}{a!b!}(1-t)^at^b.\]
Let us also define
\[\mu_{a,b}^{(\lambda)}(X_1,\ldots,X_{p-1})=\sum_{\sigma\in\Sigma_{p-1}}
\lambda^{\asc(a+\frac12,\sigma)}(\lambda-1)^{\des(a+\frac12,\sigma)}
 X_{\sigma(1)}\ldots X_{\sigma(p-1)}.\]
 This makes
\begin{equation}
\mu_{a,b,0}^{(\lambda)}(X_1,\ldots,X_{p-1}) =\lambda\cdot\mu_{a,b}^{(\lambda)}(X_1,\ldots,X_{p-1}) .
\plabel{eq:door3}
\end{equation}
Let
\[
\Theta_{a,b}^{(\lambda),\mathrm h\mathcal A}=\frac1{(p-1)!}\left| \mu_{a,b}^{(\lambda)}(Y_1,\ldots,Y_{p-1})\right|_{\mathrm F\mathcal A}
.
\plabel{eq:ThetahADef}
\]
Then, by \eqref{eq:door2} and \eqref{eq:door3}
\[\Theta_{a,b,c}^{(\lambda),\mathrm h\mathcal A}=\lambda \Theta_{c+a,b}^{(\lambda),\mathrm h\mathcal A};\]
moreover, by \eqref{eq:qoor1},
\[
\tilde \Theta_{a,b,c}^{(\lambda),\mathrm h\mathcal A}=(1-\lambda) \Theta_{c+a,b}^{(1-\lambda),\mathrm h\mathcal A}
.
\]
(Here, and in similar situations, the cases $\lambda=0,1$ can be reached as limits.)
\begin{theorem}
\plabel{th:kernelexp}
For $t_0\leq t_p$,
\begin{equation}
K^{(\lambda),\mathrm h\mathcal A}_{p-1}(t_0,t_p)=\lambda\cdot\sum_{a+b=p-1}p_{a,b}(t_p-t_0)\Theta_{a,b}^{(\lambda),\mathrm h\mathcal A}.
\plabel{eq:Kcoup}
\end{equation}
For $t_0 \geq t_p$,
\begin{equation}
K^{(\lambda),\mathrm h\mathcal A}_{p-1}(t_0,t_p)=(1-\lambda)\cdot\sum_{a+b=p-1}p_{a,b}(t_0-t_p)\Theta_{a,b}^{(1-\lambda),\mathrm h\mathcal A}.
\plabel{eq:Kdecoup}
\end{equation}
\begin{proof}
This is \eqref{eq:Kdeco} and \eqref{eq:KdecoAntis}  combined with Lemma \ref{lem:csere} and the binomial theorem.
\end{proof}
\end{theorem}

\begin{cor}
\plabel{cor:kerneldiff}
(a) For a fixed p,
$K^{(\lambda),\mathrm h\mathcal A}_{p-1}(t_0,t_p)$ depends only on $\lambda$ and $t_p-t_0$.
Hence, the notation
\begin{equation}
K^{(\lambda),\mathrm h\mathcal A}_{p-1}(t_0,t_p)\equiv K^{(\lambda),\mathrm h\mathcal A}_{p-1}(t_p-t_0)
\plabel{eq:Ksonc}
\end{equation}
 is reasonable.

(b) Furthermore,
 \[K^{(\lambda),\mathrm h\mathcal A}_{p-1}(t)=K^{(1-\lambda),\mathrm h\mathcal A}_{p-1}(-t) .\]
\begin{proof}
(a) is immediate from the previous theorem; (b) follows from \eqref{eq:Ksymm}.
\end{proof}
\end{cor}

\begin{commentx}
Let
\[ \mu_{b,c}^{\prime(\lambda)}(X_1,\ldots,X_{p-1})=\sum_{\sigma\in\Sigma_{p-1}}
\lambda^{\asc(\sigma,b+\frac12 )}(\lambda-1)^{\des(\sigma, b+\frac12)}
 X_{\sigma(1)}\ldots X_{\sigma(p-1)}.\]
This makes
\[\mu_{0,b,c}^{(\lambda)}(X_1,\ldots,X_{p-1}) =\lambda\cdot\mu_{b,c}^{\prime(\lambda)}(X_1,\ldots,X_{p-1}) .\]
Then, by Lemma \ref{lem:csere},
\[\mu_{a,b}^{(\lambda)}(X_1,\ldots,X_{p-1})
=\mu_{b,a}^{\prime(\lambda)}(X_{a+1},\ldots,X_{p-1},X_1,\ldots ,X_a).
\plabel{eq:door33}\]
\end{commentx}
\begin{lemma}\plabel{lem:ThetaLT}
(a)
\[\mu_{a,b}^{(\lambda)}(X_1,\ldots,X_{p-1})
=\mu_{b,a}^{(1-\lambda)}(-X_{p-1},\ldots,-X_{1})
.
\plabel{eq:door4}\]

(b) Consequently,
\[\Theta_{a,b}^{(1-\lambda),\mathrm h\mathcal A}=\Theta_{b,a}^{(\lambda),\mathrm h\mathcal A}.\]
\begin{proof}
(a) follows from the previous identites \eqref{eq:door1}, \eqref{eq:door2}, \eqref{eq:door3}.
(b) follows from (a).
\end{proof}
\end{lemma}
For $t\in[0,1]$, we set the reduced kernel by
\[\widetilde{K}_{p-1}^{(\lambda),\mathrm h\mathcal A}(t)=\sum_{a+b=p-1}p_{a,b}(t)\Theta_{a,b}^{(\lambda),\mathrm h\mathcal A}.\]
\begin{theorem}
\plabel{th:kernelred}
\begin{equation}
K_{p-1}^{(\lambda),\mathrm h\mathcal A}(t)=
\begin{cases}
\lambda\widetilde K_{p-1}^{(\lambda),\mathrm h\mathcal A}(t) &\text{if}\quad t\in[0,1],\\
(1-\lambda)  \widetilde K_{p-1}^{(\lambda),\mathrm h\mathcal A}(t+1)&\text{if}\quad t\in[-1,0].
\end{cases}
\plabel{eq:Kred}
\end{equation}
\begin{proof}
It is easy to see that
\[p_{a,b}(1-t)=p_{b,a}(t).\]
By this and Lemma \ref{lem:ThetaLT}.(b), we obtain that for $t_0>t_p$,
\begin{equation}
K^{(\lambda),\mathrm h\mathcal A}_{p-1}(t_0,t_p)=(1-\lambda)\cdot\sum_{a+b=p-1}p_{b,a}(1+t_p-t_0)\Theta_{b,a}^{(\lambda),\mathrm h\mathcal A};
\plabel{eq:Kdecoupvar}
\end{equation}
rewriting the kernel (in the second case).
\end{proof}
\end{theorem}

\begin{remark}
\plabel{rem:kernelnull}

(a) We know that \eqref{eq:Ksonc} is continuous for $p-1\geq 1$.
Moreover,
\[K^{(\lambda),\mathrm h\mathcal A}_{p-1}(0)=\lambda(1-\lambda)\cdot
\frac1{(p-1)!}\left|\mu^{(\lambda)}_{p-1}(Y_1,\ldots,Y_{p-1})\right|_{\mathrm F\mathcal A}
=\lambda(1-\lambda)\cdot\Theta^{(\lambda),\mathrm h\mathcal A}_{p-1}
.
\]
According to this, for $p-1\geq1$,
\begin{equation}
\widetilde{K}^{(\lambda),\mathrm h\mathcal A}_{p-1}(0)=(1-\lambda)\Theta^{(\lambda),\mathrm h\mathcal A}_{p-1}
,
\plabel{eq:dis1}
\end{equation}
and
\begin{equation}
\widetilde{K}^{(\lambda),\mathrm h\mathcal A}_{p-1}(1)=\lambda\Theta^{(\lambda),\mathrm h\mathcal A}_{p-1}
.
\plabel{eq:dis2}
\end{equation}

(b) Strictly speaking,
$\widetilde{K}_{p-1}^{(\lambda),\mathrm h\mathcal A}(t)$, defined only for $t\in[0,1]$
is not a kernel; but we can obtain a convolution kernel by setting
$\widetilde{K}_{p-1}^{(\lambda),\mathrm h\mathcal A}(t)=\widetilde{K}_{p-1}^{(\lambda),\mathrm h\mathcal A}(t+1)$
for $t\in[-1,0]$.
This, however, introduces an ambiguity, or, rather, discontinuity for $t=0$ (if $\lambda\neq\frac12$),
as \eqref{eq:dis1} and \eqref{eq:dis2} show.
Such an ambiguity is otherwise harmless.
\qedremark
\end{remark}
\snewpage

\section{Some explicit estimates for the cumulative radius of the Magnus expansion}
\plabel{sec:resConcrete}
As a demonstration of our methods, here we apply the techniques of the previous sections for
$\mathcal A_q=\mathcal{UMQ}_q/\mathbb K
=\mathrm h\mathcal{UMQ}_q/\mathbb K$, using norm gains from degree $4$.
(This automatically provides lower estimates to the cases $\mathcal{UMD}_q/\mathbb K$ or $\mathrm h\mathcal{UMD}_q/\mathbb K$.)
In effect, we consider our weakest practical uniform convexity condition, using it up only in the smallest nontrivial degree.
This limited setting, however, has the advantage that we can provide exact values for some terms instead of relying on just upper estimates.
We should keep in mind that in our case $\mathcal A_q=\mathrm h\mathcal A_q$.

\subsection{The delay method}\plabel{ss:exdelay}~\\

For pedagogical reasons, we will start with the case of the Cayley transform.

\begin{lemma}
\plabel{lem:prekernel4}
If $\mathcal A_q=\mathcal{UMQ}_q/\mathbb K$, then
\[\Theta_{4}^{(1/2),\mathcal A_q}=
\frac18\left(\frac23+\frac13\cdot2^{-\frac1q}\right)
<
\Theta_{4}^{(1/2)}=
\frac18.\]

\begin{proof}
One finds
\begin{align}
\mu^{(1/2)}( Y_1,Y_2,Y_3,Y_4)=\frac1{8}\Bigl(&
+{Y}_{1234}
- {Y}_{1243}
- {Y}_{2134}
+{Y}_{2143}
\plabel{eq:mumid}
\\\notag
&- {Y}_{1324}
-  {Y}_{1342}
-  {Y}_{3124}
+{Y}_{3142}
\\\notag
&- {Y}_{1423}
+{Y}_{1432}
- {Y}_{4123}
+{Y}_{4132}
\\\notag
&- {Y}_{2314}
- {Y}_{2341}
+{Y}_{3214}
+{Y}_{3241}
\\\notag
&- {Y}_{2413}
+{Y}_{2431}
+{Y}_{4213}
+{Y}_{4231}
\\\notag
&- {Y}_{3412}
+{Y}_{3421}
+{Y}_{4312}
-{Y}_{4321}
\Bigr),
\end{align}
where we have used the notation $Y_{ijkl}=Y_iY_jY_kY_l$.
When we take $|\cdot|_{\mathrm F\mathrm h\mathcal A_q}$,
the lines in the RHS of \eqref{eq:mumid} separate in terms of the
linear programming problem (where the generators are quasi-monomially induced).
We can apply \eqref{eq:UMQ} in the second and third lines optimally, and with no use in the other lines.
(Cf. the more detailed explanation in the proof of Lemma \ref{lem:kernel4}.)
Thus, we find
\[\frac1{4!}\left|\mu^{(1/2)}_4( Y_1,Y_2,Y_3,Y_4) \right|_{\mathrm F\mathrm h\mathcal A_q}=
\frac{ 2+ 2^{-\frac1q} }{24}
<
\frac1{4!}\left|\mu^{(1/2)}_4( Y_1,Y_2,Y_3,Y_4) \right|_{\ell^1}=
\frac18.\qedhere\]
\end{proof}
\end{lemma}
\begin{theorem}
\plabel{th:magimprove}
If $\mathcal A_q=\mathcal{UMQ}_q/\mathbb K$, then regarding
the convergence radius $\mathrm C^{\mathcal A_q}_\infty$ of $\Theta^{\mathcal A_q}(x)$,
\[\mathrm C^{\mathcal A_q}_\infty\geq\mathrm C^{(\log),\mathcal A_q}_\infty>2.\]
\begin{proof}
This follows from applying Corollary \ref{cor:dec} to the previous Lemma \ref{lem:prekernel4}.
\end{proof}
\end{theorem}
\snewpage
Let us now consider some more quantitative consequences.
Let $\widehat{\Theta}^{(1/2),q}(x)$ be the solution of the (formal) IVP
\[\frac{\mathrm d\widehat{\Theta}^{(1/2)}(x)}{\mathrm dx}
=1+\widehat{\Theta}^{(1/2),q}(x)+\frac14\widehat{\Theta}^{(1/2),q}(x)^2-4x^3\frac{1-2^{-\frac1q}}{24},\]
\[\widehat{\Theta}^{(1/2),q}(0)=0.\]
Then $\Theta ^{(1/2),\mathcal A_q}(x)\stackrel{\forall x}\leq \widehat{\Theta}^{(1/2),q}(x)$.
Let $\widehat{\mathrm C}_\infty^{(1/2),q}$ be convergence radius of $\widehat{\Theta}^{(1/2),q}(x)$.
 Then, of course, $\mathrm C_\infty^{(1/2),\mathcal A_q} \geq\widehat{\mathrm C}_\infty^{(1/2),q}$.
The IVP above, which is of Riccati type, can be solved explicitly in terms Bessel functions.
We refrain from working this out here, we merely note that the convergence radius can be determined with
 arbitrary precision for any $q\in[1,+\infty)$.
In particular, we find that
\[\widehat{\mathrm C}_\infty^{(1/2),2}=2.0133601\ldots \]
and
\[\widehat{\mathrm C}_\infty^{(1/2),1}=2.0232461\ldots \]
hold.
In order to obtain a not very technical estimate for any $q\in[1,+\infty)$, let us make a very crude delay estimate in
\begin{lemma}
If $\mathcal A_q=\mathcal{UMQ}_q/\mathbb K$, then
\[\widehat{\mathrm C}_\infty^{(1/2),q}>2+\frac{1-2^{-\frac1q}}{95+2^{-\frac1q} }.\]
\begin{proof}
Integrating on $x\in[0,1]$, we find
\begin{multline*}
\widehat{\Theta}^{(1/2)}(1)=\int_{x=0}^1 1+\widehat{\Theta}^{(1/2),q}(x)+\frac14\widehat{\Theta}^{(1/2),q}(x)^2-4x^3\frac{1-2^{-\frac1q}}{24}\,\mathrm dx
\\
<\int_{x=0}^1 1+ {\Theta}^{(1/2) }(x)+\frac14 {\Theta}^{(1/2) }(x)^2-4x^3\frac{1-2^{-\frac1q}}{24}\,\mathrm dx=
{\Theta}^{(1/2) }(1)- \frac{1-2^{-\frac1q}}{24}.
\end{multline*}
This means that by $x=1$, the time delay of $\widetilde{\Theta}^{(1/2),q }(x)$ compared to ${\Theta}^{(1/2) }(x)$
is more than
\[1-\left({\Theta}^{(1/2) }\right)^{-1}  \left({\Theta}^{(1/2) }(1)- \frac{1-2^{-\frac1q}}{24}\right)=
\frac{1-2^{-\frac1q}}{95+2^{-\frac1q} }. \]
Adding this to the convergence radius $2$ of $\Theta^{(1/2)}(x)$, we obtain the statement.
\end{proof}
\end{lemma}
The  lemma above yields $\widehat{\mathrm C}_\infty^{(1/2),2}>2.0030603\ldots$,
and $\widehat{\mathrm C}_\infty^{(1/2),1}> 2.0052356\ldots$, which are not very sharp.

\begin{remark}
Using the Bessel functions, one can obtain, for example,
\[\widehat{\mathrm C}_\infty^{(1/2),q}>2+\frac{1-2^{-\frac1q}}{22+2^{-\frac1q} },\]
yielding $\widehat{\mathrm C}_\infty^{(1/2),2}> 2.0128987\ldots$,
and $\widehat{\mathrm C}_\infty^{(1/2),1}>  2.0222222\ldots$, which are closer.
\qedremark
\end{remark}
\snewpage

The general case, fortunately, is not much complicated:
\begin{lemma}
\plabel{lem:prekernel44}
If $\mathcal A_q=\mathcal{UMQ}_q/\mathbb K$, $\lambda\in[0,1]$, then
\[\Theta_{4}^{(\lambda),\mathcal A_q}=
\frac{1+8\lambda(1-\lambda)-8\lambda(1-\lambda)\min(\lambda,1-\lambda)(1-2^{-\frac1q})}{24}
<
\Theta_{4}^{(\lambda)}=
\frac{1+8\lambda(1-\lambda)}{24}.\]
Moreover,
\begin{equation}
\Theta_{4}^{(\lambda),\mathcal A_q}\leq\Theta_{4}^{(1/2),\mathcal A_q}.\plabel{eq:tome}
\end{equation}
\begin{proof}
This is similar to the proof of Lemma \ref{lem:prekernel4}.
The additional inequality \eqref{eq:tome} is then an elementary calculation.
\end{proof}
\end{lemma}
\begin{theorem}
If $\mathcal A_q=\mathcal{UMQ}_q/\mathbb K$, then
\[\mathrm C^{(\log),\mathcal A_q}_\infty\geq \widehat{\mathrm C}_\infty^{(1/2),q} .\]
\begin{proof}
Let $\widehat{\Theta}^{(\lambda),q}(x)$ be the solution of the (formal) IVP
\[\frac{\mathrm d\widehat{\Theta}^{(\lambda),q}(x)}{\mathrm dx}
=1+\widehat{\Theta}^{(\lambda),q}(x)+\lambda(1-\lambda)\widehat{\Theta}^{(\lambda),q}(x)^2-4x^3 \frac{8\lambda(1-\lambda)\min(\lambda,1-\lambda)(1-2^{-\frac1q})}{24},\]
\[\widehat{\Theta}^{(\lambda),q}(0)=0.\]
Then $\Theta ^{(\lambda),\mathcal A_q}(x)\stackrel{\forall x}\leq \widehat{\Theta}^{(\lambda),q}(x)$.
Let $\widehat{\mathrm C}_\infty^{(\lambda),q}$ be convergence radius of $\widehat{\Theta}^{(\lambda),q}(x)$.
However, by induction it is easy to prove that
$ \widehat{\Theta}^{(\lambda),q}(x)\stackrel{\forall x}\leq \widehat{\Theta}^{(1/2),q}(x)$
(only the coefficient of $x^4$ needs work, but this is just \eqref{eq:tome}).
This implies that ${\mathrm C}_\infty^{(\lambda),\mathcal A_q}\geq \widehat{\mathrm C}_\infty^{(\lambda),q} \geq\widehat{\mathrm C}_\infty^{(1/2),q}$, leading to the conclusion.
\end{proof}
\end{theorem}
This concludes a demonstration of the delay method.
In general, estimates obtained from the Eulerian  delay method are both cumbersome and weak.\\
\snewpage
\subsection{The chronological decomposition method}~\\

Here we will use only the plain method, which is theoretically weak but technically relatively unassuming.
We set
\[U_0^{(\lambda),q}(T)=\Theta^{(\lambda)}(T)-\frac{8\lambda(1-\lambda)\min(\lambda,1-\lambda)(1-2^{-\frac1q})}{24}T^4.\]
By Lemma \ref{lem:prekernel44}, $\Theta^{(\lambda),\mathcal A_q}(T)\stackrel{\forall T}\leq U_0^{(\lambda),q}(T)$.
Let us set up the recursion by
\[U_{k+1}^{(\lambda),q}(T)=\frac{2U_{k}^{(\lambda),q}(T/2) + U_{k}^{(\lambda),q}(T/2)^2}{1-\lambda(1-\lambda)U_{k}^{(\lambda),q}(T/2)^2 }
-\frac{7\lambda(1-\lambda)\min(\lambda,1-\lambda)(1-2^{-\frac1q})}{24}T^4.
\]
This is stationary mod $O(T^5)$, thus the correction term is valid.
Then it is easy to see that $\Theta^{(\lambda),\mathcal A_q}(T)\stackrel{\forall T}\leq U_k^{(\lambda),q}(T)$.
Actually, the $U_k^{(\lambda),q}(T)$ are monotone decreasing.
Let $\overline{\mathrm C }_\infty^{(\lambda),q}= \sup_k \mathrm r\left(U_k^{(\lambda),q}(T)\right)$.
Then $\mathrm C_\infty^{(\lambda),\mathcal A_q}\geq\overline{\mathrm C }_\infty^{(\lambda),q}$.

Fortunately, one can easily see by induction that
\[U_{k+1}^{(\lambda),q}(T)\stackrel{\forall T}\leq U_{k+1}^{( 1/2),\lambda}(T).\]
(Indeed, this is nontrivial only in the coefficient of $T^4$, where it is just \eqref{eq:tome}).
In particular, $\overline{\mathrm C }_\infty^{(\lambda),q}\geq\overline{\mathrm C }_\infty^{(1/2),q}$.
Therefore, it is sufficient to estimate the convergence radii for $\lambda=1/2$.
In this case
 \[U_0^{(1/2),q}(T)=\frac  T{1-\frac12 T}-\frac{ 1-2^{-\frac1q} }{24}T^4,\]
 and
\[U_{k+1}^{(1/2),q}(T)=\frac{2U_{k}^{(1/2),q}(T/2)}{1-\frac12U_{k}^{(1/2),q}(T/2) }
-\frac78\cdot \frac{ 1-2^{-\frac1q} }{24}T^4 .
\]
Let us first consider some concrete values.
After a couple iterations we see that
\[\overline{\mathrm C }_\infty^{(1/2),2}>2.00722428\]
and
\[\overline{\mathrm C }_\infty^{(1/2),1}>2.01243882\]
hold.
(Actually, these are approximative values here as the convergence radii are convergent.)
For a general estimate we will be content to use a single iteration step:
\begin{theorem}
\[ \mathrm C _\infty^{(\log),\mathcal A_q}\geq\overline{\mathrm C }_\infty^{(1/2),q}>
 \mathrm r\left(U_1^{(1/2),q}(T)\right)>2+\frac{1-2^{-\frac1q}}{47+2^{-\frac1q} }\]
\begin{proof}
The latter inequality is a discussion in elementary analysis.
\end{proof}
\end{theorem}
(This yields $\overline{\mathrm C }_\infty^{(1/2),2}> 2.00613940 $
and
$\overline{\mathrm C }_\infty^{(1/2),1}> 2.01052631$ .)
Here the plain chronological decomposition method was even weaker than the delay method,
but it was better than our crude estimate with the delay method.
\snewpage
\subsection{The kernel method}\plabel{ss:exkernel}
~\\

Now, we will compute $\Theta_{a,b}^{(\lambda),\mathcal A_q}$
for $a+b=p-1=4$ , $\mathcal A_q=\mathcal{UMQ}_q/\mathbb K$,  $\lambda\in[0,1]$.
By Lemma \ref{lem:ThetaLT}, it is sufficient to compute $\Theta_{a,p-1-a}^{(\lambda),\mathrm h\mathcal A_q}$ only for $0\leq a\leq\lfloor\frac{p-1}2\rfloor$.

\begin{lemma}
\plabel{lem:kernel4}
 For $\mathcal A_q=\mathcal{UMQ}_q/\mathbb K$,  $\lambda\in[0,1]$,
\[\Theta_{0,4}^{(\lambda),\mathcal A_q}
=\frac1{4!}\left(-8\,{\lambda}^{3}+8\,{\lambda}^{2}+\lambda
-(1-2^{-\frac1q})\cdot 8\lambda^2(1-\lambda)\min(\lambda,1-\lambda)
\right)
;\]
\[\Theta_{1,3}^{(\lambda),\mathcal A_q}
=\frac1{4!}\left(4\,{\lambda}^{4}-14\,{\lambda}^{3}+8\,{\lambda}^{2}+2\,\lambda
-(1-2^{-\frac1q})\cdot 8\lambda^2(1-\lambda)\min(\lambda,1-\lambda)\right);\]
\[\Theta_{2,2}^{(\lambda),\mathcal A_q}
=\frac1{4!}\left(8\,{\lambda}^{4}-16\,{\lambda}^{3}+4\,{\lambda}^{2}+4\,\lambda
-(1-2^{-\frac1q})\cdot 4\lambda(1-\lambda)\min(\lambda,1-\lambda)\right).\]
\begin{proof}
Let us consider $\Theta_{0,4}^{(\lambda),\mathcal A_q}$.
Here
\begin{equation}
\mu^{(\lambda)}_{0,4}( Y_1,Y_2,Y_3,Y_4)=
\plabel{eq:rusen}
\end{equation}
\begin{equation}
{\lambda}^{4}{Y}_{1234}
- {\lambda}^{3}\left( 1-\lambda \right){Y}_{1243}
- {\lambda}^{3}\left( 1-\lambda \right){Y}_{2134}
+{\lambda}^{2}\left( 1-\lambda \right) ^{2}{Y}_{2143}
\tag{\ref{eq:rusen}a}\plabel{eq:rusena}
\end{equation}
\begin{equation}
- {\lambda}^{3}\left( 1-\lambda \right) {Y}_{1324}
- {\lambda}^{3}\left( 1-\lambda \right) {Y}_{1342}
-{\lambda}^{3} \left( 1-\lambda \right) {Y}_{3124}
+{\lambda}^{2} \left( 1-\lambda \right) ^{2}{Y}_{3142}
\tag{\ref{eq:rusen}b}\plabel{eq:rusenb}
\end{equation}
\begin{equation}
- {\lambda}^{3}\left( 1-\lambda \right){Y}_{1423}
+{\lambda}^{2}\left( 1-\lambda \right) ^{2}{Y}_{1432}
-{\lambda}^{3} \left( 1-\lambda \right) {Y}_{4123}
+{\lambda}^{2} \left( 1-\lambda \right) ^{2}{Y}_{4132}
\tag{\ref{eq:rusen}c}\plabel{eq:rusenc}
\end{equation}
\begin{equation}
- {\lambda}^{3}\left( 1-\lambda \right){Y}_{2314}
- {\lambda}^{3}\left( 1-\lambda \right){Y}_{2341}
+{\lambda}^{2} \left( 1-\lambda \right) ^{2}{Y}_{3214}
+{\lambda}^{2} \left( 1-\lambda \right) ^{2}{Y}_{3241}
\tag{\ref{eq:rusen}d}\plabel{eq:rusend}
\end{equation}
\begin{equation}
- {\lambda}^{3}\left( 1-\lambda \right){Y}_{2413}
+{\lambda}^{2}\left( 1-\lambda \right) ^{2}{Y}_{2431}
+{\lambda}^{2} \left( 1-\lambda \right) ^{2}{Y}_{4213}
+{\lambda}^{2} \left( 1-\lambda \right) ^{2}{Y}_{4231}
\tag{\ref{eq:rusen}e}\plabel{eq:rusene}
\end{equation}
\begin{equation}
-{\lambda}^{3} \left( 1-\lambda \right) {Y}_{3412}
+{\lambda}^{2} \left( 1-\lambda \right) ^{2}{Y}_{3421}
+{\lambda}^{2} \left( 1-\lambda \right) ^{2}{Y}_{4312}
-\lambda\, \left( 1-\lambda \right) ^{3}{Y}_{4321}
\tag{\ref{eq:rusen}f}\plabel{eq:rusenf}
,
\end{equation}
where we have used the notation $Y_{ijkl}=Y_iY_jY_kY_l$.
Here the monomially induced norm is $-8\,{\lambda}^{3}+8\,{\lambda}^{2}+\lambda$,
the sum of the absolute value of the coefficients.
However, one can do better here in terms of $|\cdot|_{\mathrm F\mathcal A_q}$:
Beside the monomial terms $\pm Y_{ijkl}$ of cost $1$,
we can also use the cross-terms $\pm\frac{Y_{ijkl}+Y_{ijlk}+Y_{jikl}- Y_{jilk}}4$
of cost $\frac12\leq2^{-\frac1q}<1$.

Due due simple nature of the terms,  the minimization problem splits
into six independent problems in lines  \eqref{eq:rusena}--\eqref{eq:rusenf}
respectively.
Restricted to a line,
it is easy to see that if we use two different cross-terms with positive
weights, then we can replace them with monomial terms at less or equal cost.
Similarly, the single cross-term used must be aligned in sign with the
monomial terms used, or we can do a monomial replacement again.
Based on this, using cross-terms is advantageous only in lines
\eqref{eq:rusenb} and \eqref{eq:rusene}.
In line \eqref{eq:rusenb},
the cross-term $\frac{-  {Y}_{1324}- {Y}_{1342}- {Y}_{3124}+{Y}_{3142} }4$
can be used, best with coefficient
$4\cdot\min({\lambda}^{3} \left( 1-\lambda \right)
, {\lambda}^{2} \left( 1-\lambda \right)^{2}   )=4\lambda^2\left( 1-\lambda \right)
\min(\lambda,1-\lambda)$.
This causes the gain (i.~e.~loss)
$ (1-2^{-\frac1q})\cdot 4\lambda^2\left( 1-\lambda \right)
\min(\lambda,1-\lambda)$
regarding the norm.
In line \eqref{eq:rusenb}, the same applies
but regarding the cross-term
$\frac{-  {Y}_{2413}+ {Y}_{2431}+ {Y}_{4213}+ {Y}_{4231} }4$.
Adding all up,
and considering the normalization by $\frac1{(p-1)!}$,
 we obtain the expression indicated for
$\Theta_{0,4}^{(\lambda),\mathcal A_q}$.

The computation of $\Theta_{1,3}^{(\lambda),\mathcal A_q}$
and $\Theta_{2,2}^{(\lambda),\mathcal A}$ proceeds along similar lines.
\end{proof}
\end{lemma}

Then, for $\mathcal A_q=\mathcal{UMQ}_q/\mathbb K$,
we can compute the kernels  $K_4^{(\lambda),\mathcal A_q}(t)$ without trouble.
\begin{lemma}
\plabel{lem:fel}
For $\mathcal A_q=\mathcal{UMQ}_q/\mathbb K$,
specifying to $\lambda=1/2$,
\[K_4^{(1/2),\mathcal A_q}(t)=\frac1{32}\left(\frac23+\frac13 2^{-\frac1q}\right)\]
(independently from $t$).
\begin{proof}
This follows from writing down the kernel explicitly.
\end{proof}
\end{lemma}
\begin{theorem}\plabel{th:caylower}
For $\mathcal A_q=\mathcal{UMQ}_q/\mathbb K$,
regarding the convergence radius $\mathrm C_{\infty}^{(1/2),\mathcal A_q}$ of $\Theta^{(1/2),\mathcal A_q}(x)$,
\[\mathrm C_{\infty}^{(1/2),\mathcal A_q}\geq\sqrt[5]{\mathrm C_{\infty,4}^{(1/2),\mathcal A_q}}=
\dfrac{2}{ \sqrt[5]{\dfrac23+\dfrac13 2^{-\frac1q}}   }>2. \]
\begin{proof}
This follows from Theorem \ref{th:pst} and the previous Lemma \ref{lem:fel}.
\end{proof}
\end{theorem}
I. e. the convergence radius of the
(real) Cayley transform of the time-ordered exponential is at least the value above.
Note that the estimate above can be much improved.
Indeed, we considered the case $p-1=4$, the first degree where
the condition $(\mathcal{UMQ}_q)$ starts to make a difference at all.

Regarding $\mathrm C_{\infty}^{\mathcal A_q}$,
 we expect  $\mathrm C_{\infty}^{(\log),\mathcal A_q}=\mathrm C_{\infty}^{(1/2),\mathcal A_q}$.
This hope is motivated by the idea that regarding the Magnus expansion, $\lambda=1/2$ is the critical case.
However, the case of the BCH expansion can make us cautious.
Now, due to the weaknesses of our methods,
$\mathrm C_{\infty}^{\mathcal A_q}\geq \dfrac{2}{ \sqrt[5]{\dfrac23+\dfrac13 2^{-\frac1q}}   }$
 is likely to be true anyway; however, disappointingly, numerical estimates show that
 $\sqrt[5]{w^{(\lambda),\mathcal A_q}_4}$
 is not maximized by $\lambda=1/2$ neither for $q=1$ or $q=2$ (nor, likely, in general).
Thus, we will be content giving only the following crude lower estimate:
\snewpage

\begin{theorem}\plabel{th:maglower}
For $\mathcal A_q=\mathcal{UMQ}_q/\mathbb K$, regarding
the convergence radius $\mathrm C_{\infty}^{\mathcal A_q}$ of $\Theta^{\mathcal A_q}(x)$,
\[\mathrm C_{\infty}^{\mathcal A_q}\geq\mathrm C_{\infty}^{(\log),\mathcal A_q}\geq\sqrt[5]{\mathrm C_{\infty,4}^{(\log),\mathcal A_q}}>
\dfrac{2}{ \sqrt[5]{\dfrac34+\dfrac14 2^{-\frac1q}}   }>2. \]
\begin{proof}
For $\lambda\in[0,1]$, let us set
\[B(\lambda,t)=\begin{cases}
\frac13\,{\lambda}^{2}\left(
1-\lambda \right)\min \left( \lambda,1-\lambda \right)
\left(
1
-\lambda
-3\,{t}^{2}
+2\,{t}^{3}
+6\,\lambda\,{t}^{2}
-4\,\lambda\,{t}^{3}
\right)
&\text{if } t\in[0,1],
\\
\frac13\, \lambda\left( \lambda-1 \right) ^{2}\,
\min \left( \lambda,1-\lambda \right)
\left(
\lambda
+3\,{t}^{2}
+2\,{t}^{3}
-6\,\lambda\,{t}^{2}
-4\,\lambda\,{t}^{3}
\right)
&\text{if } t\in[-1,0].
\end{cases}\]
Then
\[K^{(\lambda),\mathcal A_q}_4(t)
=K^{(\lambda)}_4(t)-\left(1-2^{-\frac1q}\right)B(\lambda,t).\]

For $\lambda\in\left[ \frac25,\frac35\right]$,
it is easy to check numerically that
$\frac{B(\lambda,t)}{K^{(\lambda)}_4(t)}>\frac14$
(uniformly).
Then, by the trivial estimate $w^{(\lambda)}\leq\frac12$,
\[w^{(\lambda)} \sqrt[5]{S_4(\lambda)}<
\frac12 \sqrt[5]{1-\frac14\left(1-2^{-\frac1q}\right) }=
\frac12 \sqrt[5]{\frac34+\frac142^{-\frac1q} }
.\]
For $\lambda\in\left[ \frac13,\frac23\right]\setminus\left[ \frac25,\frac35\right]$,
it is easy to check numerically that
$\frac{B(\lambda,t)}{K^{(\lambda)}_4(t)}>\frac15$
(uniformly).
Then by the trivial estimate $w^{(\lambda)}\leq w^{(2/5)}$,
\[w^{(\lambda)} \sqrt[5]{S_4(\lambda)}<
w^{(2/5)}\sqrt[5]{1-\frac15\left(1-2^{-\frac1q}\right) }<
\frac12 \sqrt[5]{\frac34+\frac142^{-\frac1q} }.
\]
(The latter inequality can be checked by taking the fifth power.)
For $\lambda\in[0,1]\setminus \left[ \frac13,\frac23\right]$,
\[w(\lambda) \sqrt[5]{S_4(\lambda)}\leq w(\lambda)\leq w\left(\frac13\right)<
\frac12 \sqrt[5]{\frac34+\frac142^{-1} }\leq\frac12 \sqrt[5]{\frac34+\frac142^{-\frac1q} }
.\]
Altogether, we find $w(\lambda) \sqrt[5]{S_4(\lambda)}<\frac12 \sqrt[5]{\frac34+\frac142^{-\frac1q} }$
(actually, with a quantifiable uniform gap.)
 Now, the statement follows from Theorem \ref{th:Compar}.
\end{proof}
\end{theorem}
The kernel method here happens to produce stronger estimates than our previous ones.
We will not details this here, but see the numerical values in the forthcoming examples.

\subsection{Upper estimates for the cumulative radii and comparisons}
\begin{theorem}
\plabel{th:cayupper}
For $\mathcal A_q=\mathcal{UMQ}_q/\mathbb K$,
regarding the convergence radius $\mathrm C_{\infty}^{(1/2),\mathcal A_q}$ of $\Theta^{(1/2),\mathcal A_q}(x)$,
\[\mathrm C_{\infty}^{(1/2),\mathcal A_q}\leq 2\cdot 2^{\frac1{3q}}. \]
\begin{proof}
One can see that
$\Theta^{(1/2),\mathcal A_q}_n\geq \left(2^{-\frac1{3q}} \right)^n \Theta^{(1/2)}_n$.
Indeed, this follows from reducing the cost of the monomials $M$
to $\left(2^{-\frac1{3q}} \right)^{\deg M}$, where the
conditions coming from $(\mathcal{UMQ}_q)$ become irrelevant.
(Into a monomial $M$ of degree $\deg M$ at most $\frac13 \deg M$
many `$\Xi^{\mathrm{symb}}$' can be inserted.)
However, we know that the convergence radius of $\Theta^{(1/2)}(x)$ is $2$.
\end{proof}
\end{theorem}
(The estimate above, however, says nothing for concrete algebras.)

\begin{example}
\plabel{ex:cay}
For $q=2$, the upper and lower estimates yield
\[\dfrac{2}{ \sqrt[5]{\dfrac23+\dfrac13 2^{-\frac12}}   }=2.041\ldots\leq
\mathrm C_{\infty}^{(1/2),\mathcal{UMQ}_2/\mathbb K}
\leq 2\cdot \sqrt[6]{2} = 2.244\ldots
\]
as a consequence.
This shows that the class $\mathcal A_2=\mathcal{UMQ}_2/\mathbb K$
is still quite distant from the class of Hilbert spaces, where
$\mathrm C_{\infty}^{(1/2),\mathrm{Hilbert}} =\pi$ is known.
(Using norm inequalities to characterize Banach algebras
is not as an entirely hopeless idea, as the case of $C^*$-algebras
shows, but the homogeneous condition ($\mathcal{UMQ}_q$) is apparently too weak.)
Even for $q=1$, our estimates yield only
\[\dfrac{2}{ \sqrt[5]{\dfrac23+\dfrac13 2^{-1}}   }= 2.074\ldots\leq
\mathrm C_{\infty}^{(1/2),\mathcal{UMQ}_1/\mathbb K}
\leq 2\cdot \sqrt[3]{2} =  2.519\ldots\quad.
\eqedexer\]
\end{example}

We have similar trivial upper estimates as before:
\begin{theorem}
\plabel{th:magupper}
For $\mathcal A=\mathcal{UMQ}_q/\mathbb K$, regarding
the convergence radius $\mathrm C_{\infty}^{(\lambda),\mathcal A_q}$ of $\Theta^{(\lambda),\mathcal A_q}(x)$,
\[\mathrm C_{\infty}^{(\lambda),\mathcal A_q}\leq \frac1{w(\lambda)}\cdot 2^{\frac1{3q}}. \]
Furthermore, regarding the convergence radius $\mathrm C_{\infty}^{ \mathcal A_q}$ of $\Theta^{ \mathcal A_q}(x)$,
\[\mathrm C_{\infty}^{ (\log),\mathcal A_q}\leq \mathrm C_{\infty}^{ \mathcal A_q}\leq 2\cdot 2^{\frac1{3q}}. \]
\begin{proof}
Estimating the norms in the expansion of $\mathrm Z^{\mathcal A_q}_{[0,1)}$,
we can relax the cost of monomials as in the proof of Theorem \ref{th:cayupper}.
\end{proof}
\end{theorem}
\begin{example}
\plabel{ex:mag}
Again, we can consider special cases for $q$, where numerical estimates are easy due to Theorem \ref{th:average}.
For $q=2$, the  estimates yield
\begin{multline*}
\dfrac{2}{ \sqrt[5]{\dfrac34+\dfrac14 2^{-\frac12}}   }=2.030\ldots<
\sqrt[5]{\mathrm C_{\infty,4}^{(\log),\mathcal{UMQ}_2/\mathbb K}}=2.040800\ldots
\leq\\\leq
\mathrm C_{\infty}^{(\log),\mathcal{UMQ}_2/\mathbb K}\leq \mathrm C_{\infty}^{\mathcal{UMQ}_2/\mathbb K}
\leq 2\cdot \sqrt[6]{2} = 2.244\ldots\quad.
\end{multline*}

For $q=1$, the estimates yield
\begin{multline*}
\dfrac{2}{ \sqrt[5]{\dfrac34+\dfrac14 2^{-1}}   }=2.054 \ldots
<
\sqrt[5]{\mathrm C_{\infty,4}^{(\log),\mathcal{UMQ}_1/\mathbb K}}=2.071801\ldots
\leq\\\leq
\mathrm C_{\infty}^{(\log),\mathcal{UMQ}_1/\mathbb K}\leq\mathrm C_{\infty}^{\mathcal{UMQ}_1/\mathbb K}
\leq 2\cdot \sqrt[3]{2} =  2.519\ldots\quad.
\end{multline*}

In this cases $\sqrt[5]{\mathrm C_{\infty,4}^{(\log),\mathcal{UMQ}_1/\mathbb K}}$ is still rather close to
$\sqrt[5]{\mathrm C_{\infty,4}^{(1/2),\mathcal{UMQ}_1/\mathbb K}}$;
thus the estimate of Theorem \ref{th:maglower} is indeed not too sharp.
\qedexer
\end{example}
\snewpage
\section{The case of the BCH expansion}
\plabel{sec:resBCH}
Two natural ways to consider the convergence of the BCH expansion
are absolute convergence grouped by joint homogeneity in the variables
(that is as a Magnus expansion) and absolute convergence grouped by separate homogeneity in the variables
(that is the ``bigraded'' version).

Here we can use the algebras $\mathrm F^{\mathcal A}[Y_1,Y_2]$ in order to deal with the convergence question.
For $x_1,x_2\geq0$, we define
\[\Gamma^{{\mathcal A}}(x_1,x_2)=
\sum_{n=1}^\infty\Biggl|\underbrace{\sum_{k=0}^{n} \BCH_{k,n-k}(x_1Y_1,x_2Y_2) }_{\equiv \BCH_{n}(x_1Y_1,x_2Y_2)}\Biggr|_{\mathrm F\mathcal A}.\]
One can see that $0\leq \tilde x_1\leq x_1$ and $0\leq \tilde x_2\leq x_2$ imply that
$\Gamma^{{\mathcal A}}(\tilde x_1,\tilde x_2)\leq \Gamma^{{\mathcal A}}(x_1,x_2)$.
(This is because of universal algebras where defined in terms of the $\leq$ relation, and the variables can be rescaled.)
Then in any $\mathcal A$-algebra $\mathfrak A$, the BCH expansion of $X_1$ and $X_2$
(in joint homogeneity) converges if $\Gamma^{{\mathcal A}}(|X_1|_{\mathfrak A},|X_2|_{\mathfrak A})<+\infty$.
Conversely, if $\Gamma^{{\mathcal A}}(x_1,x_2)=+\infty$, then a counterexample
for the convergence is provided by $X_1=x_1Y_1$ and $X_2=x_2Y_2$ in $\mathrm F^{\mathcal A}[Y_1,Y_2]$.

In a similar manner,
\[\Gamma^{{\mathrm h \mathcal A}}(x_1,x_2)
=\sum_{n=1}^\infty\left|\sum_{k=0}^{n} \BCH_{k,n-k}(x_1Y_1,x_2Y_2) \right|_{\mathrm F\mathrm h\mathcal A}
=\sum_{n=1}^\infty\sum_{k=0}^{n} \left|\BCH_{k,n-k}(x_1Y_1,x_2Y_2) \right|_{\mathrm F\mathcal A}
\]
concerns the absolute convergence in separate homogeneity.
We will deal with this latter version.
Thus we are looking for $x_1,x_2$ such that $\Gamma^{{\mathrm h\mathcal A}}(x_1,x_2)<+\infty$.
(But note, for $\mathcal A=\mathcal{UMQ}_q/\mathbb K$ we have `$\mathcal A =\mathrm h\mathcal A$'.)

For $\lambda\in[0,1]$, we set
\[\Upsilon^{(\lambda)}(x_1Y_1,x_2Y_2)=\lambda(1-\lambda)\mathcal R^{(\lambda)}(\exp x_1Y_1)\mathcal R^{(\lambda)}(\exp x_2Y_2).\]
As a formal series this exists, but it also exists in $\mathrm F^{1}[Y_1,Y_2]$  (thus also in $\mathrm F^{\mathcal A}[Y_1,Y_2]$)
if $x_1,x_2<\pi$.

\snewpage
\begin{theorem}
\plabel{th:resBCH}
Suppose that $0\leq x_1,x_2<\pi$. If for some  $n\geq1$,
\[\sup_{\lambda\in[0,1]}\sqrt[n]{\left|\Upsilon^{(\lambda)}(x_1Y_1,x_2Y_2)^n\right|_{\mathrm F\mathrm h\mathcal A}}<1, \]
then
\[\Gamma^{\mathrm h\mathcal A}(x_1,x_2)<+\infty.\]

In particular, if for the $|\cdot|_{\mathrm F\mathrm h\mathcal A}$-spectral radius
\[\sup_{\lambda\in[0,1]}\mathrm r_{|\cdot|_{\mathrm F\mathrm h\mathcal A}}\left(\Upsilon^{(\lambda)}(x_1Y_1,x_2Y_2)^n\right)<1, \]
then the conclusion applies.
\begin{proof} According to  Part I,  (formally)
\begin{align}
\BCH(x_1Y_1,x_2Y_2)=&\int_{\lambda=0}^1\mathcal R^{(\lambda)}((\exp x_1Y_1)(\exp x_2Y_2))\,\mathrm d\lambda\notag\\
=&\int_{\lambda=0}^1(1-\Upsilon^{(\lambda)}(x_1Y_1,x_2Y_2)   )^{-1}\mathcal R^{(\lambda)}(\exp x_1Y_1)\notag\\
&+\mathcal R^{(\lambda)}(\exp x_2Y_2)(1-\Upsilon^{(\lambda)}(x_1Y_1,x_2Y_2)  )^{-1}\notag\\
&+\lambda \mathcal R^{(\lambda)}(\exp x_1Y_1)\mathcal R^{(\lambda)}(\exp x_2Y_2)(1-\Upsilon^{(\lambda)}(x_1Y_1,x_2Y_2)  )^{-1} \notag\\
&+(\lambda-1) \mathcal R^{(\lambda)}(\exp x_2Y_2)(1-\Upsilon^{(\lambda)}(x_1Y_1,x_2Y_2)  )^{-1}\mathcal R^{(\lambda)}(\exp x_1Y_1) \notag\\
&\mathrm d\lambda\notag,
\end{align}
completely well-defined in every $(Y_1,Y_2)$-grade.
Then, via the relevant Neumann series, the norm of the expression is bounded.
\end{proof}
\end{theorem}
The statement also applies to the  case of  $|\cdot|_{\ell^1}$ (cf. Part I),
except in that case there is no difference between the spectral radius and the norm of $\Upsilon^{(\lambda)}(x_1Y_1,x_2Y_2)$.
So, in Part I only the $|\cdot|_{\ell^1}$ norm was used.
We have demonstrated in Part I that on the domain $0\leq x_1+x_2\leq\mathrm  C_2=2.89847930\ldots$, $\lambda\in[0,1]$ the inequality
\[\left|\Upsilon^{(\lambda)}(x_1Y_1,x_2Y_2)\right|_{\ell^1}
\leq1, \]
holds; and in case of equality $x_1=x_2=\frac12 \mathrm  C_2$ and $0.35865<\min(\lambda,1-\lambda)<0.35866$.
(Thus, by the symmetry $\lambda\leftrightarrow 1-\lambda$ equality occurs at least for two such $\lambda$,
but, although unlikely, there might more than two such values.)
The statement which requires more work is that the $\BCH$ expansion of $\frac12 \mathrm  C_2\cdot Y_1$
and $\frac12 \mathrm  C_2\cdot Y_2$ will diverge in   $\mathrm F^{\mathcal A}[Y_1,Y_2]$, thus
$\mathrm  C_2$ is the general convergence radius of the BCH expansion regarding the cumulative norm
in the general.

\begin{lemma}
\plabel{lem:resBCHcrit}
For $\mathcal A=\mathcal{UMQ}_q/\mathbb K$,
the domain condition
\begin{equation}\text{
$x_1=x_2=\frac12 \mathrm  C_2$ and $0.35865\leq\min(\lambda,1-\lambda)\leq 0.35866$
}
\plabel{eq:domain}
\end{equation}
implies
\[
\left|\Upsilon^{(\lambda)}(x_1Y_1,x_2Y_2)^3\right|_{\mathrm F\mathrm h\mathcal A}<
\left|\Upsilon^{(\lambda)}(x_1Y_1,x_2Y_2)^3\right|_{\ell^1}
.\]
\begin{proof}
Let us  compare $\left|\Upsilon^{(\lambda)}(x_1Y_1,x_2Y_2)^3\right|_{\mathrm F\mathrm h\mathcal A}$ and $
\left|\Upsilon^{(\lambda)}(x_1Y_1,x_2Y_2)^3\right|_{\ell^1}$
The first one is less or equal than the second one, actually degree-wise (in $Y_1$ and $Y_2$ separately).
Let us consider the part $\deg_{(Y_1,Y_2)}=(3,5)$.
After some computation, one finds that
\begin{multline}\left(\Upsilon^{(\lambda)}(x_1Y_1,x_2Y_2)^3\right)_{\deg_{(Y_1,Y_2)}=(3,5) }=(x_1)^3(x_2)^5 \cdot\lambda^3(1-\lambda)^3\cdot\\
\Biggr(
 \lambda^8\cdot\left( {\lambda}^{2}-\lambda+\frac14 \right)Y_{12212212}
+\lambda^8\cdot\left( {\lambda}^{2}-\lambda+\frac14 \right)Y_{12122122}\\
+\lambda^8\cdot\left( {\lambda}^{2}-\lambda+\frac14 \right)Y_{12212122}
+\lambda^8\cdot\left( {\lambda}^{2}-\lambda+\frac16 \right)Y_{12122212}
+\text{other terms}
\Biggl),
\plabel{eq:san}
\end{multline}
where  $Y_{12122122}\equiv Y_1Y_2Y_1Y_2Y_2Y_1Y_2Y_2 $, etc.
Regarding the norm $|\cdot|_{\mathrm F\mathcal A}$ of \eqref{eq:san},
it becomes advantageous to use norm gain for the quasi-monomial
\begin{equation}
Y_1Y_2\Xi( Y_2, Y_1, Y_2Y_1, Y_2) Y_2
=
\frac{Y_{12212212}+Y_{12122122}+Y_{12212122}-Y_{12122212}}4.
\plabel{eq:cosan}
\end{equation}
(Remark: there are several other quasi-monomial presentations for this given non-commutative polynomial.)
Indeed, under \eqref{eq:domain},
the coefficients of the monomials $Y_{12212212}$, $Y_{12122122}$, $Y_{12212122}$, $Y_{12122212}$
are of sign $+,+,+,-$, respectively,
both in \eqref{eq:san} and \eqref{eq:cosan}.
In fact, the norm gain coming from this is
\[(x_1)^3(x_2)^5 \cdot\lambda^3(1-\lambda)^3\cdot4\cdot
\lambda^8\cdot\left( {\lambda}^{2}-\lambda+\frac14 \right)\left(1-2^{-\frac1q}\right).\]

This implies the statement.
\end{proof}
\end{lemma}
Let us define the cumulative radius of the BCH-$\mathcal A$ expansion as
\[\mathrm C^{\mathcal A}_2 =\inf\{ x_1+x_2\,:\, \Gamma^{\mathcal A}(x_1,x_2)=+\infty  \}.\]
Similar definition can be made regarding `$\mathrm h\mathcal A$'.

\begin{theorem}
\plabel{th:bchimprove}
For $\mathcal A=\mathcal{UMQ}_q/\mathbb K$,
\[\mathrm C^{\mathcal A}_2= \mathrm C^{\mathrm h\mathcal A}_2>\mathrm C_2.\]
\begin{proof}
We know that for $x_1+x_2\leq\mathrm  C_2$,
\[\sqrt[3]{\left|\Upsilon^{(\lambda)}(x_1Y_1,x_2Y_2)^3\right|_{\mathrm F\mathrm h\mathcal A}}\leq
\sqrt[3]{\left|\Upsilon^{(\lambda)}(x_1Y_1,x_2Y_2)^3\right|_{\ell^1}}=\left|\Upsilon^{(\lambda)}(x_1Y_1,x_2Y_2)\right|_{\ell^1}\leq1 \]
holds.
The second inequality is strict outside \eqref{eq:domain}, while the first inequality is strict on \eqref{eq:domain}.
Thus, for $x_1+x_2\leq\mathrm  C_2$,
\[\sqrt[3]{\left|\Upsilon^{(\lambda)}(x_1Y_1,x_2Y_2)^3\right|_{\mathrm F\mathrm h\mathcal A}}<1\]
holds.
By the continuity of the LHS for $x_1,x_2\leq\pi$, and compactness, we know that this extends for
$x_1+x_2\leq\mathrm  C_2+\varepsilon$ with some $\varepsilon>0$.
This yields $\mathrm C^{\mathrm h\mathcal A}_2>\mathrm C_2$, while
$\mathcal A=\mathrm h\mathcal A$ is known.
\end{proof}
\end{theorem}

\section{Conclusion and discussion}
\plabel{sec:resConclude}
By this we have shown that
for a large class algebras exhibit convergence improvement
with respect to the Magnus  expansion compared to the general
case of Banach algebras.

\begin{remark}
\plabel{rem:Hilbert}
We can define the class $\mathcal A=\mathrm{Hil}/\mathbb K$, by considering all noncommutative polynomials
$P(X_1,\ldots,X_m)$ over $\mathbb K$, and we can consider all possible (optimal)
estimates \[\|P(X_1,\ldots,X_m)\|\leq C_P\]
applicable to Hilbert space operators $X_i$ with $\|X_i\|\leq 1$.
In practice, this large family is not manageable.
In theory, however, our method is applicable to approximate the cumulative convergence radius $\pi$
for the Magnus expansion in the Hilbert operator case.
Indeed, taking sufficiently refined mBCH approximations (whose norm-growth factor we can quantify as in Part I),
we can obtain estimates for the Magnus expansion even from the finite-variable case(s).
However, the spectral inclusion method of Part II is completely manageable.
On the other hand, the analogous homogeneous case $\mathrm h\mathcal A=\mathrm h\mathrm{Hil}/\mathbb K$
measures the growth of the Magnus commutators, which cannot be done directly with the spectral inclusion method
(but recursive methods are, in general, applicable).
As for now, $\mathrm C_\infty^{\mathrm h\mathrm{Hil}/\mathbb K}\geq
\mathrm C_{\infty}^{\mathcal{UMQ}_2/\mathbb K}
\geq\sqrt[5]{\mathrm C_{\infty,4}^{\mathcal{UMQ}_2/\mathbb K}}
=2.0408\ldots>2$
is a very weak but explicit (and easy-to-improve) estimate in that regard.
The quasifree class $\mathrm{Hil}/\mathbb K$ is, in spirit,
similar to $\mathcal{UMD}_q/\mathbb K$.
\qedremark
\end{remark}
\begin{remark}
\plabel{rem:BanachLie}
In Part III, we apply the resolvent method to the case of Banach--Lie algebras
(where the norm condition given by $\|[X,Y]\|\leq\|X\|\cdot\|Y\|$).
There the universal Banach algebras are given not by general norm relations
but by prescriptions given to commutator monomials of generating variables.
It results the quasifree class $\mathrm{Lie}/\mathbb K$.
This quasifree class $\mathrm{Lie}/\mathbb K=\mathrm h\mathrm{Lie}/\mathbb K$ is, in spirit,
similar to $\mathcal{UMQ}_q/\mathbb K$.
\qedremark
\end{remark}

Note that the resolvent method, as it was given, provides lower estimates not directly for $\mathrm C_\infty^{\mathcal A}$,
but through $\mathrm C_\infty^{(\log),\mathcal A}$.
Therefore, as the scalar case shows, it might be not the best method if the cumulative convergence radius $\mathrm C_\infty^{\mathcal A}$ is greater than $\pi$; or, in the Lie case, if we aim above convergence radius $2\sqrt2$.

On the other hand,  the resolvent method can be applied well to study the convergence of the Magnus expansion
of individual ordered measures.
In that case the resolvent estimating kernels might not be particularly symmetric anymore.
\snewpage

\appendix
\section{ Integral operators on $L^2([0,1])$ with nonnegative kernels}
\plabel{sec:int}
In the text we primarily consider integral operators of continuous kernel, but
here we state the relevant theorems in somewhat greater generality.

For the sake of simplicity, we consider integral operators on $L^2([0,1])$ (real or complex, it does not matter).
Recall $K\in L^2([0,1]^{2})$ means that $K$ is a (real or complex) function $[0,1]^2$ well-defined almost everywhere
such that $L^2$ norm as
\[|K|_{L^2}=\sqrt{\int_{(s,t)\in[0,1]^2}|K(s,t)|^2 \mathrm ds\,\mathrm dt}<+\infty.\]
The situation is similar for $f\in L^2([0,1])$.
If $K_1,K_2\in L^2([0,1]^{2})$, then we can define the function $K_1*K_2$ on $[0,1]^2$ by
\[K_1*K_2(s,t)=\int_{r=0}^1 K_1(s,r)K_2(r,t)\, \mathrm dr.\]
This is well-defined almost everywhere and
\[|K_1*K_2|_{L^2}\leq |K_1 |_{L^2} |K_2|_{L^2} ;\]
in particular, it yields $K_1*K_2\in L^2([0,1]^{2})$.
Similarly, for $K\in L^2([0,1]^{2})$, $f\in L^2([0,1])$  we can define the function $K*f$ on $[0,1]$ by
\[K_1*f(s)=\int_{r=0}^1 K_1(s,r)f(r)\, \mathrm dr.\]
This is well-defined almost everywhere and
\[|K*f|_{L^2}\leq |K |_{L^2} |f|_{L^2} ;\]
in particular, it yields $K*f\in L^2([0,1])$.
The associative rules
\[K_1*(K_2*K_3)=(K_1*K_2)*K_3\]
and
\[K_1*(K_2*f)=(K_1*K_2)*f\]
hold for $K_1,K_2,K_3\in L^2([0,1]^{2})$ and $f\in L^2([0,1])$.
In what follows we drop the term `almost everywhere', as it will be understood.

If $K\in L^2([0,1]^2)$, then it defines the integral operator $I_K$ by
\[I_K: f\in L^2([0,1])\mapsto K*f\in L^2([0,1]).\]
It is a consequence of the associative rule that
$I_{K_1*K_2}=I_{K_1}I_{K_2}$ holds, etc.
According to the previous discussion, regarding the operator norm,
\begin{equation}
\|I_K\|_{L^2}\leq |K|_{L^2}.
\plabel{eq:trix}
\end{equation}
Now, $K\in L^2([0,1]^2)$ can be approximated by
rectangularly based step-functions $K_n$ in $|\cdot|_{L^2}$.
Then, by \eqref{eq:trix}, $I_K$ gets approximated by $I_{K_n}$ in $\|\cdot\|_{L^2}$.
However, these latter $I_{K_n}$ are operators of finite rank.
This yields that $I_K$ is compact as a linear operator on $L^2([0,1])$.
Consequently, the spectrum of $I_K$ is discrete (with finite multiplicities) except at $0\in\spec(I_K)$.
\snewpage

A major advantage is that the operations `spectrum' and `spectral radius' are not only
upper semicontinuous but continuous at compact operators.
More precisely: If $A_n\rightarrow A$ for bounded operators, then
\[\spec(A)\supset\bigcap_{N}\overline{\bigcup_{n \geq N}\spec(A_n)}\]
and, in particular,
\[{\mathrm r}(A)\geq \limsup_{n} {\mathrm r}(A_n)\]
hold. (This follows from elementary resolvent calculus.)
If $A$ is compact, then, however,
\[\spec(A)=\bigcap_{N}\overline{\bigcup_{n \geq N}\spec(A_n)}\]
and, in particular,
\[{\mathrm r}(A)= \lim_{n} {\mathrm r}(A_n)\]
hold. (This follows because, for possibly small perturbations of a compact operator, multiplicities can be tested by line integrals of the
resolvent.)

Regarding the nonnegative kernels in $L^2([0,1]^2)$,
one deals with the generalization of the classical
Perron--Frobenius theory initiated by
Perron \cite{Per0}, \cite{Per} and Frobenius \cite{Fro1}, \cite{Fro2}
(see Gantmacher \cite{Ga} for a classical review.)

First of all, let us observe the following monotonicity statements.
If $J_1,J_2,K_1,K_2\in L^2([0,1]^2$, then
\[|J_1|\leq K_1,|J_2|\leq K_2\qquad\Rightarrow\qquad |J_1*J_2|\leq K_1*K_2.\]
Similarly, if $J,K\in L^2([0,1]^2$, $g,f\in L^2([0,1])$, then
\[|J|\leq K,|g|\leq f\qquad\Rightarrow\qquad |J*g|\leq K*f.\]
From this it is easy to deduce
\begin{theorem}
(a) If $0\leq K\in L^2([0,1]^2)$, then
\[|I_K|_{L^2}=\sup\{ |K*f|_{L^2} \,:\, f\in L^2([0,1]),\, |f|_{L^2}=1, \, f\geq0\}.\]

(b) If $0\leq K_1\leq K_2$ or just $|K_1|\leq K_2$, then
\[|I_{K_1}|_{L^2}\leq|I_{K_2}|_{L^2}.\]

(c) If  $0\leq K_1\leq K_2$ or just $|K_1|\leq K_2$, then
\[\mathrm r(I_{K_1}) \leq\mathrm r(I_{K_2}).\]
\begin{proof}
(a) and (b) are immediate from the monotonicity statements.
(c) follows from monotonicity and the general Banach algebraic rule $\mathrm r(A)=\liminf_n \sqrt[n]{|A^n|}$.
\end{proof}
\end{theorem}
Here point (c) generalizes the majorization theorem of Frobenius \cite{Fro2}.
Perron's theorem is generalized by
\begin{theorem}[Jentzsch \cite{Jen} (1912), cf. Hochstadt \cite{Hch}]
\plabel{th:Jentzsch}

Suppose that $K$ is positive and continuous.
Then ${\mathrm r}(I_K)\in\spec(I_K)$.
This eigenvalue ${\mathrm r}(I_K)$ has multiplicity $1$ and it allows a positive and continuous eigenvector.
All other eigenvalues are of smaller absolute value.
\qed
\end{theorem}

Continuity in itself is not essential in the theorem above.
Historically, Jentzsch \cite{Jen} uses the theory of Fredholm \cite{Fre} (cf. Birkhoff \cite{Bi0}),
which applies only for continuous kernels.
However, analytic Fredholm theory was extended to $L^2$ kernels by Hilbert \cite{Hil} and Carleman \cite{Car}
(cf. Smithies \cite{Smi} or Simon \cite{Si}).
Then `positive and continuous' can be replaced by `positively bounded' (from above and below; measurability is understood),
without essential change in the argument.
This stronger statement, however, was spelled out only relatively late by Birkhoff \cite{Bi1}, but already in a much greater generality.

\begin{theorem}[Birkhoff \cite{Bi1} (1957), special case]
\plabel{th:Birkhoff}
Assume that $ m\cdot 1_{[0,1]^2} \leq  K \leq M\cdot1_{[0,1]^2}$ (almost everywhere), where $0<m\leq M<+\infty$.
Then ${\mathrm r}(I_K)\in\spec(I_K)$.
This eigenvalue ${\mathrm r}(I_K)$ has multiplicity $1$; and
for the corresponding nonnegative eigenvector $f$, it can be assumed that $ m\cdot 1_{[0,1]} \leq  f \leq M\cdot1_{[0,1]}$.
All other eigenvalues are of smaller absolute value.
The ratio of the other (subdominant) eigenvalues to the (dominant) eigenvalue $\mathrm r(I_K)$ can be estimated by
some explicit expressions $\omega(K)\leq\omega(m,M)<1$ (in particular, uniformly in $m,M$).
\qed
\end{theorem}
\begin{proof}[Remark]
Applied to rectangularly based positive step-functions this directly generalizes Perron's theorem.
\renewcommand{\qedsymbol}{$\triangle$}
\end{proof}

Indeed, a more general approach (in terms of Banach lattices) was put forward previously by Krein and Rutman \cite{KR}
in order
to treat phenomena regarding nonnegative kernels.
\begin{theorem}[Krein, Rutman \cite{KR} (1948), special case]
\plabel{th:KreinRutman}
Assume that $K\geq0$. Then:

(a) ${\mathrm r}(I_K)\in\spec(I_K)$.

(b) If ${\mathrm r}(I_K)>0$, then $I_K$ admits a nonnegative eigenvector for ${\mathrm r}(I_K)$.
\qed
\end{theorem}
This generalizes the general (weak) Perron--Frobenius theorem.
Subsequent development (using the Banach lattice terminology)
 led to
\begin{theorem}[And\^{o} \cite{An} (1957), special case]
\plabel{th:Ando}
Assume that $K\geq0$. Assume that $K$ is irreducible, i.~e.~for any $J\subset[0,1]$ with
$\mathbf1(J)>0$ and $\mathbf1([0,1]\setminus J) >0$ (Lebesgue measure)
\[\int_{(s,t)\in J\times([0,1]\setminus J )} K(s,t)\, \mathrm ds\,\mathrm dt>0\]
holds. Then ${\mathrm r}(I_K)>0$.
\qed
\end{theorem}
This generalizes the (sharper) theorem of Frobenius.
(The corresponding much more general statement is the so-called Ando--Krieger theorem, after And\^{o} \cite{An} and Krieger \cite{Kr},
cf. Dodds \cite{Do}.)
A statement generalizing the (sharper) theorem of Perron is

\begin{theorem}[Schaefer \cite{Schae} (1974), special case]
\plabel{th:Schaefer}
Assume that $K>0$ almost everywhere.
Then ${\mathrm r}(I_K)>0$ has multiplicity $1$ in the spectrum and all other eigenvalues are of smaller absolute value.
\qed
\end{theorem}

For our purposes it will be sufficient to know only Theorem \ref{th:KreinRutman}
(which in its present form is an easy limiting case of Perron's theorem via the continuity of the spectrum).
A useful consequence of Theorem \ref{th:KreinRutman} is
\begin{theorem}\plabel{th:Dominate}
(Spectral locality, special case.)
Assume that $K\geq0$. Then the following quantities exist and  are equal:

(i) \[{\mathrm r}(I_K)\equiv\max\{|\lambda|\,:\,\lambda\in \spec(I_K)\}=\lim_n \sqrt[n]{\|(I_K)^n\|_{L^2}}
=\inf_{n\in\mathbb N\setminus\{0\}} \sqrt[n]{\|(I_K)^n\|_{L^2}};\]

(ii) \[ \lim_n \sqrt[n]{|(I_K)^n1_{[0,1]}|_{L^2}};\]

(iii) \[
\lim_n \sqrt[n]{\langle 1_{[0,1]},(I_K)^n1_{[0,1]} \rangle }.\]
\begin{proof}
(i) contains well-known equivalent (general Banach algebraic) descriptions of the spectral radius of $I_K$.
In general, note that
\begin{equation}
\sqrt[n]{\langle 1_{[0,1]},(I_K)^n1_{[0,1]} \rangle }\leq \sqrt[n]{|(I_K)^n1_{[0,1]}|_{L^2}}  \leq \sqrt[n]{\|(I_K)^n\|_{L^2}} .
\plabel{eq:krmon}
\end{equation}

From this,
\begin{equation}
\limsup_n \sqrt[n]{\langle 1_{[0,1]},(I_K)^n1_{[0,1]} \rangle }\leq \limsup \sqrt[n]{|(I_K)^n1_{[0,1]}|_{L^2}}  \leq {\mathrm r}(I_K).
\plabel{eq:krsup}
\end{equation}
is immediate.

If ${\mathrm r}(I_K)=0$, then limits are all $0$, implying the statement.

 If ${\mathrm r}(I_K)>0$, but $K$ is essentially bounded from above, then $I_K$ has an eigenvector $f$
associated to the eigenvalue ${\mathrm r}(I_K)$ such that $0\leq f\leq 1_{[0,1]} $ can be assumed.
Thus
\begin{equation}
\liminf_n \sqrt[n]{\langle 1_{[0,1]},(I_K)^n1_{[0,1]} \rangle }\geq
\liminf_n \sqrt[n]{\langle f,(I_K)^n f \rangle }=\liminf_n {\mathrm r}(I_K)\sqrt[n]{\langle f, f \rangle }={\mathrm r}(I_K).
\plabel{eq:krinf}
\end{equation}
Comparing \eqref{eq:krinf} and \eqref{eq:krsup} implies the statement.

In general, if ${\mathrm r}(I_K)>0$, then let $K_n=\max(K,n)$ where $n\in \mathbb N$.
Then $K_n\rightarrow K$ in $L^2$ norm.
By continuity of the spectrum ${\mathrm r}(K_n)\rightarrow {\mathrm r}(K)$.
Then, by the monotonicity of (iii) / (ii) / (i) in nonnegative $K$, the statement follows.
\end{proof}
\end{theorem}
(We could easily replace $1_{[0,1]}$ by any positively bounded function in the statement above, but it is sufficient for us in its present form.)
We can reformulate the previous theorem using some extra terminology.
Assume that $K_n\in  L^2([0,1]^2)$ for $n\in\mathbb N\setminus\{0\}$ such that $K_n\geq0$.
We say that the assignment $K_\bullet:n\mapsto K_n$ forms a submultiplicative family, if for any $n,m\in\mathbb N\setminus\{0\}$, the inequality
$K_n*K_m\leq K_{n+m}$ holds.

\begin{lemma}
\plabel{lem:bunch}
Suppose that  $K_\bullet:n\mapsto K_n$ forms a submultiplicative family of nonnegative kernels. Then

(a)
\begin{equation}
\inf_n \sqrt[n]{{\mathrm r}(I_{K_n})}=\lim_n \sqrt[n]{{\mathrm r}(I_{K_n})}=
 \inf_n \sqrt[n]{\|I_{K_n}\|_{L^2}}=
 \lim_n \sqrt[n]{\|I_{K_n}\|_{L^2}}
.
\plabel{eq:bunch}
\end{equation}

(b)
\begin{multline}
\underbrace{\liminf_n\sqrt[n]{\langle 1_{[0,1]},(I_{K_n})1_{[0,1]} \rangle }}_{{\mathrm r}'''(K_\bullet):=}\leq
\underbrace{\limsup_n\sqrt[n]{\langle 1_{[0,1]},(I_{K_n})1_{[0,1]} \rangle }}_{{\mathrm r}''(K_\bullet):=}\leq \\\leq
\underbrace{\limsup_n\sqrt[n]{|(I_{K_n})1_{[0,1]}|_{L^2}}}_{{\mathrm r}'(K_\bullet):=}  \leq
\underbrace{\lim_n\sqrt[n]{\|(I_{K_n})\|_{L^2}}}_{{\mathrm r}(K_\bullet):=}
 .
 \plabel{eq:submon}
\end{multline}
\begin{proof}
This follows from the monotonicity relations directly (without applying any Perron--Frobenius theory).
\end{proof}
\end{lemma}
We may say that  $K_\bullet$ is relatively local if ${\mathrm r}'''(K_\bullet)={\mathrm r}(K_\bullet)$, i.~e.~if
equality holds in \eqref{eq:submon} throughout.
Then Theorem \ref{th:Dominate} says that in case of $K\geq0$, the assignment $n\mapsto K^{*n}$ ($n\in\mathbb N\setminus\{0\}$)
is relatively local. This viewpoint is not particularly important for us, but Lemma \ref{lem:bunch} is has some practicality.

For the following statement, it is hard to point out a ``first'';
it was likely known to every investigator of (the generalized) Perron--Frobenius theory in the particular setting they used:
\begin{theorem}
\plabel{th:average}
(Averaging principle, special case.)
Assume that $K\geq0$.
For $n\in\mathbb N$,
\begin{equation}
n\mapsto\left[\essinf\frac{(I_K)^{n+1}1_{[0,1]}}{(I_K)^n1_{[0,1]}},
\esssup\frac{(I_K)^{n+1}1_{[0,1]}}{(I_K)^n1_{[0,1]}}\right]
\plabel{eq:average}
\end{equation}
yields a sequence of encapsulated intervals (all) containing ${\mathrm r}(I_K)$.

(Here $\frac00=$``undecided''; if the quotient is $\frac00$ almost everywhere, i.~e.~if $(I_K)^n1_{[0,1]}=0$ is reached, then we set the interval to be $[0,0]$.)
\begin{proof} If $C\in[0,+\infty)$ and $0\leq f,g\in L^2([0,1])$ and $f\leq C\cdot g$, then by monotonicity,
$K*f\leq C\cdot K*g$. Consequently, both lower an upper estimates for $\frac{(I_K)^{n+1}1_{[0,1]}}{(I_K)^n1_{[0,1]}}$ by $C$ remain valid
after iterations by $I_K$. Thus, the intervals are encapsulated.
If the intervals would get outside of ${\mathrm r}(I_K)$, then the situation would be in contradiction to Theorem \ref{th:Dominate}.
(This argument is valid until we reach $(I_K)^n1_{[0,1]}=0$.)
\end{proof}
\end{theorem}

Despite its simplicity, the theorem above can be of immense value for locating ${\mathrm r}(I_K)$ if $(I_K)^n1_{[0,1]}$ is sufficiently easily computable.

\begin{commentx}
\begin{remark}
\plabel{rem:nonlin}
Despite their somewhat nonlinear history, the all the previous theorems here (in their given form) can be deduced using
only standard qualitative arguments from functional analysis
(like the continuity of the spectrum for compact linear operators)
and Perron's theorem.
\qedremark
\end{remark}
\end{commentx}

Now, already Birkhoff \cite{Bi1} has more quantified statements regarding the setting of his theorem, see also Ostrowski \cite{Ost}.
The most effective approach in that regard is, however, due to E. Hopf \cite{Hpf}, \cite{Hpff}.
He  obtains quite precise bounds for the subdominant eigenvalues and also for the dominant eigenvalue (that is the spectral radius).
If $K>0$ almost everywhere, then we may consider
\[\chi(K)=\frac{\sqrt{\displaystyle{\esssup_{x,x',y,y'\in[0,1]}\frac{K(x,y)K(x',y')}{K(x',y)K(x,y')} }-1}}{\displaystyle{\sqrt{\esssup_{x,x',y,y'\in[0,1]}\frac{K(x,y)K(x',y')}{K(x',y)K(x,y')} }+1}}\]
(where $\frac{\infty-1}{\infty+1}=1$).
If $ m\cdot 1_{[0,1]^2} \leq  K \leq M\cdot1_{[0,1]^2}$ with $0<m\leq M<+\infty$, then
\[\chi(K)\leq \frac{M-m}{M+m}<1\]
holds.

A more quantitative version of Theorem \ref{th:average} is given by
\begin{theorem}[E. Hopf \cite{Hpf}, \cite{Hpff} (1963), special case]
\plabel{th:Hopf}
Regarding the length of the encapsulated intervals in \eqref{eq:average},
\[|\mathcal E_{n+1}|\leq \chi(K)\,|\mathcal E_n|.\]

In particular, if
$ m\cdot 1_{[0,1]^2} \leq  K \leq M\cdot1_{[0,1]^2}$ with $0<m\leq M<+\infty$, then

\begin{equation*}
\left(\mathrm{ess\,sup\,}\frac{(I_K)^{n+1}1_{[0,1]}}{(I_K)^n1_{[0,1]}}\right)
-
\left(\mathrm{ess\,inf\,}\frac{(I_K)^{n+1}1_{[0,1]}}{(I_K)^n1_{[0,1]}}\right)
\leq\left(\frac{M-m}{M+m}\right)^n(M-m).
\end{equation*}
\begin{proof}[Remark on proof]
Hopf  \cite{Hpf}/\cite{Hpff} asks for pointwise definedness for $(I_K)^n1_{[0,1]}$, but
the argument works out in this $L^2$ setting (as long as the underlying measure is finite).
\end{proof}
\end{theorem}

\begin{commentx}
In particular, if $K$ is continuous, then
\[\left|{\mathrm r}(I_K)- \frac{(I_K)^{n+1}1_{[0,1]}(t)}{(I_K)^n1_{[0,1]}(t)}\right|\leq \left(\frac{M-m}{M+m}\right)^n(M-m)\]
holds for any $t\in[0,1]$.
\end{commentx}

Consequently, in the setting of the previous theorem,
\[\left|{\mathrm r}(I_K)-
\frac{\langle1_{[0,1]},(I_K)^{n+1}1_{[0,1]}\rangle}{\langle1_{[0,1]},(I_K)^{n}1_{[0,1]}\rangle  }
\right|\leq \left(\frac{M-m}{M+m}\right)^n(M-m).\]

If $m$ is small, then majorization and minorization by rectangularly based step functions
provide easily computable absolute estimates (with relatively greater tolerance).
In general, it can be useful to pass to powers of $I_K$ in order to get better estimates for the spectral radius.
For the sake of completeness, we state

\begin{theorem}[E. Hopf \cite{Hpf}, \cite{Hpff} (1963), special case]
\plabel{th:Hopf2}
Assume that $K>0$ almost everywhere.
Then, for any $\lambda\in\spec(I_K)$ with $\lambda\neq{\mathrm r}(I_K)$, one has
\[|\lambda|\leq\chi(K)\,{\mathrm r}(I_K).\]

In particular, if $m\cdot 1_{[0,1]^2} \leq  K \leq M\cdot1_{[0,1]^2}$ with $0<m\leq M<+\infty$, then
\[|\lambda|\leq\frac{M-m}{M+m}{\mathrm r}(I_K).\]
\begin{proof}[Remark on proof]
Again, the $L^2$ setting is slightly different from the original setting of Hopf  \cite{Hpf}/\cite{Hpff}.
Nevertheless Hopf's arguments  work out in a straightforward manner in the case $m\cdot 1_{[0,1]^2} \leq  K \leq M\cdot1_{[0,1]^2}$ with $0<m\leq M<+\infty$.
In general, by, say, dyadic averaging, we have an approximating sequence $K_n\rightarrow K$.
As, averaging does not increases $\esssup_{x,x',y,y'\in[0,1]}\frac{K(x,y)K(x',y')}{K(x',y)K(x,y')} $,
the spectrum of $I_{K_n}$ have the desired property.
Then the statement follows from the continuity of the spectrum.
\end{proof}
\end{theorem}
\begin{remark}
Although the arguments for Hopf's theorems require some (minimal) adaptation to the
the $L^2$ case, we remark that  the original setting of Hopf  \cite{Hpf}/\cite{Hpff} applies
directly when the kernel is of two-sided continuous Volterra type like the resolvent estimating kernels
we consider in this paper.
\qedremark
\end{remark}

We say that $K$ is of Toeplitz type, if $K(t_1,t_2)$ depends only on $t_2-t_1$.
In that case we may write $K(t_2-t_1)\equiv K(t_1,t_2)$.
If $K$ is of Toeplitz type and $K(t)=K(t-1)$ holds for $t\in[0,1]$, then
we say that $K$ is of convolution type.
It is easy to show that if $K\geq0$ and $K$ is of convolution type, then
$\mathrm r(I_K)=\int_{t=0}^1 K(t)\,\mathrm dt$.

\begin{commentx}
\begin{point}
Assume that $A$ is a linear operator on $L^2([0,1])$.
Then the principal local characteristic function of $A$
is given by
\[\nu\in\mathbb C\mapsto P_A(\nu)=\left\langle(\Id-\nu A)^{-1} 1_{[0,1]},1_{[0,1]}\right\rangle_{L^2([0,1])},\]
which is well-defined on the complement of the inverse spectrum but possibly extending. $P_A(0)=1$.
If $A=I_K$, then the inverse spectrum is discrete.
If $K\geq 0$, ${\mathrm r}(I_K)>0$, then, by Theorem \ref{th:Dominate}, $P_{I_K}$ still have a non-removable
a singularity at $\nu=1/{\mathrm r}(I_K)$.
Thus, if $K\geq 0$, then the radius of the disk of analyticity of $P_{I_K}$ around $\nu=0$
is still given by $1/{\mathrm r}(I_K)$, the first possible non-removable singularity appearing on the positive real axis.
Thus  $P_{I_K}$ is informative about ${\mathrm r}(I_K)$ on the positive real axis.

\begin{lemma}
\plabel{lem:princpert}
If $\alpha\in\mathbb C$, then
\[P_{A+\beta I_{1_{[0,1]^2}} }(\nu)=\frac{P_A(\nu)}{1-\beta \nu P_A(\nu)}\]
holds in a neighborhood of $\nu=0$ (but possibly extending).
\begin{proof}
This follows from the identity
\[\left\langle A_1 I_{1_{[0,1]^2}} A_2 1_{[0,1]}, 1_{[0,1]}\right\rangle_{L^2([0,1])}=
\left\langle A_1 , 1_{[0,1]}\right\rangle_{L^2([0,1])}
\cdot
\left\langle A_2 1_{[0,1]}, 1_{[0,1]}\right\rangle_{L^2([0,1])},
 \]
and simple combinatorial principles.
\end{proof}
\end{lemma}

\end{point}
\end{commentx}

\end{document}